\documentclass[10pt]{article}

\RequirePackage{amsthm,amsmath,amsfonts,amssymb}
\RequirePackage[numbers,sort&compress]{natbib}
\RequirePackage[colorlinks,citecolor=blue,urlcolor=blue]{hyperref}
\RequirePackage{graphicx}
\usepackage[margin=0.8in]{geometry}

\theoremstyle{plain}

\newtheorem{theorem}{Theorem}[section]
\newtheorem{lemma}[theorem]{Lemma}
\newtheorem{proposition}[theorem]{Proposition}
\theoremstyle{remark}

\newtheorem*{remark}{Remark}

\usepackage[normalem]{ulem}
\usepackage{mathtools}
\usepackage{subcaption}
\usepackage{diagbox}
\usepackage{bm}
\usepackage{bbm}
\usepackage{multirow}
\usepackage{rotating}

\allowdisplaybreaks[1]

\newcommand{\calE}{\mathcal{E}}
\newcommand{\calC}{\mathcal{C}}
\newcommand{\calN}{\mathcal{N}}
\newcommand{\calX}{\mathcal{X}}
\newcommand{\calY}{\mathcal{Y}}
\newcommand{\calD}{\mathcal{D}}

\newcommand{\boldcalA}{\boldsymbol{\mathcal{A}}}
\newcommand{\calM}{\mathcal{M}}
\newcommand{\calB}{\mathcal{B}}
\newcommand{\calL}{\mathcal{L}}
\newcommand{\calS}{\mathcal{S}}

\newcommand{\calU}{\mathcal{U}}
\newcommand{\calF}{\mathcal{F}}
\newcommand{\bfV}{\mathbf{V}}
\newcommand{\bfX}{\mathbf{X}}
\newcommand{\bfW}{\mathbf{W}}
\newcommand{\bfR}{\mathbf{R}}
\newcommand{\bfS}{\mathbf{S}}
\newcommand{\bfA}{\mathbf{A}}
\newcommand{\bfM}{\mathbf{M}}
\newcommand{\bfT}{\mathbf{T}}
\newcommand{\bfP}{\mathbf{P}}
\newcommand{\bfB}{\mathbf{B}}
\newcommand{\bfD}{\mathbf{D}}
\newcommand{\bfC}{\mathbf{C}}

\newcommand{\bfL}{\mathbf{L}}

\newcommand{\bfI}{\mathbf{I}}
\newcommand{\bfF}{\mathbf{F}}
\newcommand{\bfG}{\mathbf{G}}
\newcommand{\rmc}{\mathrm{c}}
\newcommand{\rmd}{\mathrm{d}}

\newcommand{\Vol}{\mathrm{Vol}}
\newcommand{\fhat}{\widehat{f}}

\newcommand{\diag}{\text{diag}}
\newcommand{\Cov}{\mathrm{Cov}}
\newcommand{\Tr}{\mathrm{Tr}}
\newcommand{\Op}{\mathrm{Op}}
\newcommand{\Spec}{\mathrm{Spec}}
\newcommand{\AR}{\mathrm{AR}}
\newcommand{\nsplit}{\mathrm{split}}

\newcommand{\calib}{\mathrm{calib}}
\newcommand{\histo}{\mathrm{histo}}
\newcommand{\test}{\mathrm{test}}
\newcommand{\opt}{\mathrm{opt}}

\newcommand{\dt}{\mathrm{d}t}
\newcommand{\ith}{i\textsuperscript{th}}
\newcommand{\nath}{n^{\text{th}}_{\alpha}}

\newcommand{\ceil}[1]{\left\lceil#1\right\rceil}
\newcommand{\stkout}[1]{\ifmmode\text{\sout{\ensuremath{#1}}}\else\sout{#1}\fi}

\newcommand{\Pbb}[1]{\mathbb{P}\left(#1\right)}

\newcommand{\Pb}{\mathbb{P}}
\newcommand{\Rbb}{\mathbb{R}}

\newcommand{\adri}{\textcolor{black}}

\newcommand{\iain}{\textcolor{black}}

\usepackage{authblk}

\begin{document}

\title{Adaptive inference with random ellipsoids through Conformal Conditional Linear Expectation}





\date{}

\author[1]{Iain Henderson \thanks{iain.henderson@isae-supaero.fr}}
\author[2]{Adrien Mazoyer}
\author[2]{Fabrice Gamboa}
\affil[1]{Fédération ENAC ISAE-SUPAERO ONERA, Université de Toulouse, Toulouse, 31055, France}
\affil[2]{Institut de Mathématiques de Toulouse, Université de Toulouse, CNRS, UPS, 31062,
Toulouse, France}
\maketitle

\begin{abstract}
    We propose two new conformity scores for conformal prediction, in a general multivariate regression framework. The underlying score functions are based on a covariance analysis of the residuals and the input points. 
    We give theoretical guarantees on the prediction sets, which consist in explicit ellipsoids. We study the asymptotic properties of the ellipsoids, and show that their volume is reduced compared to that of classic balls, under ellipticity assumptions. Finally, we illustrate the effectiveness of all our results on an in-depth numerical study, including heavy-tailed as well as non-elliptical distributions.
\end{abstract}



\section{Introduction}
Uncertainty quantification aims to provide mathematical techniques to quantify uncertainties in computational and real worlds. One very common problem in uncertainty quantification is to provide a confidence set for a given prediction method.  When the prediction method is based on a statistical model that is not too complex, classical  statistical tools can lead to a  confidence set. In black box models, as those developed in machine learning or in complicated regression models the inference process to build such confidence sets is not so straightforward. One now classical way to build such confidence sets stands on a non-parametric learning of the fluctuations of the predictors. This technique called {\it conformal inference} allows the construction of a confidence set for a given trained predictor by only observing its behaviour on a calibration sample.

Conformal inference has a long story beginning early in the forty's by  the pioneer works in reliability of Wilks (\cite{wilks1941determination,wilks1942statistical}). Conformal prediction has then been subsequently developed in the early 2000s by a research group around Vapnik and Vovk (see e.g. \cite{vovk2005algorithmic}). With the spectacular progress and the massive implementation of black-box models in neural machine learning, there has been an impressive revival of conformal prediction for at least five or seven years. In a nutshell, conformal prediction is   
a general concept that includes many different means of providing confidence of pointwise predictions produced using a statistical model or machine learning, without any knowledge on the predictor and under mild assumptions. Indeed, the construction of a prediction set having a high probability of containing a quantity of interest is a general challenging problem in the uncertainty quantification paradigm.
More formally, let $\calX$ and $\calY$ be some state spaces. Let further $(X,Y)\in\calX\times\calY$ be a random vector. Given a sample $(X_i,Y_i)_{i=1,\ldots,n}$ and a predictor $\fhat:\calX\to\calY$ (independent of the sample),  for  $X_{n+1}\in\calX$ (a new draw of $X$), we aim to build a random measurable region $C_\alpha^n(X_{n+1})\subseteq\calY$ {\it as tight as possible},  such that for a given $\alpha$ (close to $0$ and set by the user), we have
$
\Pbb{Y_{n+1}\in C_\alpha^n(X_{n+1})} \geq 1-\alpha.$ 
In some ideal cases,
the known statistical properties or characteristics of the predictor $\fhat$ 
can be sole used to construct $C_\alpha^n(\cdot)$. However, this usually requires model assumptions on the data, that
are sometimes too restrictive or unrealistic. Even worse, it may happen that none of the properties of $\fhat$ are
useful for building $C_\alpha^n(\cdot)$.
The aim of Conformal Inference (CI) is to construct such regions, with a very appealing advantage: CI requires only mild assumptions on the sample distribution and on the predictor. Actually, the only hypothesis needed is an exchangeability one, that will be discussed later (see the beginning of Section \ref{sec:bg}). Roughly speaking, the intuition behind CI is 
that a relevant prediction for the unknown output $Y_{n+1}$ should be any $y\in\calY$ such that $(X_{n+1} , y,\fhat(X_{n+1}))$ seems to be \textit{conform} to the sample $(X_i,Y_i, \fhat(X_i))_{i=1,\ldots,n}$. The main ingredient here is to define and to work with an appropriate score of conformity. Of course, the appropriateness of a given score greatly depends on the context. In particular, the shape and accuracy of the prediction set is directly related both to the performance of the predictor $\fhat$ and to the conformity score.

Since its introduction in the early 2000s (see e.g. \cite{Vovketal05}), conformal prediction has been developed in many different contexts as for example classification \cite{DevetyarovNouretdinov10,Balasubramanianetal14}, regression \cite{LeiWasserman14,BurnaevVovk14,Romanoetal19}, functional data \cite{Diquigiovannietal22}, outlier detection \cite{Batesetal23,Liangetal24}, neural networks \cite{papadopoulos2008inductive}, etc. For more insights  and examples, we refer to the recent reviews  \cite{lei2018, AngelopoulosBates23, fontana2023, Katoetal23}.
In many works dealing with CI, the size and shape of $C_\alpha^n(X_{n+1})$ only mildly depend on $X_{n+1}$. This can make sense if for example the covariance of the residual $Y-\widehat{f}(X)$ is homogeneous. \adri{Dealing with heterogeneous data can be approached with different objectives in mind. The grail of conformal inference would be to obtain a conditional coverage property for $C_\alpha^n(X_{n+1})$, that is $\Pbb{Y_{n+1}\in C_\alpha^n(X_{n+1})|X_{n+1}} \geq 1- \alpha$. However, it is theoretically impossible to ensure the conditional coverage with a finite sample \cite{Vovk12}. A less demanding property could be then to be able to capture some geometric features of the conditional distribution $\mathbb{P}(Y|X)$, typically heteroscedasticity in data. The resulting prediction set should then adapt in terms of shape and size with respect to the input.
For example, \cite{lei2018} proposed to use some normalizing statistic in the score, based on the previous sample points. In this last work, the trick consists in normalizing the residuals, for example with respect to their median. Notice furthermore that many of the works in CI consider only the case where $\calY\subseteq\Rbb$, leading thus to prediction intervals. The direct extension to $\calY\subseteq\Rbb^{\ell}$ then consists in using a norm of the residual vector $R \coloneqq Y-\widehat{f}(X)$, leading to prediction balls centered on $\widehat{f}(X)$. However, our wish is that the geometry of the prediction set adapts with the input, which cannot be achieved by considering balls. The multidimensional framework is still widely studied, with some ongoing works considering responses taking values in separable metric spaces \cite{LugosiMatabuena2024}, exploring approaches based on transport tools \cite{ZhouMuller2024, Thurinetal2025}, or formulation in terms of minimum-volume coverage for arbitrary norm-based sets \cite{Braunetal2025}.}

 In this work, we propose and fully investigate CI with a score built on the empirical covariance
 of $(X,Y-\widehat{f}(X))$, where \textit{both} $X$ and $Y$ are multidimensional. The covariance estimate and score include the prediction point $X_{n+1}$ and the potential residual $y-\fhat(X_{n+1})$.
We show that our procedure amounts to estimating the conditional linear expectation of the residual $Y-\fhat(X)$ given the input $X$, formulated in a CI setting (see Remark in Section \ref{ssec:main_theo}). We will thus call our method the conformal conditional linear expectation (CCLE). The main advantage of the associated region $C_\alpha^n$ is that it is well tailored to multivariate outputs $Y\in\mathbb{R}^{\ell}$, while remaining easy to compute. Moreover, it leads to adaptive prediction sets in terms of size and shape. \adri{It is important to recall that despite the conditional aspect of the score proposed here, its main motivation is not to improve the conditional coverage. \iain{Rather, our primary goals are $(i)$ the adaptivity of the prediction set with reference to the distribution of $(X,R)$ and $(ii)$ decreasing the volume of the prediction set.} In Section~\ref{sec:asymp}, we discuss the necessary adjustments to recover the conditional coverage in the asymptotic regime.}

The following theorem summarises and gives in a nutshell our main results (Theorem \ref{thm:eq_ellipsoid}, Theorem \ref{thm:ellip_prime}, Proposition \ref{prop:barthe}, Proposition \ref{prop:gen_k} and Proposition \ref{prop:vol_fna}). 


\begin{theorem}\label{thm:intro}
Let $((X_1, Y_1), \dots , (X_{n+1},Y_{n+1}))$ be exchangeable pairs of random vectors and $\alpha\in(1/(n+1),1)$.
Then, 
\begin{itemize}
    \item There exists explicit $\boldcalA_n, \rho_{n,\alpha}$ and $\widetilde{Y}_{n+1}$, only depending on $((X_1, Y_1), \dots , (X_{n},Y_{n})), \ X_{n+1}$ and $\alpha$, such that the $\ell$-dimensional ellipsoid
\begin{align}
       E_{\alpha}^n = \{y\in\mathbb{R}^{\ell}: (y-\widetilde{Y}_{n+1})^\top \boldcalA_n^{-1}(y-\widetilde{Y}_{n+1}) \leq \rho_{n,\alpha}\}
\end{align}
satisfies $\mathbb{P}(Y_{n+1} \in E_{\alpha}^n) \geq 1 - \alpha.$ 
\item Set $R_1 = Y_1-\widehat{f}(X_1)$ and assume that $(X_1, R_1)$ follows an elliptical distribution. 
Then, as $n\rightarrow +\infty$ and under mild assumptions on the distribution of $(X_1, R_1)$, 
the volume $\Vol(E_{\alpha}^{n})$ converges in distribution towards a random variable 
$\Vol(E_{\alpha}^{\infty})$ 
satisfying
\begin{align} 
    \Vol(E_{\alpha}^{\infty}) \leq  \Vol(B^{\infty}_{\alpha}) \  a.s..
\end{align}
Here, $\Vol(B^{\infty}_{\alpha})$ is the (deterministic) limit of the volume of the ball obtained using the standard score $\|Y-\widehat{f}(X)\|$ (see Proposition~\ref{prop:asymptotic_ball} for its expression).
\end{itemize}
\end{theorem}

In other words, we show that the prediction set obtained using CCLE is an ellipsoid $E_\alpha^n\subset\mathbb{R}^{\ell}$. Furthermore, the volume of this prediction set is asymptotically smaller than that of the classical conformal prediction ball, under mild ellipticity assumptions (Section \ref{sub:comparison_norm_score}). Elliptical distributions are those whose level sets are elliptical, such as the Gaussian ones. \iain{
As we will see later, we are actually able to identify two different $\rho_{n,\alpha}$ satisfying Theorem \ref{thm:intro}, each with different properties concerning the associated ellipsoid. The first one, corresponding to the \iain{joint} Mahalanobis score (Section \ref{ssec:mahalanobis_score}), leads to ellipsoid $E_{\alpha}^n$ such that $\mathbb{P}(E_{\alpha}^n =\varnothing) > 0$ and $\mathbb{P}(E_{\alpha}^n =\mathbb{R}^\ell) = 0$; whereas the second, corresponding to adjusted Mahalanobis score (Section \ref{ssec:adjusted_score}), leads to ellipsoid $F_{\alpha}^n$ such that $\mathbb{P}(F_{\alpha}^n =\varnothing) = 0$ and $\mathbb{P}(F_{\alpha}^n =\mathbb{R}^\ell) > 0$ (this probability tending to 0 as $n$ tends to $\infty$). Notably, the second ellipsoid reduces to the conditional prediction set with a minimal volume, in the case of Gaussian data, when $n$ tends to $\infty$ (Propositions \ref{prop:conditional_gaussian} and \ref{prop:barthe}).}

Let us note that to calculate the conditional linear expectation, we will use a ridge estimator for the covariance. This is motivated by three facts. First, it stabilizes the numerical procedure \cite{linear_cond_exp}, Section 4.4. Secondly, it simplifies several proofs (see e.g. Proposition \ref{prop:bound}). Last but not least, our method will be carried out in an infinite dimensions setting (see \cite{GHMR_forthcoming}) where finite size ellipsoids are rather well-understood (see e.g. Propositions 6.1 to 6.3 in \cite{dudley_compact}).

Before presenting the outline of our paper, let us further discuss the use of ellipsoidal confidence regions in CI. In \cite{pmlr-v152-johnstone21a} and \cite{messoudi22a} the authors construct local ellipsoidal confidence regions in a standard multivariate CI setting. While our method effectively relies on the conditional \emph{linear} covariance (see Section \ref{ssec:main_theo}), the local ellipsoids described in \cite{messoudi22a} correspond to a convex combination of empirical estimates of the global covariance of the residual $R$ and the \emph{true} conditional covariance $\Cov(R|X=x)$ (even though those local ellipsoids are not presented as estimates of $\Cov(R|X=x)$ in \cite{messoudi22a}). The latter are estimated by only selecting the observed residuals $r_i$ (realisation of $R_i$) such that the corresponding $x_i$ (realisation of $X_i$) are the $k$ closest to the observed realisation $x_{n+1}$ of $X_{n+1}$. A drawback of this nearest neighbour algorithm is that a sufficient amount of observed inputs $x_i$ close to $x_{n+1}$ is required to obtain meaningful estimates of $\Cov(R_{n+1}|X_{n+1}=x_{n+1})$.
  The authors of \cite{messoudi22a} then show the practical efficiency of their CI method, but do not further study its theoretical properties. CCLE can also be seen as an intermediate solution between two paradigms: relying on the global (unconditional) covariance matrix $\Cov(R)$, or on the local (fully conditional) covariance matrix $\Cov(R|X=x)$. CCLE is then an alternative solution to the convex combination proposed in \cite{messoudi22a}. The recent paper \cite{xu2024conformal} also proposes a CI on the residuals leading also to a confidence ellipsoid. The statistical frame therein is a single realisation of a time series. Here a stationarity assumption is replacing the exchangeability one. The method stands on the estimated empirical covariance of the residuals, estimated from the past. Notice that the method does not encompass any conditional information $R|X$. The authors show the practical advantage of their method when compared to copula-based CI \cite{MESSOUDI-copula}, and establish conditional coverage properties, in the limit of the large observational horizon (this would correspond to the large sample size in our framework). In the case of both \cite{messoudi22a} and \cite{xu2024conformal}, the authors do not study the theoretical properties of the size (volume) of the corresponding confidence regions. We believe that this feature is important in practical applications, along with the conditional coverage. In our work, we study the volumetric properties of $E_\alpha^n$ in Section \ref{sub:comparison_norm_score}. Finally, in both  \cite{messoudi22a} and \cite{xu2024conformal}, the ellipsoids are centered at the predictor $\widehat{f}(X_{n+1})$, while our method naturally corrects $\widehat{f}(X_{n+1})$ with the empirical conditional linear expectation of $R$ given $X$ (Section \ref{ssec:main_theo} and Proposition \ref{prop:pred_unbiased}). This feature is interesting as $\widehat{f}$ can be biased, which may happen e.g. when $\widehat{f}$ is a black-box neural network.

The paper is organized as follows. We recall  some preliminaries on split conformal prediction for regression in Section \ref{sec:bg}.  The heart of our work is the Section \ref{sec:score} where we discuss the construction of the conformal score and give some initial results on the induced conformal prediction. This is followed by the Section \ref{sec:asymp} where we study some asymptotic properties of the conformal ellipsoid, \iain{notably its volume under ellipticity assumptions}.  We illustrate and demonstrate the practical effectiveness of our results on simulated data in Section \ref{sec:numexp}. All the technical proofs and technical lemmas are postponed to Appendices \ref{app:lemmas} and \ref{app:proofs}.

\section{Background for Split Conformal Inference} \label{sec:bg}
Let $(U_i)_{i=1,\ldots,n+1}$
be $(n+1)$ exchangeable vectors of random variables, with $U_i=(X_i,Y_i)\in\calU=\calX\times\calY$. 
This means that  the joint distribution of  $(U_i)_{i=1,\ldots, n+1}$ is invariant by  any permutation of the symmetric group of order $(n+1)$. For example, an i.i.d. sample satisfies this assumption. In our framework we assume that $(U_i)_{i=1,\ldots, n}$, called the calibration set (or calibration sample), has been observed and that we get
$X_{n+1}$. We assume further that we have at hand a predictor $\widehat{f}(X_{n+1})$ of $Y_{n+1}$. In this paper, for the sake of simplicity we assume that this predictor is deterministic. For example, it had been built using a sample independent of $(U_i)_{i=1,\ldots, n+1}$ and so we are working conditionally to the training sample (this is the main assumption of \textit{split} CI, in contrast with \textit{full} CI where $\widehat{f}$ is retrained for each new sample point).   
The aim is to provide a prediction set for $Y_{n+1}$.
More precisely, as discussed in the Introduction,  our objective is  to construct a {\it conservatively} valid prediction set for $Y_{n+1}$ with a given level of confidence. Hence, we aim to build a set function $C_\alpha^n$ defined on $(\calX\times\calY)^n\times \calX$ that will return a measurable subset of $\calY$ for $Y_{n+1}$ such that
	\begin{equation}\label{eq:cov_classicCI}
		\Pbb{Y_{n+1}\in C_\alpha^n\left((X_i,Y_i)_{i=1\dots n};X_{n+1}\right)}\geq 1-\alpha\,.
	\end{equation}
Notice that in the above probability we have integrated over $(X_{i},Y_{i})_{i=1\dots n+1}$.
The construction of $C_\alpha^n$ stands on the notion of conformity. To begin with, we define a nonconformity score. It consists in a  map $S$ from $\calU^n\times \calU$ to $\Rbb$, symmetric in its first $n$ arguments. The empirical distribution built from $(S(U_1,\ldots,U_n;U_j))_{j=1,\ldots,n}$ will be the main tool to quantify how conformal to $U_1,\ldots, U_n$, $u\in\calU$ is. 
Archetypal scores in the usual regression framework with $\calU=\Rbb^k\times\Rbb^\ell$  are related to the magnitude of the empirical residual, the most popular example being
\begin{equation}\label{eq:basic_res_score}
    S((X_1,Y_1),\dots,(X_n,Y_n),(x,y)) = \|y-\fhat(x)\|,
\end{equation}
where $\|\cdot\|$ is the Euclidean norm of $\mathbb{R}^\ell$.
Notice that in this last example the score function only depends on $u=(x,y)$ but not on the calibration sample $(U_1, \ldots, U_n)$. This will not be the case in the framework that we are developing in this paper (see the next section).
%
%
%

\noindent
Let $\alpha\in(1/(n+1),1)$. For the sake of conciseness, we will use the notation $S(x,y) \coloneqq S((X_1,Y_1),\dots,(X_n,Y_n);(x,y))$. For $i=1,\ldots, n+1$, we will also write  $S_i \coloneqq S(X_i,Y_i)$.
Following for example \cite{AngelopoulosBates23}, the conformalized prediction set for $Y_{n+1}$ when observing $X_{n+1}$ is then
	\begin{equation} \label{eq:sci_pred_set}
		C_\alpha^n(X_{n+1})\equiv C_\alpha^n((X_1, Y_1)\dots (X_n, Y_n);X_{n+1}) = \left\{y\in\calY\;{:}\;S(X_{n+1},y) \leq S_{(n_\alpha)}\right\}.
	\end{equation}
Here, $n_\alpha  \coloneqq \ceil{(1-\alpha)(n+1)}$ and $S_{(1)}\leq\dots\leq S_{(n)}$ denotes the order statistics associated to $(S_i)_{i=1\dots n}$. By construction, the above set $C_\alpha^n(X_{n+1})$ satisfies
	\begin{align}\label{eq:coverage}
	\Pbb{Y_{n+1} \in C_\alpha^n(X_{n+1})} \geq 1-\alpha.
	\end{align}
	 The key property ensuring that equation \eqref{eq:coverage} holds is that $(S_i)_{i=1,\ldots, n+1}$ is exchangeable. Moreover, if the nonconformity scores $(S_i)_{i=1\dots n}$ have a continuous joint distribution, then we have the following upper bound,
\begin{align}\label{eq:reverse_conformal}
 	\Pbb{Y_{n+1} \in C_\alpha^n(X_{n+1})} \leq 1-\alpha+\frac{1}{n+1}\,.
\end{align}

%
Returning to the regression framework and considering the residual score \eqref{eq:basic_res_score}, the prediction set for $Y_{n+1}$ given the covariable $X_{n+1}$ is
\begin{align*}
C_\alpha^n(X_{n+1}) = 
B\left(\widehat{f}(X_{n+1}), S_{(n_\alpha)}\right)\, ,
\end{align*}
i.e. the ball  with center $\widehat{f}(X_{n+1})$ and radius $S_{(n_\alpha)}$.

\section{The Mahalanobis score and main results} \label{sec:score}
\subsection{Notations} 
We will assume that $\calX=\Rbb^k$ and $\calY=\Rbb^{\ell}$, ($\ell,k \in \mathbb{N}$).
For any positive integer $n_1$ and $n_2$,  $\calM_{n_1,n_2}$ denotes the set of all $n_1\times n_2$ real matrices. For, $i=1\ldots, n$, $e_i$ denotes the $\ith$ element of the canonical basis of $\Rbb^n$. 
Set $p=k+\ell$, for $v\in\Rbb^p\setminus\{0\}$, $\pi^{\perp}_v$ denotes the orthogonal projector onto $\text{Span}(v)^{\perp}$. In others words,  $\pi^{\perp}_v = \mathbf{I}_p - \|v\|^{-2}vv^\top $. $\mathbbm{1}$ denotes the vector compounded by ones on all its components (its dimension will be implicit from the context). For any $\alpha\in(0,1)$, we set $n_{\alpha} \coloneqq \lceil (1-\alpha)(n+1)\rceil$. All non-column matrices are denoted with bold letters. For any block matrix $\bfM = \bigl( \begin{smallmatrix}\bfA & \bfB\\ \bfC & \bfD\end{smallmatrix}\bigr)$, we will denote the Schur complement (whenever it exists ) by $\bfM/\bfA \coloneqq \bfD - \bfC\bfA^{-1}\bfB$. The abbreviation SLLN stands for the strong law of large numbers. Given $A\in\calM_{n_1,n_2}$ $(n_1,n_2\geq 1)$,  its Frobenius norm is denoted as $\|A\|_F^2\coloneqq \Tr(AA^\top)$. Given $A, B\in \calM_{n,n}$ two symmetric matrices, $A\preccurlyeq B$ means that $B-A$ is nonnegative definite. Let $\xi$ be real random variable and $\alpha\in(0,1)$, $q_{1-\alpha}(\xi)$ denotes its quantile of order $1-\alpha$, i.e. $q_{1-\alpha}(\xi)=\inf\{r\in\Rbb\,{:}\,\Pb(\xi>r) \leq  \alpha\}$. Recall also that we have a predictor $\widehat{f}$ independent of the calibration sample and that we denote the residual by $R\coloneqq Y-\widehat{Y},\, \widehat{Y} = \widehat{f}(X)$.
For a square integrable random vector $\xi$, $\mu_\xi$ and $\mathbf{\Sigma}_\xi$ denote its mean vector and covariance matrix, respectively. We abbreviate ``almost surely'' as ``a.s.''. $\xlongrightarrow[]{\mathbb{P}}$ stands for convergence in probability, $\stackrel{\rmd}{=}$ and $\xrightarrow[]{\rmd}$ stand for equality and convergence in distribution, respectively.
To finish, we represent all finite-dimensional vectors as column matrices.

\subsection{\iain{Joint} Mahalanobis score} \label{ssec:mahalanobis_score}
As discussed in the previous section, a useful score used in CI is the (squared) norm of the residual.
In this paper, we will consider  a natural generalization  of the previous score by considering some statistical evaluation of  $(R - \mu_R)^\top  \mathbf{\Sigma}_R^{-1} (R - \mu_R)$ (assuming that everything here is well defined). 
%
Before explaining how we will estimate and handle the previous quantity, let it first bulk up
by including $X$ in the score. For this purpose, let $V \coloneqq (X^\top\ R^\top)^\top \in\mathbb{R}^{p}$ and set
\begin{align}\label{eq:ellipsoid_fixed_theta}
    \calS = (V - \mu_{V})^\top  \mathbf{\Sigma}^{-1}_{V} (V - \mu_{V}).
\end{align}
As before, we assume here that everything is well defined. The intuition behind this score is to work with standardized quantities. One may recognize in \eqref{eq:ellipsoid_fixed_theta} the so-called Mahalanobis quadratic form, which in turn defines the Mahalanobis metric. This metric is classical in multivariate analysis and widely used for example for classification tasks (see \cite{mahalanobis2018generalized}).
\iain{The idea of including the covariate $X$ in the score is inspired by the CI framework for time series, where $X$ can contain the $k$ previous residuals (see e.g. \cite{Linetal22a}, Section 4.2, or the end of this subsection) : the joint distribution $(X^\top R^\top)^\top$ is relevant for predicting the future $\ell$ residuals (i.e. $R$).}

In the context of CI, the position and covariance parameters of $V$ are of course unknown or even undefined, and it is natural to replace them with their empirical estimators. 
Nevertheless,  in order to preserve exchangeability and so the finite sample theoretical guarantees provided by the classical CI procedure, one has to process such empirical estimators in an exchangeable fashion. 
Of course, another way to proceed would be to estimate these parameters separately, similarly to the pre-trained predictor, leading to a {\it double split} procedure. The first option is more challenging and this is our approach here. To begin with,
we now introduce precisely our working score. We first set
\begin{equation}
V_i=
\begin{pmatrix}
X_i\\
 R_i
\end{pmatrix} 
\; (i=1, \ldots,n)
\mbox{ and, for } z\in\Rbb^{\ell}, \;
V_{n+1}(z) =
\begin{pmatrix}
X_{n+1}\\
 z
\end{pmatrix}
.
\label{eq:decomp_Vi_Xi_Ri}
\end{equation}
Here,  the variable $z= y - \widehat{f}(X_{n+1})$ (resp. $y$) is  used to locate the {\it most likely} potential residuals (resp. predictions) in the CI machinery. 
We store the observation vectors $V_i,\; i=1\ldots,n$ and the guess $V_{n+1}(z)$ in a matrix $\bfV(z)\in \calM_{n+1,p}$ whose first $n$ row are the $V_i^\top, i=1,\ldots, n$ and the last one is  $V_{n+1}^\top(z)$.
Following equation \eqref{eq:decomp_Vi_Xi_Ri}, we decompose $\bfV(z)$ as $\bfV(z) = (\bfX \ \bfR(z))$ with $\bfX\in\calM_{n+1,k}$ and $\bfR(z)\in \calM_{n+1,\ell}$. The empirically centered counterpart of $\bfV(z)$ is  $\bfW(z) = \pi^{\perp}_{\mathbbm{1}}\bfV(z)$.
We further  compute the ridge empirical covariance matrix $\widehat{\mathbf{\Sigma}}_{\lambda}(z)\in\calM_{p,p}$,
\begin{align}\label{eq:def_sig_hat}
    \widehat{\mathbf{\Sigma}}_{\lambda}(z) = \frac{1}{n} \bfW(z)^\top \bfW(z) + \lambda \mathbf{I}.
\end{align}
Here,  $\lambda>0$ is the ridge parameter.
With all these supplementary notations, we are now able to define the score $S_i(z)\; (1\leq i \leq n+1)$ as the square of the empirical Mahanalobis norm of $W_i(z)$ (the $i^\mathrm{th}$ row of $\bfW(z)$ written in column).  Namely, we have
\begin{align}
    S_i(z) = \|\widehat{\mathbf{\Sigma}}_\lambda^{-1/2}(z)W_i(z)\|^2 = W_i(z)^\top \widehat{\mathbf{\Sigma}}_\lambda^{-1}(z)W_i(z).
\end{align}
It is readily seen that $S_i(z)$ is the $\ith$ diagonal element of the following matrix,
\begin{align}\label{eq:def_score_S}
    \bfS(z) &= \bfW(z)\bigg(\frac{1}{n} \bfW(z)^\top  \bfW(z) + \lambda \mathbf{I}\bigg)^{-1}\bfW(z)^\top \nonumber \\ &= n\bfW(z)\Big(\bfW(z)^\top  \bfW(z) + n\lambda \mathbf{I}\Big)^{-1}\bfW(z)^\top .
\end{align}
Remarkably, and contrarily to the norm residual score, the first $n$ scores $(S_1(z),\ldots,S_n(z))$ \textit{also depend} on $z$. This property induces a number of technical difficulties which are dealt with in Section \ref{sub:to_the_score}. To alleviate the notations we will sometimes omit the dependence in $z$ in the matrix valued functions. 

\paragraph*{\adri{First intuition on the confidence regions associated to $S(z)$}} \adri{Before continuing with the main results, let us consider the Gaussian case $(X^\top R^\top)^\top\sim\calN(0,\Sigma)$. Anticipating Section \ref{sec:asymp}, we replace the ridge empirical covariance with $\Sigma$ and set $\lambda = 0$. 
Then, setting $T(z) = z - \Sigma_{21}\Sigma_{11}^{-1}X$ \iain{with $z = y -\widehat{f}(X)$} and using the block inversion lemma (equation \eqref{eq:apply_block_lemma}), our score amounts to
\begin{align}
    S(z) &= \begin{pmatrix}
         X \\z
    \end{pmatrix}^\top\mathbf{\Sigma}^{-1}\begin{pmatrix}
         X \\z
    \end{pmatrix} = X^\top(\mathbf{\Sigma}^{11})^{-1}X + T(z)^\top(\mathbf{\Sigma}/\mathbf{\Sigma}^{11})^{-1}T(z).
\end{align}
\iain{Assume also that the quantile of order $1-\alpha$ of $S(R)$ (denoted by $q_{1-\alpha}^\calE$) is known, so that we can replace the empirical quantile usually computed in CI by its theoretical value $q_{1-\alpha}^\calE$.}
The confidence region becomes
\begin{align}
    \calE_\alpha^\infty &= \{z : X^\top(\mathbf{\Sigma}^{11})^{-1}X + T(z)^\top(\mathbf{\Sigma}/\mathbf{\Sigma}^{11})^{-1}T(z) \leq q_{1-\alpha}^\calE\}.
\end{align}
\iain{Under our Gaussianity assumption, $q_{1-\alpha}^\calE = q_{1-\alpha}(\chi^2(p))$. The ellipsoid $\calE_\alpha^\infty$ has two important features : it is centered on the conditional expectation $\mathbb{E}[R|X] = \Sigma_{21}\Sigma_{11}^{-1}X$ and its matrix is equal to the conditional covariance $\Cov(R|X) = \mathbf{\Sigma}/\mathbf{\Sigma}^{11}$. However, }
 we can see in this simple case that $\calE_\alpha^\infty = \varnothing \iff X^\top(\mathbf{\Sigma}^{11})^{-1}X \geq q_{1-\alpha}^\calE$. This observation suggests the introduction of $S'(z)$, an adjusted version of the score $S(z)$, defined as 
\begin{equation}\label{eq:def_sprime_infinite}
    S'(z) = S(z) -  X^\top(\mathbf{\Sigma}^{11})^{-1}X = T(z)^\top(\mathbf{\Sigma}/\mathbf{\Sigma}^{11})^{-1}T(z).
\end{equation}
The score $S'$ is studied in Section \ref{ssec:adjusted_score}.
This new score leads to a second confidence region
\begin{equation}
    \calF_\alpha^\infty = \{z :  T(z)^\top(\mathbf{\Sigma}/\mathbf{\Sigma}^{11})^{-1}T(z) \leq q_{1-\alpha}^\calF\},
\end{equation}
where $q_{1-\alpha}^\calF = q_{1-\alpha}(\chi^2(\ell))$ \iain{in the Gaussian case}. \iain{Contrarily to $\calE_\alpha^\infty$, this region is never empty as $\Sigma_{21}\Sigma_{11}^{-1}X\in\calF_\alpha^\infty$. In return, its finite sample counterpart $\calF_\alpha^n$ can equal to $\mathbb{R}^\ell$ (Section \ref{ssec:adjusted_score}).
Observe finally that when working with the score matrix \eqref{eq:def_score_S},
it is not obvious anymore that the prediction region built with the score matrix \eqref{eq:def_score_S} is an ellipsoid, although we expect to obtain a prediction set similar to $\calE_\alpha^\infty$.}
}

\paragraph*{Computation of the score} 
Let us comment the effective computation of equation \eqref{eq:def_score_S} before stating the upcoming Theorem \ref{thm:eq_ellipsoid}.
The evaluation of $\bfS(z)$ in \eqref{eq:def_score_S} is not directly tractable: for each candidate $z$,  both a matrix inversion and  matrix products have to be performed.  Fortunately, it turns out that explicit inversion tricks can be used to overcome this first issue. First, notice that the matrix $\bfV(z)$ is  a linear function of $z$. More explicitly, we may write $\bfV(z) =\bfV(0) + \bfA(z)$, where $\bfA(z)$ is the rank one matrix given by 
\begin{align}
\bfA(z) = e_{n+1}\bigg(\sum_{s=1}^{\ell}z_se_{k+s}\bigg)^\top  = e_{n+1}(\mathbf{L}z)^\top ,    
\end{align}
where $\mathbf{L} \in \calM_{p,\ell}$ is the matrix such that $\mathbf{L}^\top  = (\mathbf{0}_{\ell,k}\ \vline \ \mathbf{I}_{\ell})$. Thus,
\begin{align}\label{eq:W(z)_rank_one}
\bfW(z) = \pi^{\perp}_{\mathbbm{1}}\bfV(0) + \pi^{\perp}_{\mathbbm{1}}e_{n+1}(\mathbf{L}z)^\top  = \bfW(0) + v(\mathbf{L}z)^\top.
\end{align}
Here, 
\begin{align}\label{eq:def_v}
v = \pi^{\perp}_{\mathbbm{1}}e_{n+1} =e_{n+1} -\frac{1}{n+1}\mathbbm{1} = \frac{1}{n+1}(-1 -1 \hdots -1\ n)^\top, \ \ \ \text{with} \ \ \ \|v\|^2 = \frac{n}{n+1}.
\end{align}
\noindent
Equation \eqref{eq:W(z)_rank_one} shows that $\bfW(z)$ is a rank one perturbation of $\bfW(0)$. This crucial property is exploited in Lemma \ref{lemma:scores_transpose}, to prove our main Theorem \ref{thm:eq_ellipsoid}.

\begin{remark}[Time series framework]
 In the context of  time series \cite{Stankeviciuteetal21, Linetal22a}, $U_i$ takes the form $U_i = (U_i^t)_{t\in\mathbb{N}}$. Given a multi-horizon predictor $(\widehat{U}_{t+1}^i,\ldots,\widehat{U}_{t+\ell}^i) = \widehat{f}(U_{t-k+1}^i,\ldots, U_t^i)$ one then computes the time series of the residuals, $R_t^i = U_t^i-\widehat{U}_t^i$.
 At a given time $t$, we aim at localizing with high probability the next $\ell$ residuals, $Y = (R_{t+1}, \dots , R_{t+\ell})$  using the $k$ previous ones, $X = (R_{t-k+1}, \dots , R_{t})$. It may also happen that the time series and so the residuals are both  multidimensional ($R_s \in \Rbb^d$). This is e.g. the case for time series describing the evolution of a position in $\Rbb^3$ ($d=3$). In this case, $X$ is a vector in $\Rbb^{kd}$ and $Y$ as a vector in $\Rbb^{\ell d}$. The vector corresponding to the $\ith$ time series takes the form
 \begin{align}
     V_i = (X_i^\top \ Y_i^\top)^\top \in\Rbb^{dk+d\ell}.
 \end{align}
 This time series setting falls in the framework of this paper and will be developed extensively in 
 \cite{GHM_forthcoming1}.
\end{remark}
The most direct interpretation of the score vector $\diag(\bfS(z)) = (S_1(z),\ldots,S_{n+1}(z))$ is the empirical Mahalanobis distance. This metric is e.g. also the central tool in linear discriminant analysis (LDA)
    \cite{muirhead1982}, Chapter 9-10. This is closely related to the Karhunen-Loève decomposition, which consists in finding an orthonormal basis $(\xi_i\otimes\phi_i)_{i\in\mathbb{N}}\subset L^2(\mathbb{P})\otimes L^2(\Rbb^d)$ on which a random field is naturally decomposed.
\adri{It is worth to notice that other interpretations in terms of leverage score and shape theory can also be made. The interested reader will find some insights in Appendix \ref{app:score_interpretation}.}

\subsection{The first ellipsoid \texorpdfstring{$\calE_\alpha^n$}{E alpha n} : main theorem} \label{ssec:main_theo}
We first need to introduce some notations.
While the score $\bfS(z)$ is built on a procedure which is exchangeable in the $(n+1)$-sample, we will in fact need quantities that are obtained only using the $n$ sample $(V_1, \dots , V_n)$. For this, we introduce the matrices $\bfB_n \in\calM_{n,p}$, $\widehat{\mathbf{\Sigma}}_{n,\lambda}\in \calM_{p,p}$ and $\bfP_{n,\lambda} \in \calM_{n,n}$ such that
\begin{align}
    (\bfB_n)_{ij} &\coloneqq (V_i)_j - \frac{1}{n}\sum_{k=1}^n(V_k)_j, \ \
    \widehat{\mathbf{\Sigma}}_{n,\lambda} \coloneqq \frac{1}{n}\bfB_n^\top \bfB_n + \lambda \mathbf{I}_{p} = \begin{pmatrix}
                \widehat{\mathbf{\Sigma}}_{n,\lambda}^{11}   & \widehat{\mathbf{\Sigma}}_{n}^{12} \\
                \widehat{\mathbf{\Sigma}}_{n}^{21}    &\widehat{\mathbf{\Sigma}}_{n,\lambda}^{22}
\end{pmatrix}, \label{eq:def_Bn} \\
\bfP_{n,\lambda} &\coloneqq \bfB_n(\bfB_n^\top \bfB_n + n\lambda \mathbf{I}_p)^{-1}\bfB_n^\top = \frac{1}{n} \bfB_n \widehat{\mathbf{\Sigma}}_{n,\lambda}^{-1}\bfB_n^\top. \label{eq:def_Pnl}
\end{align}
Above,
$\widehat{\mathbf{\Sigma}}_{n,\lambda}^{11}\in\calM_{k,k},\ \widehat{\mathbf{\Sigma}}_{n,\lambda}^{22}\in\calM_{\ell,\ell}, \  \widehat{\mathbf{\Sigma}}_{n}^{12}\in\calM_{k,\ell},$ and $\widehat{\mathbf{\Sigma}}_{n}^{21} = (\widehat{\mathbf{\Sigma}}_{n}^{12})^\top\in\calM_{\ell,k}$.
$\bfB_n$ contains the data of the $n$-sample after being centered with their empirical mean, and $\widehat{\mathbf{\Sigma}}_{n,\lambda}$ is the corresponding empirical ridge covariance matrix. 
The matrix $\bfP_{n,\lambda}$ is a regularized orthogonal projector, that is, $\bfP_{n,0}$ is an orthogonal projector, $\bfP_{n,\lambda}^\top = \bfP_{n,\lambda}$ and $0 \preccurlyeq \bfP_{n,\lambda}\preccurlyeq \bfP_{n,0}$ for $\lambda > 0$. Denoting $p_{i,n} = (\bfP_{n,\lambda})_{ii}$ ($p_{i,n}\in [0,1]$) we set
\begin{align}
q_{n,\alpha} \coloneqq np_{(n_{\alpha})},
\end{align}
the $\nath$ order statistic of the $n$-tuple $(np_{1,n}, \dots , np_{n,n})$. Note that $p_{(n_{\alpha})}$ is the empirical quantile of order $(1-\alpha)(n+1)/n$ for the $n$-tuple $(p_{1,n}, \dots , p_{n,n})$, and not $1-\alpha$ (apply e.g. \cite{van2000asymptotic}, p. 305). Contrarily to standard $n$-samples, the $p_{i,n}$ are not independent. For example, they are constrained by the following deterministic inequality: $\sum_{i=1}^np_{i,n} = \Tr(\bfP_{n,\lambda}) \leq \Tr(\bfP_{n,0}) = \text{rank}(\bfP_{n,0}) \leq p$.

We can now state our first main result.
\begin{theorem}\label{thm:eq_ellipsoid}
Let $\alpha\in (1/(n+1),1)$. let $((X_1, Y_1), \dots , (X_{n+1},Y_{n+1}))$ \iain{$\in (\calX\times\calY)^{n+1}$} exchangeable pairs of random vectors. If $q_{n,\alpha} \geq n-1$, we set $\calE_{\alpha}^n\coloneqq \Rbb^{\ell}$. If $q_{n,\alpha} < n-1$, we define $\calE_{\alpha}^n$ as the following ellipsoid,
\begin{align}\label{eq:ellipsoid}
   \calE_{\alpha}^n = \{z\in\mathbb{R}^{\ell}: (z-Z_{0}^n)^\top \boldcalA_n^{-1}(z-Z_{0}^n) \leq \rho_{n,\alpha}\},
\end{align}
where, setting $\overline{X}_n \coloneqq n^{-1}\sum_{i=1}^nX_i, \ \overline{R}_n \coloneqq n^{-1}\sum_{i=1}^nR_i$ and $X_{n+1}^\rmc = X_{n+1} - \overline{X}_n$,
\begin{align}
    \boldcalA_n &= \widehat{\mathbf{\Sigma}}_{n,\lambda}/\widehat{\mathbf{\Sigma}}_{n,\lambda}^{11}, \label{eq:def_A} \\
    Z_{0}^n &=\widehat{\mathbf{\Sigma}}_{n}^{21}(\widehat{\mathbf{\Sigma}}_{n,\lambda}^{11})^{-1}X_{n+1}^\rmc + \overline{R}_n,\label{eq:def_z0n}\\
    \rho_{n,\alpha} &= \frac{q_{n,\alpha} + 1}{1- (q_{n,\alpha}+1)/n}-1 - (X_{n+1}^\rmc)^\top (\widehat{\mathbf{\Sigma}}_{n,\lambda}^{11})^{-1}X_{n+1}^\rmc. \label{eq:def_rho_n} 
\end{align}
Then the set $\calE_{\alpha}^n$ satisfies
\begin{align}
\mathbb{P}(Y_{n+1} - \widehat{Y}_{n+1}\in \calE_{\alpha}^n) \geq 1 - \alpha.
\end{align}
\end{theorem}
The ellipsoid in Theorem \ref{thm:intro} is recovered by $E_{\alpha}^n = \{\widehat{Y}_{n+1}\} + \calE_{\alpha}^n$, in the sense of the usual Minkowski set addition.

An important feature of Lemma \ref{lemma:simplify} is that $\boldcalA_n$ does not depend on $V_{n+1}$. Likewise, $\rho_{n,\alpha}$ and $Z_0^n$ depend on $V_{n+1}$ only through $X_{n+1}$. Interestingly, when the distribution of $V_1$ is elliptical (see Section \ref{sub:comparison_norm_score}) and $V_1$ admits second order moment, the matrix $\boldcalA_n$ corresponds to an empirical estimate of the conditional covariance matrix $\Cov(R_1|X_1)$, up to a random positive multiplicative constant (see \cite{muirhead1982}, Theorem 1.5.4). 
In fact, one can see from equation \eqref{eq:ellipsoid} that only $\rho_{n,\alpha}\boldcalA_n$ is uniquely defined. This allows us to interpret $\boldcalA_n$ as an empirical conditional covariance matrix for such distributions, up to a rescaling of $\rho_{n,\alpha}$ with the said multiplicative constant.
\begin{remark}[Empirical linear conditional expectation]\label{rk:CCLE}
Observe that $ Z_0^n = \widehat{e}_n + \widehat{\bfA}_nX_{n+1}$, where
\begin{align}
    \widehat{\bfA}_n = \widehat{\mathbf{\Sigma}}_{n}^{21}(\widehat{\mathbf{\Sigma}}_{n,\lambda}^{11})^{-1}, \ \ \widehat{e}_n = \overline{R}_n -\widehat{\bfA}_n \overline{X}_n.
\end{align}
Interestingly, it is easily shown that
$(\widehat{e}_n, \widehat{\bfA}_n)$ is an M-estimator obtained by minimizing
    \begin{align}
        T_n(e,\bfA) = \frac{1}{n}\sum_{i=1}^n\|R_i - e - \bfA X_i\|^2 + \lambda \|\bfA\|_F^2,
    \end{align}
    which simply corresponds to ridge multivariate linear regression. The continuous counterpart of $T_n$ is $T(e,\bfA) = \mathbb{E}[\|R-e-\bfA X\|^2]+\lambda\|\bfA\|_F^2$, which is minimal for $\bfA_* = \boldsymbol{\Sigma}^{21}(\boldsymbol{\Sigma}_\lambda^{11})^{-1}$ and $e_* = \mathbb{E}[R_1] - \bfA_*\mathbb{E}[X_1]$. The pair $(e_*,A_*)$ yields the best affine approximation of $R$ in terms of $X$, in the sense of $L^2(\mathbb{P})$, and the resulting affine map is known as the linear conditional expectation (\cite{linear_cond_exp}, e.g. Theorem 4.14). In comparison, the (full) conditional expectation is obtained by minimizing the map $T(f) = \mathbb{E}[\|R-f(X)\|^2]$ over all measurable maps $f$. Finally,  $\boldcalA_n$ and $(\widehat{e}_n, \widehat{\bfA}_n)$ are linked together by $\min T_n = T_n(\widehat{e}_n, \widehat{\bfA}_n) = \Tr(\boldcalA_n) - \ell\lambda$.
\end{remark}

In the previous theorem, the introduced objects are not exchangeable in terms of the $(n+1)$-sample.
Netherveless, we will show in next section how the ellipsoid $\calE_{\alpha}^n$ can be exchangeably rewritten.

\subsection{From the score matrix \texorpdfstring{$\bfS(z)$}{S(z)} to the ellipsoid \texorpdfstring{$\calE_\alpha^n$}{E alpha n}}\label{sub:to_the_score}
An important first step consists in rewriting the score matrix $\bfS(z)$, using the fact that $\bfW(z)$ can be written as a rank-one perturbation (see eq.~\eqref{eq:W(z)_rank_one}). Applying then Lemma \ref{lemma:scores_transpose} to $\bfS(z)$, we obtain that $\bfS(z)$ is of the form
\begin{align}
    \bfS(z) &= n\bfC_n - n\frac{b^n(z)b^n(z)^\top }{1 + d_n(z)}.\label{eq:scores_ simple}
\end{align}
\iain{Detailed expressions of $\bfC_n, b^n(z)$ and $d_n(z)$ are given in Lemma \ref{lemma:expr_bn_dn_rn}.} Notably, the vector $b^n(z)$ is affine in $z$ and $d^n(z)$ is a positive quadratic form in $z$.
Recall that we are interested in the diagonal terms $(\bfS(z))_{ii}$, hereinafter denoted by $S_i(z)$ for short.
Equation \eqref{eq:scores_ simple} implies that the evaluation of the full score vector $(S_1(z),\ \ldots \ , S_{n+1}(z))$ requires only one matrix inversion ($\bfD_{\mu}^{-1}$, see Lemma \ref{lemma:scores_transpose}), instead of one inversion per $z$ candidate if using formula \eqref{eq:def_score_S}. The second step consists in applying Lemma \ref{lemma:simplify}, which shows that $\bfC_n$ in equation \eqref{eq:scores_ simple} is given by
\begin{align}
    \bfC_n = \begin{pmatrix}
             \bfP_{n,\lambda}  &\mathbf{0}_{n,1} \\
             \mathbf{0}_{1,n}  &0 
    \end{pmatrix} + ww^\top \in \calM_{n+1,n+1}, \ \ \ w \coloneqq \frac{v}{\|v\|}\,.\label{eq:decomp_Cn}
\end{align}
Here, $\bfP_{n,\lambda}$ is given in equation \eqref{eq:def_Pnl}. 
In particular, for $i\in\{1, \dots , n\}$, $(\bfC_n)_{ii} = p_{i,n} + 1/n(n+1)$. In fact, $\bfC_n$ does not depend on $V_{n+1}$, while $d_n(z)$ and $b^n(z)$ only depend on $V_{n+1}$ through $X_{n+1}$. 

Note that up until now, the expression for the score $\bfS(z)$ is exact. However, a limitation remains: the scores $S_i(z)$ of the first $n$ examples also depend on $z = y - \widehat{Y}^{n+1}$.
In particular, the conformal region $C_\alpha^n$ for our score is $C_\alpha^n = \{\widehat{Y}_{n+1}\} + \calC_{\alpha}^n$, where
\begin{align}\label{eq:conformal_region}
    \calC_{\alpha}^n = \{ z\in\mathbb{R}^{\ell}: S_{n+1}(z) \leq S_{(n_{\alpha})}(z) \}.
\end{align}
Here, $S_{(n_{\alpha})}(z)$ is the $\nath$ order statistic of the $n$-tuple $(S_1(z), \dots , S_n(z))$.
As such, the computation of the region $\calC_{\alpha}^n$ still requires that the scoring rule be tested for each $y\in \mathcal{Y}_{\mathrm{trial}}$. To alleviate this second computational difficulty, a simple approximation of the confidence region is given by the next lemma.

\begin{lemma}[Conservative approximation of $\calC_{\alpha}^n$]\label{lemma:approx_score}
Let $\alpha\in (1/(n+1), 1)$. Introduce the set $\widetilde{\calC}_{\alpha}^n$, defined as
\begin{align}\label{eq:approx_q}
\widetilde{\calC}_{\alpha}^n \coloneqq \bigg\{ z \in \mathbb{R}^{\ell}: S_{n+1}(z) \leq  q_{n,\alpha} + \frac{1}{n+1} \bigg\}.
\end{align}
Then $\calC_{\alpha}^n \subset \widetilde{\calC}_{\alpha}^n$; as a result, $\Pb(Y_{n+1} - \widehat{Y}_{n+1}\in  \widetilde{\calC}_{\alpha}^n) \geq 1 -\alpha$.
\end{lemma}
Before commenting Lemma \ref{lemma:approx_score} above, we can already state our main result, which completes the proof of Theorem \ref{thm:eq_ellipsoid}.
\begin{proposition}\label{prop:c_a_is_e_a_n}
Under the assumptions of Theorem \ref{thm:eq_ellipsoid}, $\widetilde{\calC}_{\alpha}^n = \calE_{\alpha}^n$.
\end{proposition}
The approximation given in equation \eqref{eq:approx_q} corresponds to discarding the contribution $b^n(z)b^n(z)^\top$ in $S_i(z)$ for $i \neq n+1$ in equation \eqref{eq:scores_ simple}, observing that $S_i(z) \leq n(\bfC_n)_{ii}$ uniformly in $z$. This approximation is further justified by Proposition \ref{prop:bound} below, which states that under additional moment assumptions, the surplus volume converges to $0$ in probability. Note however that currently, we are not able to prove an upper bound similar to equation \eqref{eq:reverse_conformal} for $\widetilde{C}_\alpha^n$. 
\begin{proposition}\label{prop:bound}
 Assume that the $V_i$ are iid, that $\lambda>0$ and that $\mathbb{E}[\|V_1\|^{4q}] < + \infty$ for some $q>1$. Then for all compact set $K\subset \mathbb{R}^{\ell}$,
\begin{align}\label{eq:vol_surplus}
    \Vol\big((\calE_{\alpha}^n \setminus \calC_{\alpha}^n)\cap K\big) \xrightarrow[n\rightarrow \infty]{\mathbb{P}} 0.
\end{align}
\end{proposition}

\adri{The interested reader will find in Appendix \ref{app:rem_prop:bound} discussions on the settings of the above Proposition, in particular the case $\lambda=0$ and the finite $4q$-th moment assumption.}


We conclude this section with the following lemma, which provides a sufficient condition on $n$ and $p=k+\ell$ so that full space ellipsoids ($\calE_{\alpha}^n = \mathbb{R}^{\ell}$) never occur.
\begin{lemma}[Sufficient condition for bounded ellipsoids]\label{lemma:q_bound} For all $\lambda \geq 0$, we have $q_{n,\alpha} < n-1$ almost surely as soon as $n > p +1$ and
\begin{align}\label{eq:cond_non_empty}
    \alpha 
    \geq \frac{p+1}{n+1}.
\end{align}
\end{lemma}
A sharper bound, where $\lambda$ and $\max \Spec(\widehat{\mathbf{\Sigma}}_{n,\lambda})$ appear, can be obtained by writing an SVD of $\bfP_{n,\lambda}$.
Equation \eqref{eq:cond_non_empty} is to be compared with the standard requirement in conformal inference that $\alpha > 1/(n+1)$. Note also that $n+1$ and $p$ are the number of rows and columns of $\bfV$, respectively.
In practice, we observe that $q_{n,\alpha} < n-1$ may hold even if $p+1\geq n$, the proof of Lemma \ref{lemma:q_bound} being a worst-case analysis.
\paragraph*{Metrics of the ellipsoid} An important
metric related to the shape of an ellipsoid is its principal eccentricity. Given an ellipsoid $ \calE = \{ z \in\mathbb{R}^{\ell}: (z-z_0)\boldcalA^{-1}(z-z_0)\leq 1\}$, this corresponds to
\begin{align}
    e = \sqrt{1-\lambda_m/\lambda_M} \in [0,1),
\end{align}
where $\lambda_m$ is the smallest eigenvalue of $\boldcalA$ and $
\lambda_M$ is the largest one. Roughly speaking, this metric measures how ``different'' is a given ellipsoid from a ball, the latter corresponding to $e = 0$. Likewise, if $e$ is close to $1$ then $ \calE$ is accordingly flat.
The volume of $\calE$ above is
\begin{align}
\text{Vol}( \calE) = v_{\ell} \sqrt{\det(\boldcalA)}, \quad \text{where}  \quad  v_{\ell} = {\pi^{\ell/2}}/{\Gamma({\ell}/{2}+1)}.
\end{align}
Here, $ v_{\ell}$ corresponds to the volume of the $\ell$-dimensional unit ball.
In our case, the determinant of $\boldcalA_n$ can be further computed as (see e.g. \cite{horn}, Section 0.8.5)
\begin{align}
\det(\boldcalA_n) &= \det\big(\widehat{\mathbf{\Sigma}}_{n,\lambda}/\widehat{\mathbf{\Sigma}}_{n,\lambda}^{11}\big) = \det\big(\widehat{\mathbf{\Sigma}}_{n,\lambda}\big)/\det\big(\widehat{\mathbf{\Sigma}}_{n,\lambda}^{11}\big).
\end{align}
\paragraph*{Case where $\calE_{\alpha}^n$ is empty}
Empty confidence regions $\calE_{\alpha}^n$ correspond to $\rho_{n,\alpha} < 0$. This event has a non-null probability, as e.g. visible from equation \eqref{eq:proba_empty} (still, $\mathbb{P}(\calE_{\alpha}^n = \varnothing) \leq \alpha$ by construction). To understand the event $\{\calE_{\alpha}^n = \varnothing\}$, observe that $\calE_{\alpha}^n$ can be written as $\calE_{\alpha}^n = \mathcal{V}(X_{n+1})\cap \calE$, where $\mathcal{V}(X_{n+1})$ is an $\ell$-dimensional affine subspace of $\mathbb{R}^{k+\ell}$ and $\calE$ is a $(k+\ell)$-dimensional ellipsoid. They are given by, for some explicit $\beta_n>0$,
\begin{align*}
\mathcal{V}(X_{n+1}) &\coloneqq \{v = (x^\top r^\top)^\top\in\mathbb{R}^{k+\ell} : x = X_{n+1}\},\\
\calE &\coloneqq \{v\in\mathbb{R}^{k+\ell} : (v-\overline{V}_n)^\top(\widehat{\mathbf{\Sigma}}_{n,\lambda})^{-1}(v-\overline{V}_n) \leq \beta_n \}, \ \ \overline{V}_n = \frac{1}{n}\sum_{i=1}^nV_i.
\end{align*}
 (to obtain $\calE$, combine equations \eqref{eq:rewrite_cna} and \eqref{eq:recover_rnz}). In this setting, $\calE_{\alpha}^n$ is empty precisely when $\mathcal{V}(X_{n+1})\cap\calE = \varnothing$. \iain{This is interpreted as $X_{n+1}$ being so abnormal (w.r.t. the score $S(z)$) that no value of $z$ can make the pair $(X_{n+1}^\top \  z^\top)^\top$ conform to the samples $(V_1, \ldots, V_n)$. While this is unexpected from a UQ perspective, this is natural in the context of outlier detection : $X_{n+1}$ is so abnormal that $(X_{n+1}^\top \  z^\top)^\top$ is always an outlier, whatever the value of $z$.} 

Empty confidence regions can be problematic from an applied perspective. Moreover, empty $\calE_{\alpha}^n$ are not devoid of information, since $\boldcalA_n$ and $Z_0^n$ are still defined when $\rho_{n,\alpha}<0$. To capitalize on $\boldcalA_n$ and $Z_0^n$ as well as avoid empty confidence regions, we can replace $\rho_{n,\alpha}$ with $\rho_{n,\alpha}^{\varepsilon} = \max(\rho_{n,\alpha},\varepsilon^2)$, where $\varepsilon>0$ is a length associated to a minimum user set volume $V_{\min}$. The ellipsoid is then given by
\begin{align}
    (z-Z_0^n)^\top (\rho_{n,\alpha}^{\varepsilon}\boldcalA_n)^{-1}(z-Z_0^n) &\leq 1.
\end{align}
$\varepsilon$ is found by solving $V_{\min} = v_{\ell}\sqrt{\det(\rho_{n,\alpha}^{\varepsilon}\boldcalA_n)} = v_{\ell}{\varepsilon}^{\ell}\sqrt{\det(\boldcalA_n)}$, i.e.
\begin{align}
    \varepsilon = \Big(V_{\min}/\big(v_{\ell}\det(\boldcalA_n)^{1/2}\big)\Big)^{1/\ell}.
\end{align}
By construction, this procedure leads to overcoverage, but if $V_{\min}$ is chosen small enough (e.g. if $V_{\min} \ll \mathbb{E}[\Vol(\calE_{\alpha}^n)]$), it will not affect the empirical average volume of the confidence ellipsoid $\calE_{\alpha,\varepsilon}^n$, since the cutoff $\rho_{\varepsilon} = \max(\rho,\varepsilon^2)$ is only activated when the volume of $\calE_{\alpha}^n$ is very small or null.

\iain{A more mathematically sound way of avoiding empty regions is found by using the score $S'(z)$ introduced below.}

\subsection{The adjusted Mahalanobis score and the second ellipsoid \texorpdfstring{$\calF_\alpha^n$}{F alpha n} } \label{ssec:adjusted_score}
\iain{Following equation \eqref{eq:def_sprime_infinite}, we can correct the score matrix \eqref{eq:def_score_S} to prevent empty confidence regions, by subtracting the $\bfX$ part of $\bfS(z)$ to $\bfS(z)$. Explicitly, define $\bfF \coloneqq \pi_{\mathbbm{1}}^\perp\bfX \in \calM_{n+1,k}$, where the $\ith$ row of $\bfX$ is $X_i^\top$, and introduce
\begin{align}\label{eq:def_Sp_Sx}
    \bfS'(z)  \coloneqq \bfS(z) - \bfS_X, \ \ \ \bfS_X = \bfF\bigg(\frac{1}{n} \bfF^\top  \bfF + \lambda \mathbf{I}_k\bigg)^{-1}\bfF^\top.
\end{align}
As with $\bfS(z)$, the associated score for $V_i$ is $S_i'(z) = \bfS'(z)_{ii}$ and the associated conformal set is $\{z : S_{n+1}'(z) \leq S_{(n_\alpha)}'(z)\}$, where $S_{(n_\alpha)}'(z)$ is an order statistic built from $(S_1'(z), \ldots, S_n'(z))$. 
Because $\bfX$ depends on $X_{n+1}$ through a rank-one perturbation, namely $\bfX = (X_1 \ \ldots \ X_n \ 0_{k,1})^\top  + e_{n+1}X_{n+1}^\top$,
we can use Lemmas \ref{lemma:scores_transpose} and \ref{lemma:simplify} to rewrite $\bfS_X$ as
\begin{align}
    \frac{1}{n}\bfS_X = \begin{pmatrix}
             \bfP_{n,\lambda}^{XX}  &\mathbf{0}_{n,1} \\
             \mathbf{0}_{1,n}  &0 
    \end{pmatrix} + ww^\top -  \frac{b_X^n(b_X^n)^\top}{1+d_n^X},
\end{align}
for some matrix $P_{n,\lambda}^{XX} \in \calM_{n,n}$, some vector $b_X^n \in \calM_{n+1,1}$ and some $d_n^X \in \mathbb{R}^+$ (see Lemma \ref{lemma:expr_bnx_dnx_rnx} for expressions). Notably,
\begin{align*}
    d_n^X &= \frac{1}{n+1}(X_{n+1}^\rmc)^\top (\mathbf{\Sigma}_{n,\lambda}^{11})^{-1}X_{n+1}^\rmc, \ \ \ X_{n+1}^\rmc = X_{n+1} - \overline{X}_n.
    \end{align*}
The matrix $\bfS'(z)$ is finally written as
\begin{align*}
    \bfS'(z) = n\begin{pmatrix}
             \bfP_{n,\lambda}-\bfP_{n,\lambda}^{XX}  &\mathbf{0}_{n,1} \\
             \mathbf{0}_{1,n}  &0 
    \end{pmatrix} - n \frac{b^n(z)b^n(z)^\top}{1+d_n(z)} + n \frac{b_X^n(b_X^n)^\top}{1+d_n^X}.
\end{align*}
As in Lemma \ref{lemma:approx_score}, we can discard the $b^n(z)$ term in $S_i'(z)$, for $i\neq n+1$, leading to a second ellipsoidal set $\calF_\alpha^n$. To describe this set, we finally need to introduce
\begin{align}\label{eq:def_pn_qn_prime}
    p_{i,n}' &\coloneqq (\bfP_{n,\lambda} - \bfP_{n,\lambda}^{XX})_{ii} + \frac{(b_X^n)_i^2}{1+d_n^X}, \ \ \ \text{and} \ \ \  q_{n,\alpha}' \coloneqq np_{(n_\alpha)}',
\end{align}
where $p_{(n_\alpha)}'$ is the order statistic of order $n_\alpha$ of $(p_{1,n}', \ldots, p_{n,n}')$.
As in Lemma \ref{lemma:approx_score}, $p_{i,n}'$ corresponds to $\bfS'(z)_{ii}$ after discarding the $b^n(z)_i$ term. The analog of Theorem \ref{thm:eq_ellipsoid} for the score $S'(z)$ is the following, which describes a second ellipsoid $\calF_\alpha^n.$
\begin{theorem}\label{thm:ellip_prime}
Take the assumptions of Theorem \ref{thm:eq_ellipsoid}, and define 
\begin{align}\label{eq:def_tna}
t_{n} \coloneqq \bigg(\frac{n+1}{n}\bigg)(1+d_n^X).
\end{align}
If $t_{n}q_{n,\alpha}'\geq n$, we set $\calF_\alpha^n = \mathbb{R}^\ell$. Else we set
\begin{align}
\calF_\alpha^n \coloneqq \{z\in\mathbb{R}^{\ell}: (z-Z_{0}^n)^\top \boldcalA_n^{-1}(z-Z_{0}^n) \leq \rho_{n,\alpha}'\},    
\end{align}
where $Z_0^n$ and $\boldcalA_n$ are given in Theorem \ref{thm:eq_ellipsoid}, and where
\begin{align}
    \rho_{n,\alpha}' \coloneqq t_{n}^2\frac{q_{n,\alpha}'}{1-t_{n}q_{n,\alpha}'/n}.
 \label{eq:def_rho_n_prime}
\end{align}Then
\begin{align}
    \mathbb{P}(Y_{n+1} - \widehat{Y}_{n+1}\in \calF_{\alpha}^n) \geq 1 - \alpha.
\end{align}
\end{theorem}
As for $E_\alpha^n$ and $\calE_\alpha^n$, the ellipsoid for $Y_{n+1}$ is recovered by $F_\alpha^n = \{\widehat{Y}_{n+1}\} + \calF_\alpha^n$.
The only difference between $\calE_\alpha^n$ and $\calF_\alpha^n$ is the parameter $\rho_{n,\alpha}'$.
The next proposition shows that $\calF_\alpha^n$ is never empty. However, there is a chance that $\calF_\alpha^n=\mathbb{R}^\ell$: there is no equivalent of Proposition \ref{prop:bound} for $\calF_\alpha^n$, except when $n$ tends to $\infty$.
\begin{proposition}\label{prop:calF_full_space}
Under the assumptions of Theorem \ref{thm:eq_ellipsoid}, $\rho_{n,\alpha}' \geq 0$. In particular, $Z_0^n\in \calF_\alpha^n$ and $\calF_\alpha^n$ is never empty. Finally, under the assumptions of Proposition \ref{prop:asymptotics},
\begin{align*}
    \mathbb{P}(\calF_\alpha^n = \mathbb{R}^\ell) \xlongrightarrow[n\rightarrow \infty]{} 0.
\end{align*}
\end{proposition}
\begin{remark}[Behaviour of $\calE_\alpha^n$ and $\calF_\alpha^n$ w.r.t. large $d_n^X$]
Observe that a large $d_n^X$ means that $X_{n+1}$ is far away from $\overline{X}_n$, in the data metric $\mathbf{\Sigma}_{n,\lambda}$. Interestingly, $\calE_\alpha^n$ and $\calF_\alpha^n$ have opposite behaviours with reference to large $d_n^X$. From equation \eqref{eq:def_rho_n}, large values of $d_n^X$ correspond to small $\calE_\alpha^n$ ($\calE_\alpha^n = \varnothing$ in the worst case, see Proposition \ref{thm:eq_ellipsoid}), while they correspond to large $\calF_\alpha^n$ ($\calF_\alpha^n = \mathbb{R}^\ell$ in the worst case, see Proposition \ref{thm:ellip_prime}). In the first case, this is interpreted as $X_{n+1}$ being so abnormal that no value of $z$ can make the pair $(X_{n+1}^\top \  z^\top)^\top$ conform to the samples $(V_1, \ldots, V_n)$. In the second case, this is interpreted as the whole calibration sample $(V_1, \ldots, V_n)$ being uninformative with reference to the current sample $(X_{n+1}^\top \ R_{n+1}^\top)^\top$, when the new input $X_{n+1}$ is too abnormal when compared to $(X_1, \ldots, X_n)$. 
\end{remark}
}
\section{Asymptotic analysis}
\label{sec:asymp}
\subsection{The first asymptotic ellipsoid \texorpdfstring{$\calE_\alpha^\infty$}{E alpha infinity}} \label{ssec:asymp_ellips}
In the limit where $n\rightarrow \infty$, we are able to state the following result. 
\begin{proposition}\label{prop:asymptotics}
Assume that $V_1$ has a well-defined covariance matrix $\mathbf{\Sigma}$, and that the vectors $V_i,\ i\in\mathbb{N}$ are iid. Denote $V_\rmc \coloneqq V_1 - \mathbb{E}[V_1]$, with the decomposition $V_\rmc = (X_\rmc^\top R_\rmc^\top)^\top, \ X_\rmc\in\mathbb{R}^k, \ R_\rmc\in\mathbb{R}^\ell$. 
Assume also that $\lambda>0$ and that the quantile function of $\|(\mathbf{\Sigma} + \lambda \mathbf{I}_p)^{-1/2}V_\rmc\|^2$ is continuous on a neighbourhood of $1-\alpha$.
Write $\mathbf{\Sigma} + \lambda \mathbf{I}_p$ in blockwise fashion, according to the decomposition $p = k + \ell$:
\begin{align}\label{eq:decomp_sigma}
    \mathbf{\Sigma}_{\lambda} &\coloneqq \mathbf{\Sigma} + \lambda \mathbf{I}_{p} = \begin{pmatrix}
                \mathbf{\Sigma}_{\lambda}^{11}   &\mathbf{\Sigma}^{12} \\
                \mathbf{\Sigma}^{21}    &\mathbf{\Sigma}_{\lambda}^{22}
\end{pmatrix}.
\end{align}
($\mathbf{\Sigma}_{\lambda}^{11}\in\calM_{k,k},\ \mathbf{\Sigma}_{\lambda}^{22}\in\calM_{\ell,\ell}, \ \mathbf{\Sigma}^{12}\in\calM_{k,\ell}$.) Denote $q_{1-\alpha}^{\calE} \coloneqq q_{1-\alpha}(V_\rmc^\top \mathbf{\Sigma}_{\lambda}^{-1}V_\rmc)$. Then, as $n\rightarrow \infty$,
\begin{enumerate}
\item (Asymptotic ellipsoid)
\begin{align}
q_{n,\alpha} &\xlongrightarrow[n\rightarrow \infty]{a.s.} q_{1-\alpha}^{\calE},\\
\boldcalA_n &\xlongrightarrow[n\rightarrow \infty]{a.s.} \boldcalA_{\infty} \coloneqq  \mathbf{\Sigma}_{\lambda}/\mathbf{\Sigma}_{\lambda}^{11},\label{eq:cv_An}\\
\rho_{n,\alpha} &\xlongrightarrow[n\rightarrow \infty]{\mathrm{d}} \rho_{\infty,\alpha} \coloneqq q_{1-\alpha}^{\calE} - X_\rmc^\top \big(\mathbf{\Sigma}_{\lambda}^{11}\big)^{-1}X_\rmc,\label{eq:cv_rho_n}\\
Z_0^n &\xlongrightarrow[n\rightarrow \infty]{\mathrm{d}} Z_0^{\infty} \coloneqq \mathbf{\Sigma}^{21}(\mathbf{\Sigma}_{\lambda}^{11})^{-1}X_\rmc + \mathbb{E}[R_1]. \label{eq:pred_unbiased}
\end{align}
    \item (Asymptotic volume and probability of empty regions) 
    \begin{align}
        \mathrm{Vol}(\calE_{\alpha}^n) &\xlongrightarrow[n\rightarrow \infty]{\rmd} v_{\ell}\sqrt{\det(\mathbf{\Sigma}_{\lambda}/\mathbf{\Sigma}_{\lambda}^{11})}\big(q_{1-\alpha}^{\calE} - X_\rmc^\top (\mathbf{\Sigma}_{\lambda}^{11})^{-1}X_\rmc\big)^{\ell/2}_+,\label{eq:asymptotic_vol} \\
        \mathbb{P}(\calE_{\alpha}^n = \varnothing) &\xlongrightarrow[n\rightarrow \infty]{} \mathbb{P}(X_\rmc^\top (\mathbf{\Sigma}_{\lambda}^{11})^{-1}X_\rmc > q_{1-\alpha}^{\calE}) \leq \alpha \label{eq:proba_empty}.
        \end{align}
\end{enumerate}
Finally, if $\mathbf{\Sigma}$ is invertible, then all the results above remain true when $\lambda = 0$.
\end{proposition}
In equation \eqref{eq:asymptotic_vol}, we underline that $(x)_+^s = \max(0,x)^s\neq \max(0,x^s)$; in particular, $(x)_+^s = 0$ if $x\leq 0$. We now define the random ellipsoid $\calE_{\alpha}^{\infty}$ as
\begin{align}
    \calE_{\alpha}^{\infty} \coloneqq \{ z\in\mathbb{R}^{\ell} : (z - Z_0^{\infty})^\top\boldcalA_{\infty}^{-1}(z-Z_0^{\infty}) \leq \rho_{\infty,\alpha} \}.
\end{align}
Observe that $\calE_{\alpha}^{\infty}$ is a confidence region of level $1-\alpha$ for $R_1$: from equation \eqref{eq:apply_block_lemma} and the expressions of $\boldcalA_{\infty}$, $Z_0^{\infty}$ and $\rho_{\infty,\alpha}$, we can show that (see equations \eqref{eq:decomp_sigma_lambda_proof} and \eqref{eq:facto_quad_conditionnel})
\begin{align}\label{eq:conditional_ellipsoid_asymptotic}
    V_\rmc^\top \mathbf{\Sigma}_{\lambda}^{-1}V_\rmc \leq q_{1-\alpha}^{\calE} \iff (R_\rmc - Z_0^{\infty})^\top\boldcalA_{\infty}^{-1}(R_\rmc-Z_0^{\infty}) \leq \rho_{\infty,\alpha},
\end{align}
and the left-hand side has probability $1-\alpha$ from the definition of $q_{1-\alpha}^{\calE}$.
Next, the distribution of the volume $\Vol(\calE_{\alpha}^{\infty})$ volume is equal to that of the right-hand side of \eqref{eq:asymptotic_vol}. In this paper, we do not tackle the question of random set convergence, i.e in what sense does the convergence $\calE_{\alpha}^n \rightarrow \calE_{\alpha}^{\infty}$
hold. This study would require the use of random set theory \cite{molchanov}, a perspective that we leave for future work.

Note that the limit volume $\Vol(\calE_{\alpha}^{\infty})$ is bounded almost surely, hence all its moments are finite (equation \eqref{eq:part_pos_control}). However, there is no guarantee that the moments of $\mathrm{Vol}(\calE_{\alpha}^n)$ converge toward those of $\mathrm{Vol}(\calE_{\alpha}^{\infty})$. Such a property requires additional uniform integrability properties on $\mathrm{Vol}(\calE_{\alpha}^n)$, which have to be checked on a case-by-case basis. For example, this property holds in the Gaussian case (Proposition \ref{prop:cas_gaussien}).
The main obstacle in the proof \iain{of Proposition \ref{prop:asymptotics}} is that of the convergence of the empirical quantile $q_{n,\alpha}$ to $q_{1-\alpha}^{\calE}$, as the diagonal elements of $\bfC_n$ are not independent. Our proof of this result relies on the study of empirical characteristic functions and an almost sure application of Lévy's theorem (Lemma \ref{lemma:cv_quantile}).
Note that the asymptotic matrix $\boldcalA_{\infty}$ is the covariance matrix of $\pi^{\perp}_XR$, where $\pi^{\perp}_XR$ is the orthogonal projection of $R = (R_1, \dots , R_{\ell})$ onto $[\mathrm{Span}(X_1, \dots , X_k)]^{\perp}$ in $L^2(\mathbb{P})$.

\adri{
It is worth noticing that the numerical experiments highlight that the asymptotic ellipsoid (as well as its volume) may still exist when the data are heavy tailed. A starting point to study that case is to note that, according to equation \eqref{eq:ellipsoid}, the region $\calE_{\alpha}^n$ is mostly determined by the product $\rho_{n,\alpha}\boldcalA_{n}$. Then scalings differing from the one used in Proposition \ref{prop:asymptotics} (i.e.  $(\rho_{n,\alpha}/u_n)(u_n\boldcalA_n)$ with $u_n \neq 1$) may be introduced, in the hope that $u_n\boldcalA_n$ and $\rho_{n,\alpha}/u_n$ may converge in some sense. The interested reader will find first insights in Appendix~\ref{app:infinite_var} for a further study of heavy-tailed data.
}

\begin{remark}[Conditional coverage]\label{rk:conditional_coverage}
Although our method is built on a form of conditional procedure, it cannot fulfill the conditional coverage \iain{described in the introduction,} even when $n\rightarrow \infty$. Indeed, if $X_1 = x$, then $\calE_\alpha^\infty$ is empty as soon as $q_{1-\alpha}^{\calE} < (x-\mathbb{E}[X_1])^\top(\mathbf{\Sigma}^{11})^{-1}(x-\mathbb{E}[X_1])$ (Proposition \ref{prop:asymptotics}). In particular, for all such $x$,
\begin{align}
    \mathbb{P}(R_1\in\calE_\alpha^\infty|X_1=x) = \mathbb{P}(R_1\in\varnothing|X_1=x) = 0 \neq 1-\alpha.
\end{align}
    To recover conditional coverage, the asymptotic quantile $q_{1-\alpha}^{\calE} = q_{1-\alpha}(V_\rmc^\top\boldsymbol{\Sigma}^{-1}V_\rmc)$ should be replaced with $q_{1-\alpha}^{\calE}(x) \coloneqq q_{1-\alpha}(V_\rmc^\top\boldsymbol{\Sigma}^{-1}V_\rmc|X_1=x)$ in equation \eqref{eq:cv_rho_n}. Indeed, let us define
    \begin{align}
    \rho_{\infty,\alpha}(x) \coloneqq q_{1-\alpha}^{\calE}(x) - (x-\mathbb{E}[X_1])^\top(\mathbf{\Sigma}^{11})^{-1}(x-\mathbb{E}[X_1]),\label{eq:def_rho_of_x}
    \end{align}
    and define $\calE_\alpha^\infty(x)$ to be the set $\calE_\alpha^\infty$ where $\rho_{\infty,\alpha}$ has been replaced with $\rho_{\infty,\alpha}(x)$. Then, from equation \eqref{eq:facto_quad_conditionnel}, we recover conditional coverage in the limit where $n\rightarrow \infty$:
    \begin{align}\label{eq:conditional_ellipsoid_asymptotic_2}
    \mathbb{P}(R_1\in\calE_\alpha^\infty(x)|X=x) &=
    \mathbb{P}((R_\rmc - Z_0^{\infty})^\top\boldcalA_{\infty}^{-1}(R_\rmc-Z_0^{\infty}) \leq \rho_{\infty,\alpha}(x)|X_1=x)\nonumber \\
    &=
\mathbb{P}(V_\rmc^\top \mathbf{\Sigma}_{\lambda}^{-1}V_\rmc \leq q_{1-\alpha}^{\calE}(x)|X_1=x) = 1-\alpha.
\end{align}

This solution also solves the problem of empty confidence regions described in Section \ref{sec:score}. Of course, the main difficulty is that of the estimation of the conditional quantile $q_{1-\alpha}(x)$. This can be done e.g. using conformalized quantile regression \cite{Romanoetal19}, although the application of this method has to be further studied in our case.
\end{remark}

\paragraph*{Gaussian data} In this paragraph, we assume that $V_1$ is a Gaussian random vector. In this case, $Z_0^\infty = \mathbb{E}[R_1|X_1]$ and $\boldcalA_{\infty} = \Cov(R_1|X_1)$. In fact, the limit ellipsoid $\calE_\alpha^\infty$ is \iain{an $\ell$-dimensional section of the $k+\ell$-dimensional confidence ellipsoid} obtained under Gaussianity assumptions over the random vector $(X_1, R_1)$ (this is also true under general ellipticity assumptions, see equation \eqref{eq:ellip_assumption} for a definition). CCLE can thus be understood as a ``non asymptotic conformal \iain{generalisation}'' of the regions provided by standard Gaussian ellipsoids. In the Gaussian case, we can further describe the volume of $\calE_{\alpha}^\infty$ when $\lambda=0$, notably its moments.
\begin{proposition}\label{prop:cas_gaussien}
Assume that the vectors $V_i$ are iid, the vector $V_1$ is Gaussian, $\lambda = 0$ and $\min\Spec(\mathbf{\Sigma})>0$. Denoting $F_{\chi^2(m)}$ the CDF of the $\chi^2(m)$ distribution, we have $q_{1-\alpha}^{\calE} = F^{-1}_{\chi^2(k+\ell)}(1-\alpha)$. Next, denoting $B(x,y)$ the Euler Beta function, we set $C_{k,\ell,q} \coloneqq 2^{-k/2}v_{\ell}^qB({k}/{2},{q\ell}/{2}+1)/\Gamma(k/2)$. Then for all $q>0$,
        \begin{align}
          \mathbb{E}[\mathrm{Vol}(\calE_{\alpha}^{n})^q] &\xrightarrow[n\rightarrow \infty]{}  \mathbb{E}[\mathrm{Vol}(\calE_{\alpha}^{\infty})^q] =
            C_{k,\ell,q}\det(\mathbf{\Sigma}/\mathbf{\Sigma}^{11})^{q/2} (q_{1-\alpha}^{\calE})^{(k+q\ell)/2}
          {_1F_1}\bigg(\frac{k}{2}, \frac{k+q\ell}{2}+1, -\frac{q_{1-\alpha}^{\calE}}{2}\bigg), \label{eq:exp_vol_gauss} \\[2ex]
          \mathbb{P}(\calE_{\alpha}^{n} = \varnothing) &\xrightarrow[n\rightarrow \infty]{}\mathbb{P}(\calE_{\alpha}^\infty = \varnothing) = 1 - F_{\chi^2(k)}\Big(F^{-1}_{\chi^2(k+\ell)}(1-\alpha)\Big), \label{eq:proba_empty_gauss}
        \end{align}
where $_1F_1$ is the Kummer confluent hypergeometric function of the first kind. In particular,
\begin{align}\label{eq:expected_vol_gaussian_calE}
     \mathbb{E}[\mathrm{Vol}(\calE_{\alpha}^{n})] &\xlongrightarrow[n\rightarrow \infty]{} \frac{2^{-k/2}\pi^{\ell/2}}{\Gamma(p/2+1)}\det(\mathbf{\Sigma}/\mathbf{\Sigma}^{11})^{1/2}(q_{1-\alpha}^{\calE})^{p/2} {_1F_1}\bigg(\frac{k}{2}, \frac{p}{2}+1, -\frac{q_{1-\alpha}^{\calE}}{2}\bigg).
\end{align}
\end{proposition}
Formula \eqref{eq:exp_vol_gauss} is empirically verified in Table \ref{tab:expectations}, \iain{and Formula \eqref{eq:proba_empty_gauss} is illustrated in Figure \ref{fig:empty_gaussian}}.
\adri{One may wish to understand the properties of $\calE_{\alpha}^n$ for finite sample size $n$. The interested reader will find in Appendix~\ref{app:non_asymp_gauss} some discussions on that point, where \iain{we identify the distributions of $\boldcalA_n$ and $d_n^X$, and approximations of those of $\bfP_{n,0}$ and $p_{i,n}$}. However we did not manage to identify that of $\rho_{n,\alpha}$.}
\subsection{The second asymptotic ellipsoid \texorpdfstring{$\calF_\alpha^\infty$}{F alpha infinity}}
\iain{
Under the assumptions of Proposition \ref{prop:asymptotics}, the center $Z_0^n$ and ``matrix'' $\boldcalA_n$ describing the ellipsoid $\calF_\alpha^n$ obviously still converge to the limits given in equations \eqref{eq:cv_An} and \eqref{eq:pred_unbiased}, as they are the same as those of $\calE_\alpha^n$. The limit of $\rho_{n,\alpha}'$ is given in the following proposition.
\begin{proposition}\label{prop:calf_asymptotic}
Take the assumptions of Proposition \ref{prop:asymptotics}, and set 
\begin{align}\label{eq:def_qaf}
 q_{1-\alpha}^\calF \coloneqq q_{1-\alpha}(T_\rmc^\top\boldcalA_\infty^{-1}T_\rmc), \ \ \text{where} \ \ \ T_\rmc = R_\rmc - \mathbf{\Sigma}^{21}(\mathbf{\Sigma}_\lambda^{11})^{-1}X_\rmc.
\end{align}
If $\mathbb{E}[\|X_1\|^{4q}]<+\infty$ for some $q>1$, then
\begin{align}
    q_{n,\alpha}' &\xlongrightarrow[n\rightarrow\infty]{\mathbb{P}} q_{1-\alpha}^\calF,\ \ \ \
    \rho_{n,\alpha}'\xlongrightarrow[n\rightarrow\infty]{\mathbb{P}} q_{1-\alpha}^\calF, \ \ \ \ \text{ and } \ \ \ \
    \Vol(\calF_\alpha^n) \xlongrightarrow[n\rightarrow\infty]{\mathbb{P}} v_\ell\det(\mathbf{\Sigma}_\lambda/\mathbf{\Sigma}_\lambda^{11})^{1/2}(q_{1-\alpha}^\calF)^{\ell/2}.
\end{align}
\end{proposition}
We denote by $\calF_\alpha^\infty$ the ellipsoid with center $Z_0^\infty$, matrix $\boldcalA_\infty$ and ``squared radius'' $q_{1-\alpha}^\calF$.
While the convergence of $q_{n,\alpha}'$ is weaker than that of $q_{n,\alpha}$, the convergence of $\rho_{n,\alpha}'$ is as strong as that of $\rho_{n,\alpha}$ (because the limit of  $\rho_{n,\alpha}'$ is constant, convergence in probability is equivalent to convergence in distribution).}

\iain{Observe also that contrarily to Proposition \ref{prop:calF_full_space}, we additionally require that $\mathbb{E}[\|X_1\|^{4q}]<+\infty$. While this assumption is not necessarily tight, it suggests that the convergence of $\calF_\alpha^n$ towards $\calF_\alpha^\infty$ may be slower than that of $\calE_\alpha^n$ towards $\calE_\alpha^\infty$, or that it should hold in a weaker sense. This is consistent with the numerical experiments on Cauchy data (Section \ref{sub:num_cauchy}).}

\iain{The remark on the conditional coverage of $\calE_\alpha^\infty$ in the previous section also holds for $\calF_\alpha^\infty$ if replacing $q_{1-\alpha}^{\calF}$ with $q_{1-\alpha}^{\calF}(x) \coloneqq q_{1-\alpha}(T_\rmc^\top\boldcalA_\infty^{-1}T_\rmc|X_1=x)$. In fact, it is equivalent to estimate $q_{1-\alpha}^{\calF}(x)$ or $q_{1-\alpha}^{\calE}(x)$, as from equation \eqref{eq:facto_quad_conditionnel},
\begin{align*}
    q_{1-\alpha}^{\calE}(x) = q_{1-\alpha}^{\calF}(x) + (x-\mathbb{E}[X_1])^\top(\mathbf{\Sigma}^{11})^{-1}(x-\mathbb{E}[X_1]).
\end{align*}
In particular, $\rho_{\infty,\alpha}(x) = q_{1-\alpha}^{\calF}(x)$, and the conditional ellipsoids $\calE_\alpha^\infty(x)$ and $\calF_\alpha^\infty(x)$ are the same.
As opposed to $\calE_\alpha^\infty$ though, we recover the conditional coverage for Gaussian data, when using $\calF_\alpha^\infty$.
\begin{proposition}\label{prop:conditional_gaussian}
Assume that $V_1\sim\calN(m,\mathbf{\Sigma})$, that $\min\Spec \mathbf{\Sigma} >0$, and set $\lambda = 0$. Then
\begin{align*}
    \forall x \in \mathbb{R}^k, \ \ \ \mathbb{P}(Y_{1} - \widehat{Y}_{1} \in \calF_\alpha^\infty|X_{1} = x) = \mathbb{P}(Y_{1} - \widehat{Y}_{1} \in \calF_\alpha^\infty) = 1-\alpha.
\end{align*}
\end{proposition}
This result is an immediate consequence of the fact that $R_{1} - \mathbf{\Sigma}^{21}(\mathbf{\Sigma}^{11})^{-1}X_{1}$ is independent from $X_{1}$, when $V_1$ is Gaussian with $\Cov(V_1) = \mathbf{\Sigma}$. Anticipating Section \ref{sub:comparison_norm_score}, this property is unique to Gaussian distributions, among the family of elliptical distributions (see equation \eqref{eq:ellip_assumption} for a definition). In the general elliptical case, the conditional distributions $R_1|X_1=x$ and $R_1-\mathbf{\Sigma}^{21}(\mathbf{\Sigma}^{11})^{-1}X_{1}$ still have the same dispersion matrix $\mathbf{\Sigma}/\mathbf{\Sigma}^{11}$, but they differ in their generator. Coming back to the Gaussian case, numerical evaluations of formula \eqref{eq:expected_vol_gaussian_calE} show that if $V_1$ is Gaussian and $\lambda=0$, then on average, $\calF_\alpha^\infty$ is better than $\calE_\alpha^\infty$ in terms of volume (see Figure \ref{fig:ratio_vol_gaussian} for a plot) :
\begin{align}
    \forall \alpha\in (0,1), \ \ \ \Vol(\calF_\alpha^\infty) \leq  \mathbb{E}[\Vol(\calE_\alpha^\infty)].
\end{align}
It remains possible that $\Vol(\calE_\alpha^\infty) \leq \Vol(\calF_\alpha^\infty)$, which corresponds to the event  $\{X_\rmc^\top (\mathbf{\Sigma}^{11})^{-1}X_\rmc \geq q_{1-\alpha}^\calE - q_{1-\alpha}^\calF\}$.
}

\paragraph*{The corrected predictor}
 An interesting feature of the ellipsoid\adri{s} for $Y_{n+1}$ is that \adri{they are} not centered at the predictor $\widehat{Y}_{n+1}$. \adri{Their common} center is $\widetilde{Y}_{n+1} = Z_0^n + \widehat{Y}_{n+1}$ which is effectively a correction of $\widehat{Y}_{n+1}$. In fact, there is no guaranty that $\widehat{Y}_{n+1}$ lies in $E_\alpha^n$ \adri{or $F_\alpha^n$}, which is desirable as the predictor may e.g. be biased. On the contrary, the next proposition shows that under same moment assumptions as for Proposition~\ref{prop:asymptotics}, the corrected predictor $\widetilde{Y}_{n+1}$ is asymptotically unbiased.
\begin{proposition}\label{prop:pred_unbiased}If $\mathbb{E}[\|Y_1\|]<+\infty$ then
under the assumptions of Proposition \ref{prop:asymptotics}, $\widetilde{Y}_{n+1}$ is asymptotically unbiased, i.e.
\begin{align}\label{eq:unbiased}
    \mathbb{E}[\widetilde{Y}_{n+1}] \xrightarrow[n\rightarrow \infty]{} \mathbb{E}[Y_1].
\end{align}
Moreover, introducing the matrix $M_{\lambda} \coloneqq \mathbf{\Sigma}_{\lambda}/\mathbf{\Sigma}_{\lambda}^{11} - \lambda (\mathbf{I}_{\ell} + \mathbf{\Sigma}^{21}(\mathbf{\Sigma}_{\lambda}^{11})^{-2}\mathbf{\Sigma}^{12})$, and assuming there exists $q>1$ such that $\mathbb{E}[\|V_1\|^{4q}] < + \infty$, we have the asymptotic covariance and mean squared error,
\begin{align}
     \Cov(\widetilde{Y}_{n+1}-{Y}_{n+1})&\xrightarrow[n\rightarrow \infty]{}M_{\lambda}, \label{eq:cov_assymptotic} \ \ \ \mathbb{E}\Big[\|\widetilde{Y}_{n+1}-{Y}_{n+1}\|_2^2\Big] \xrightarrow[n\rightarrow \infty]{} \Tr(M_{\lambda}).
\end{align}
\end{proposition}
\subsection{Comparison with the norm residual score}\label{sub:comparison_norm_score}
In regression, the standard score is the norm of the residual, $S_i'(y) = \|y-\widehat{Y}_{n+1}\|$, where $\| \cdot\|$ is the Euclidean norm in $\Rbb^{\ell}$. The corresponding confidence region for $Y_{n+1}$ is a ball $B_{\alpha}^n$ centered at $\widehat{Y}_{n+1}$, with squared radius $\beta_{n,\alpha}$ given by the $\nath$ order statistic of $(\|R_1\|^2, \dots ,\|R_n\|^2)$. As for $E_{\alpha}^n$ and $\calE_{\alpha}^n$, we introduce $\calB_{\alpha}^n$ the centered ball with squared radius $\beta_{n,\alpha}$, so that $B_{\alpha}^n = \{\widehat{Y}_{n+1}\} + \calB_{\alpha}^n$. We begin with describing the asymptotic behaviour of $\calB_{\alpha}^n$. 
\begin{proposition}\label{prop:asymptotic_ball}
Denote $\beta_{n,\alpha}$ the $\nath$ order statistic of $(\|R_1\|^2, \dots , \|R_n\|^2)$, which are assumed iid, and assume that the quantile function of $\|R_1\|^2$ is continuous on a neighbourhood of $1-\alpha$. Then
\begin{align}\label{eq:cv_quantile_iid}
    \beta_{n,\alpha} \xlongrightarrow[]{a.s.} q_{1-\alpha}(\|R_1\|^2).
\end{align}
As a result, the asymptotic volume of the associated ball is deterministic, and given by
\begin{align}\label{eq:vol_ball}
    \mathrm{Vol}(\calB_{\alpha}^n) &\xlongrightarrow[]{a.s.} v_\ell q_{1-\alpha}(\|R_1\|^2)^{\ell/2}.
\end{align}
\end{proposition}
Equation \eqref{eq:cv_quantile_iid} is quite intuitive and natural, though we did not manage to find it stated as such in standard textbooks. We thus provide a proof in the appendix\footnote{This proof is a copy of that of Théorème 8.9, p. 90 of the lecture notes \cite{agreg_quantiles} (in French).}. Similarly to $\calE_{\alpha}^{\infty}$, we define $\calB_{\alpha}^{\infty}$ to be the deterministic ball centered at $0$ and with radius $q_{1-\alpha}(\|R_1\|)$.
Note that, even in the case of Gaussian residuals, $q_{1-\alpha}(\|R_1\|^2)$ cannot be expected to be further simplified, as $\|R_1\|^2$ would follow a generalized chi-squared distribution.

\paragraph*{Elliptical distributions}
We can now compare the volumes $\Vol(\calE_{\alpha}^{\infty})$ and $\Vol(\calB_{\alpha}^{\infty})$. For this, our main assumption is that the vector $V_1$ follows an absolutely continuous elliptical distribution. This means that there exists a nonnegative function $g : \mathbb{R}_+ \rightarrow\mathbb{R}_+$ such that the density $f$ of $V_1$ is of the form 
\begin{align}\label{eq:ellip_assumption}
    f(v) \propto g((v-\mu)^\top\mathbf{\Sigma}^{-1}(v-\mu)),
\end{align}
for some $\mu\in\mathbb{R}^p$ and $\mathbf{\Sigma}\in\calM_{p,p}, \ \mathbf{\Sigma} \succ 0$ (\cite{muirhead1982}, Section 1.5). \iain{Above, $\mu, \ \mathbf{\Sigma}$ and $g$ are respectively the location parameter, the dispersion matrix and the generator of the distribution \eqref{eq:ellip_assumption} ($\mathbf{\Sigma}$ is defined up to a multiplicative constant, which can be absorbed in $g$).} Elliptical distributions include multivariate Gaussian and Cauchy distributions. Under the assumptions of Proposition \ref{prop:asymptotics}, $\mathbf{\Sigma} = \Cov(V_1)$ up to a deterministic constant (\cite{muirhead1982}, p. 34), \iain{although elliptical distributions do not necessarily have finite second order moments}. As argued in Appendix \ref{app:infinite_var}, this constant is irrelevant to define $\calE_\alpha^\infty$ (in particular its volume) in an unique way, so that in the rest of this section, we can assume that $\mathbf{\Sigma} = \Cov(V_1)$. This ellipticity assumption is equivalent to the existence of a random vector $T = (T_1, \dots , T_{p})^\top\in\mathbb{R}^{p}$ with an absolutely continuous spherical distribution, i.e. with a density of the form $f(x) = g(\|x\|^2)$, such that $V_1 = \mu + \mathbf{\Sigma}^{1/2}T$.

We first consider the case where $k = 0$ and $\lambda = 0$. In this case, $\mathbf{\Sigma}_{\lambda} = \mathbf{\Sigma} = \mathbf{\Sigma}^{22}$, \iain{and $\calE_\alpha^\infty = \calF_\alpha^\infty$}. Furthermore, denoting $\lambda_1, \dots , \lambda_{\ell}$ the eigenvalues of $\mathbf{\Sigma}$, we have $\det(\mathbf{\Sigma}) = \lambda_1 \dots  \lambda_{\ell}$. Under such assumptions, $V_\rmc^\top\mathbf{\Sigma}^{-1}V_\rmc = \sum_{i=1}^{\ell}T_i^2$ and $\|R_1\|^2$ is equal to $\sum_{i=1}^{\ell}\lambda_i(T_i+s_i)^2$ in distribution, for some $s = (s_1,\ldots,s_\ell)^\top\in\mathbb{R}^\ell$ (see the upcoming Lemma \ref{lemma:distrib_norm_R1} for a more general result). In particular, 
\begin{align}
     \bigg(\frac{\Vol(\calE_{\alpha}^{\infty})}{\Vol(\calB_{\alpha}^{\infty})}\bigg)^{2/\ell} &= \det(\mathbf{\Sigma})^{1/\ell}\frac{q_{1-\alpha}(\sum_{i=1}^{\ell} T_i^2)}{q_{1-\alpha}(\sum_{i=1}^{\ell}\lambda_i (T_i+s_i)^2)} \nonumber\\
      &= \frac{q_{1-\alpha}(\sum_{i=1}^{\ell} T_i^2)}{{(\lambda_1 \cdots  \lambda_{\ell})^{-1/\ell}}q_{1-\alpha}(\sum_{i=1}^{\ell}\lambda_i (T_i+s_i)^2)} = \frac{q_{1-\alpha}(\sum_{i=1}^{\ell} T_i^2)}{q_{1-\alpha}(\sum_{i=1}^{\ell}\delta_i (T_i+s_i)^2)},\label{eq:case_k_0_ratio_vol}
\end{align}
where $\delta_i = \lambda_i/(\lambda_1 \cdots  \lambda_{\ell})^{1/\ell}>0$ verify $\delta_1\cdots\delta_{\ell} = 1$. The comparison of $\Vol(\calE_{\alpha}^{\infty})$ and $\Vol(\calB_{\alpha}^{\infty})$ is then settled by the next proposition.
\begin{proposition}\label{prop:barthe}
Let $t>0$ and $(T_1, \dots , T_{\ell})$ be a random vector with an absolutely continuous spherical distribution, i.e. with a density of the form $f(x) = g(\|x\|^2),\ x \in \mathbb{R}^\ell$; assume furthermore that $g$ is non-increasing. For a fixed $t$, introduce $F_t(s, \delta) \coloneqq \mathbb{P}(\sum_{i=1}^{\ell} \delta_i (T_i+s_i)^2 \leq t)$, where $s = (s_1, \ldots, s_\ell)^\top\in\mathbb{R}^\ell$ and $\delta = (\delta_1, \ldots, \delta_\ell)^\top\in\mathbb{R}_+^\ell$. Then
\begin{align}\label{eq:argmin_barthe}
    ((0, \ldots, 0), (1,\dots, 1)) \in \underset{\substack{(s,\delta)\in\mathbb{R}^\ell\times \mathbb{R}_+^\ell\\ \prod_i \delta_i=1 }}{\arg\max} F_t(s,\delta)\,.
\end{align}
In particular, when $\lambda = 0$, $k = 0$ and under the ellipticity assumption \eqref{eq:ellip_assumption}, $\Vol(\calE_{\alpha}^{\infty}) \leq \Vol(\calB_{\alpha}^{\infty})$.
\end{proposition}
The property that the generator $g$ is nonincreasing is an unimodality assumption (see \cite{brandwein1978}, Definition 3.1.1). Proposition \ref{prop:barthe} relies on the observation that centered balls correspond to level sets of spherical distributions. Note that no moment assumptions are required for equation \eqref{eq:argmin_barthe} to hold.

We now consider the case where $k\geq 1$, still with $\lambda = 0$. \iain{We start with $\calE_\alpha^\infty$.}
Introduce a random vector $T = (T_1, \dots , T_{p})^\top$ with a spherical distribution, such that $V_1 = \mu + \mathbf{\Sigma}^{1/2}T$. To compare the volumes of $\calE_\alpha^\infty$ and $\calB_\alpha^\infty$, we will need the following lemma which describes the distribution of $\|R_1\|^2$ in terms of that of $T$.
\begin{lemma}\label{lemma:distrib_norm_R1} Assume that $T\in\mathbb{R}^p$ has a spherical distribution and $V_1 = \mu + \mathbf{\Sigma}^{1/2}T$ for some $\mu \in\mathbb{R}^{p}$ and $\mathbf{\Sigma} \in \calM_{p,p}, \ \mathbf{\Sigma} \succ 0$. Denote $(\lambda_1, \ldots, \lambda_\ell)$ the eigenvalues of $\mathbf{\Sigma}^{22}$. Then, there exists $s\in\mathbb{R}^\ell$ such that we have the equality in distribution
\begin{align}
    \|R_1\|^2 \stackrel{\rmd}{=} \sum_{i=1}^\ell \lambda_i(T_i+s_i)^2.
\end{align}
\end{lemma}
Above, the vector $s$ is given by $s= \bfD^{-1/2}\bfP^\top\mu$, where $\mathbf{\Sigma}^{22} = \bfP\bfD\bfP^\top$ is an eigendecomposition of $\mathbf{\Sigma}^{22}$. To compare $\Vol(\calE_\alpha^\infty)$ and $\Vol(\calB_\alpha^\infty)$, we will mimic the case $k=0$ (see equation \eqref{eq:case_k_0_ratio_vol}). To do so, we will require the bound
\begin{align}\label{eq:part_pos_control}
   0 \leq  (q_{1-\alpha}^{\calE} - X_\rmc^\top (\mathbf{\Sigma}^{11})^{-1}X_\rmc)_+ \leq q_{1-\alpha}^{\calE} = q_{1-\alpha}\bigg(\sum_{i=1}^{k+\ell}T_i^2\bigg)
\end{align}
as well as the ratio $c_{\alpha}(k,\ell)$ defined as
\begin{align}
     c_{\alpha}(k,\ell) &\coloneqq \frac{q_{1-\alpha}\left(\sum_{i=1}^{k+\ell}T_i^2\right)}{q_{1-\alpha}\left(\sum_{i=1}^{\ell}T_{i}^2\right)}>1. 
\end{align}
From these two equations as well as Proposition \ref{prop:barthe} and Lemma \ref{lemma:distrib_norm_R1}, we are able to prove Proposition \ref{prop:gen_k} below.
\begin{proposition}\label{prop:gen_k} Denote the dependence of $\calE_\alpha^\infty$ on $\lambda$ as $\calE_{\alpha,\lambda}^\infty$. Under the assumptions of Proposition \ref{prop:asymptotics}, if
\begin{align}\label{eq:tradeoff}
c_{\alpha}(k,\ell)\bigg(\frac{\det(\mathbf{\Sigma}/\mathbf{\Sigma}^{11})}{\det(\mathbf{\Sigma}^{22})}\bigg)^{1/\ell} \leq 1,    
\end{align}
then $\Vol(\calE_{\alpha,0}^{\infty}) \leq \Vol(\calB_{\alpha}^{\infty})$. Moreover, if the inequality in equation \eqref{eq:tradeoff} is strict, then there exists $\lambda_0>0$ such that for all $\lambda\in[0,\lambda_0]$, $\Vol(\calE_{\alpha,\lambda}^{\infty}) \leq \Vol(\calB_{\alpha}^{\infty})$. Finally, $\lambda_0$ can be chosen as the unique solution to the algebraic equation 
\begin{align}
\det(\mathbf{\Sigma}_{\lambda_0}) = c_{\alpha}(k,\ell)^{-\ell}\det(\mathbf{\Sigma}_{\lambda_0}^{11})\det(\mathbf{\Sigma}^{22})\,.
\end{align}
\end{proposition}
The condition \eqref{eq:tradeoff} can be relaxed to having the first part of equation \eqref{eq:cond_vol_gen} smaller than one. The algebraic equation for finding $\lambda_0$ is correspondingly changed as
\begin{align}
    \det(\mathbf{\Sigma}_{\lambda_0}) = \Bigg(\frac{q_{1-\alpha}(\sum_{i=1}^{k+\ell} T_i^2)}{q_{1-\alpha}(\sum_{i=1}^{\ell}\delta_i (T_{i}+s_i)^2)}\Bigg)^{-\ell}\det(\mathbf{\Sigma}_{\lambda_0}^{11})\det(\mathbf{\Sigma}^{22}).\label{eq:lambda_not_tight}
\end{align}
The condition \eqref{eq:tradeoff} enables to identify two distinctive features of the data, when aiming at optimizing the volume of $\calE_\alpha^\infty$: the intrinsic distribution of the $T_i$ and the values of $k$ and $\ell$ on the one hand, and the structure of $\mathbf{\Sigma}$ on the other hand.
In particular it shows that in order to minimize the volume of $\calE_{\alpha}^{\infty}$, we are  faced with a tradeoff when using many explanatory variables. As $k$ increases, the determinant of the Schur complement decreases \iain{(this can be seen from the block inversion lemma)}, but the ratio $c_{\alpha}(k,\ell)$ increases at the same time. 
This tradeoff is illustrated in Section \ref{sub:num_gaussian} (Table \ref{tab:optimal_k}).
\begin{remark}[Optimal choice of explanatory variables]\label{rk:optimal_k} In view of this tradeoff, a natural question is that of the selection of the best variables within $X_{1}$ to explain the response variable $R_{1}$, especially if $k$ is large. For example, if $R_{1}$ is independent of $(X_{1})_1$, the first coordinate of $X_{1}$, then taking it into account will not decrease the determinant of the Schur complement, as (when $\lambda = 0$) it is equal to the conditional covariance $\Cov(R_1|X_1)$. It will, however, increase the ratio $c_{\alpha}(k,\ell)$, thus increasing the volume of $\calE_{\alpha}^{\infty}$. Likewise, if $(X_{1})_1 = (X_{1})_2$, then only one of those random variables is required to explain $R_{1}$.
In our context, this question boils down to that of the research of active subspaces, which is e.g. studied in \cite{active_sbspace}. The use of active subspace methods in our framework is left for future work.
\end{remark}

To deal with the case $k\geq 1$, we used the bound \eqref{eq:part_pos_control}, which may seem very conservative. This is not entirely true, in the following sense. Setting $S = X_\rmc^\top (\mathbf{\Sigma}^{11})^{-1}X_\rmc$, the expected volume $\mathbb{E}[\text{Vol}(\calE_{\alpha}^{\infty})]$ is given by $\mathbb{E}[(q_{1-\alpha}^{\calE} - S)_+^{\ell/2}] = (q_{1-\alpha}^{\calE})^{\ell/2}\mathbb{E}[(1 - S/q_{1-\alpha}^{\calE})_+^{\ell/2}]$, up to a deterministic constant. Assuming that $q_{1-\alpha}^{\calE} \rightarrow + \infty$ when $\alpha \rightarrow 0^+$, we then deduce from the bound $0\leq (1-S/q)_+\leq 1 \ (q>0)$ and the dominated convergence theorem that 
\begin{align}
    \mathbb{E}[(q_{1-\alpha}^{\calE} - X_\rmc^\top (\mathbf{\Sigma}^{11})^{-1}X_\rmc)_+^{\ell/2}] \underset{\alpha \rightarrow 0^+}{\sim} (q_{1-\alpha}^{\calE})^{\ell/2}.\label{eq:hardy_littlewood}
\end{align}
In particular, the right-hand side in equation \eqref{eq:hardy_littlewood} corresponds to the bound one would obtain using the conservative bound \eqref{eq:part_pos_control}.

\begin{remark}[Conditional coverage for elliptical distributions]
If the distribution of $V_1$ is elliptical, it is possible to obtain the conditional distribution $R_1|X_1$ from the observation of $\beta \mapsto q_\beta(V_\rmc^\top\boldsymbol{\Sigma}^{-1}V_\rmc)$, as we can deduce the spherical density $g$ appearing in equation \eqref{eq:ellip_assumption} from this quantile map. For all $x$, using the definition of the conditional density and the knowledge (or estimates) of $g, \mu$ and $\boldsymbol{\Sigma}$, one can do derive the conditional distribution $R|X=x$, and $q_{1-\alpha}(V_\rmc^\top\boldsymbol{\Sigma}^{-1}V_\rmc|X_1=x)$ in particular, in integral/implicit form. As remarked in Section \ref{ssec:asymp_ellips}, this conditional quantile is the required quantity to recover conditional coverage using our methodology. The use of this observation in a practical setting is left for future work.
\end{remark}
\iain{Contrarily to $\calE_\alpha^\infty$, the case $k\geq 1, \lambda = 0$ is easily dealt with for $\calF_\alpha^\infty$.
\begin{proposition}\label{prop:vol_fna} Set $\lambda = 0$, and assume that $V_1$ follows and elliptical distribution with a non-increasing generator $g : \mathbb{R}^+\rightarrow \mathbb{R}^+$. Then we have, almost surely,
\begin{align}
    \Vol(\calF_\alpha^\infty) \leq \Vol(\calB_\alpha^\infty).
\end{align}
\end{proposition}
The key to Proposition \ref{prop:vol_fna}  is the observation that $T_\rmc^\top\boldcalA_\infty^{-1}T_\rmc \stackrel{\rmd}{=} \sum_{i=1}^\ell T_i^2$. This fact can be used together with Lemma \ref{lemma:distrib_norm_R1} and Proposition \ref{prop:barthe} to obtain the desired result. As opposed to Proposition \ref{prop:gen_k}, the dependence on $k$ of the asymptotic volume of $\calF_\alpha^\infty$ is milder than that of $\calE_\alpha^\infty$, when $V_1$ is elliptical. In particular, minimizing $\Vol(\calF_\alpha^\infty)$ as a function of $k$ is equivalent to minimizing the determinant
\begin{align}
    \det(\mathbf{\Sigma}/\mathbf{\Sigma}^{11}) = \det(\mathbf{\Sigma}^{22} - \mathbf{\Sigma}^{21}(\mathbf{\Sigma}^{11})^{-1}\mathbf{\Sigma}^{12}),
\end{align}
which can be shown to be a decreasing function of $k$ using the block-inversion lemma. Thus, one can only hope the this quantity becomes stationary for large $k$, and large values of $k$ (a form of complexity) are not penalised, contrarily to $\calE_\alpha^\infty$.
}
\paragraph*{Non elliptical distributions} In this paragraph, we assume that $k=0$ and $\lambda = 0$, so that $\mathbf{\Sigma}_{\lambda} = \mathbf{\Sigma} = \mathbf{\Sigma}^{22} = \Cov(R_1)$ \iain{and $\calE_\alpha^\infty = \calF_\alpha^\infty$}. Without ellipticity assumptions, the comparison of $\Vol(\calB_{\alpha}^{\infty})$ and $\Vol(\calE_{\alpha}^{\infty})$ can still be studied for small $\alpha$ by comparing the tail behaviour of the distributions of $\|\mathbf{\Sigma}^{-1/2}(R_1-\mathbb{E}[R_1])\|^2$ and $\|R_1\|^2/\det(\mathbf{\Sigma})^{1/\ell}$. Indeed, if
\begin{align}\label{eq:good_score}
    \mathbb{P}(\|\mathbf{\Sigma}^{-1/2}(R_1-\mathbb{E}[R_1])\|^2>t) \leq \mathbb{P}(\|R_1\|^2/\det(\mathbf{\Sigma})^{1/\ell} >t) \ \ \ \text{for large $t$},
\end{align}
then we can show that $\Vol(\calE_{\alpha}^{\infty})\leq\Vol(\calB_{\alpha}^{\infty})$ for $\alpha$ small enough (e.g. adapt the proof of Proposition \ref{prop:contre_exemple}, equations \eqref{eq:compare_logs} and \eqref{eq:compare_quantiles_mimica}). In this context, the precise decay of the tail distribution of nonnegative random variables is studied in \cite{mimica}, when the decay is either exponential or polynomial. More precisely, \cite{mimica} proves Tauberian theorems, which link this decay with properties of the Laplace transform of the said distribution. The full use of the results of \cite{mimica} in our context is left for future work. For the time being we simply observe that, without ellipticity assumptions, it is unreasonable to expect that a scoring rule solely based on a covariance analysis of the residuals will always yield smaller confidence regions than those of the norm residual score. In that sense, the next proposition provides a counterexample for which our method underperforms when compared to the norm residual score. This proposition relies on \cite{mimica}, Theorem 1.2.
\begin{proposition}\label{prop:contre_exemple}
Let $\lambda_1,  \ldots,\lambda_{\ell}>0$ be such that there exists $i\neq j$ such that $\lambda_i\neq\lambda_j$, and define $\delta_i \coloneqq \lambda_i/(\lambda_1 \cdots \lambda_{\ell})^{1/\ell}$. Then there exists a random vector $T\in\mathbb{R}^{\ell}$ with $\mathbb{E}[T]=0$ and $\Cov(T) = \mathbf{I}_{\ell}$  such that for large $t$,
\begin{align}
    \mathbb{P}\bigg(\sum_{i=1}^{\ell}T_i^2>t\bigg) > \mathbb{P}\bigg(\sum_{i=1}^{\ell}\delta_iT_i^2 >t\bigg).    
\end{align}
In particular, if $R_1 = \bfD T$ where $\bfD$ is diagonal such that $\bfD_{ii} = \lambda_i^{1/2}$, then $\mathbb{E}[R_1] = 0$, $\mathbf{\Sigma} = \Cov(R_1) = \bfD^2$ and for large $t$,
\begin{align}\label{eq:bad_score}
    \mathbb{P}(\|\mathbf{\Sigma}^{-1/2}R_1\|^2>t) > \mathbb{P}(\|R_1\|^2/\det(\mathbf{\Sigma})^{1/\ell} >t).
\end{align}
Thus, if $k=0$ and $\lambda = 0$, then for $\alpha$ small enough , $\Vol(\calB_{\alpha}^{\infty}) < \Vol(\calE_{\alpha}^{\infty})$.
\end{proposition}
We see that the inequality \eqref{eq:bad_score} is reversed when compared to \eqref{eq:good_score}.
The vector $T$ above is built using independent gamma distributions : $T_i^2\sim\gamma(\delta_i,\delta_i^{-1})$. The proof of Proposition \ref{prop:contre_exemple} relies on the fact that, despite $T_i$ being centered with unit variance, the tails of $T_i^2$ do not decay at the same rate (in particular, the $T_i$ are not identically distributed). Note that in general, such a property cannot hold if the $T_i$ are built from a one dimensional family of distributions, as fixing the variance would determine the decay rate. This is the reason why we used general gamma distributions $\gamma(k,\theta)$. In the case $k\geq 1$, \iain{the same counterexample where $\ell$ is replaced by $p$ works for $\calF_\alpha^\infty$. Indeed, from the independence of the $T_i$, $\mathbf{\Sigma}$ is diagonal, hence $Z_0^\infty = 0$ a.s.,  $\mathbf{\Sigma}/\mathbf{\Sigma}^{11} = \mathbf{\Sigma}^{22}$ and the conclusion of Proposition \ref{prop:contre_exemple} applies}. Counter-examples in expectation \iain{with $\calE_\alpha^\infty$} are easily built for $\alpha \rightarrow 0^+$ using centered spherical distributions ($\mu = 0, \ \boldsymbol{\Sigma} = \boldsymbol{I}_p$). Indeed, adapting equation \eqref{eq:first_tradeoff} applied to such distributions and using equation \eqref{eq:hardy_littlewood}, we have  $\mathbb{E}[\Vol(\calE_{\alpha}^{\infty})]/\Vol(\calB_{\alpha}^{\infty}) \sim_{\alpha\rightarrow 0^+} c_\alpha(k,\ell)^{\ell/2} >1$.



\section{Numerical application} \label{sec:numexp}

The code for generating both data and figures is given in
    \url{https://github.com/iain-pl-henderson/PyConfCov}.
The workflow is as follow. We first train the predictor on $n_{\nsplit}$ data points $(X_i, Y_i)$. Next, for each of the $n_{\test}$ test points, we generate $n_{\calib}$ calibration points and perform CCLE. By averaging over those $n_{\test}$ experiments, we obtain a realisation of the empirical average volume and coverage of the resulting ellipsoids, denoted by $\widehat{\mathbb{E}}_{n_{\test}}[\Vol(\calE_{\alpha}^{n_{\calib}})]$ and $1-\widehat{\mathbb{E}}_{n_{\test}}[\alpha_{\calE_{\alpha}^{n_\calib}}]$ (replace $\calE_\alpha^{n_\calib}$ with $\calF_\alpha^{n_\calib}$ if using $\calF_\alpha^{n_\calib}$). We use the same notations for the balls of the norm residual score, replacing $\calE$ with $\calB$ in those notations. We perform this complete procedure (except for the training of the predictor) $n_{\histo}$ times, to obtain a histogram of the distributions of $\widehat{\mathbb{E}}_{n_{\test}}[\Vol(\calE_{\alpha}^{n_{\calib}})]$ and $1-\widehat{\mathbb{E}}_{n_{\test}}[\alpha_{\calE_{\alpha}^{n_\calib}}]$. We denote the empirical means of those histograms as $\widehat{\mathbb{E}}_{n_{\histo}}\widehat{\mathbb{E}}_{n_{\test}}[\Vol(\calE_{\alpha}^{n_{\calib}})]$, and so forth.
\subsection{Gaussian data}\label{sub:num_gaussian}
We consider the case of iid observations of Gaussian random vectors $U_1, \dots  , U_{n+1},\ U_i\in\Rbb^{p}$, where $k = 6,\ \ell = 3$ and $p = k+ \ell = 9$, following a $\calN(0,\mathbf{\Sigma})$ distribution. Here, we choose $\mathbf{\Sigma}$ such that $\mathbf{\Sigma}_{i,j} = k_{3/2}(i-j)$, where $k_{\nu}$ is a Matérn covariance function of order $\nu$, with variance $\sigma^2=1$ and lengthscale $L = 5$ (see equation \eqref{eq:matern-3/2}). Such a Gaussian vector is a sampling of a centered one dimensional Gaussian process $(U_t)_{t > 0}$ at integer grid points $\{1, \dots , 9\}$, where $(U_t)_{t>0}$ has the said $3/2-$Matérn covariance function. The considered problem is thus that of the \iain{simultaneous} conformal inference for the $\ell = 3$ next time steps, given the previous $k = 6$ time steps. In the numerical experiments, we will represent the results $(U_{7}, U_{8}, U_{9})$ in $(x,y,z)$ coordinates.

The $3/2$-Matérn Gaussian process $(U_t)_{t > 0}$ is a continuous $\AR(2)$ process, meaning that it is a solution to a second order SDE \cite{gpml}, Appendix B.2. As a consequence, the knowledge of $U_{t-\dt}$ and $U_t$ determines the value of $U_{t+\dt}$ up to a noise of variance of order $O(\dt)$. If the lengthscale $L$ is large, then the vector $(U_1, \ldots, U_9)$ can be understood as a sampling of the Gaussian vector $(U_t)$ with a small sampling period $T_{\mathrm{sample}}$ compared to the characteristic length of the process ($T_{\mathrm{sample}} = 1 \ll L$). As such, using our covariance based score, we expect that for large lengthscales $L$, $U_{7}$ ($x$ coordinate) will hold low uncertainty, $U_{8}$ ($y$ coordinate) medium uncertainty and $U_{9}$ ($z$ coordinate) will hold most of the uncertainty.

The predictor $\widehat{f}$ is obtained by ridge regression trained on previous split data, that is $\widehat{Y}_i = \widehat{\beta}^\top X_i$ with $\widehat{\beta} = (X_{\mathrm{s}}^\top X_{\mathrm{s}} + \mu_0 I)^{-1}X_{\mathrm{s}}^\top Y_{\mathrm{s}}$, where $(X_{s} \ Y_{\mathrm{s}})\in\calM_{n_\nsplit,p}$ denotes the training data, stored row-wise. Following the split conformal inference framework, the matrix $\widehat{\beta}$ will be considered deterministic in the numerical experiments.
All the results correspond to $\alpha = 0.1$. 

\paragraph*{Comparison with the norm residual score}

\begin{figure}[ht!]\centering
\begin{subfigure}[b]{0.2\textwidth}
    \includegraphics[width=\textwidth]{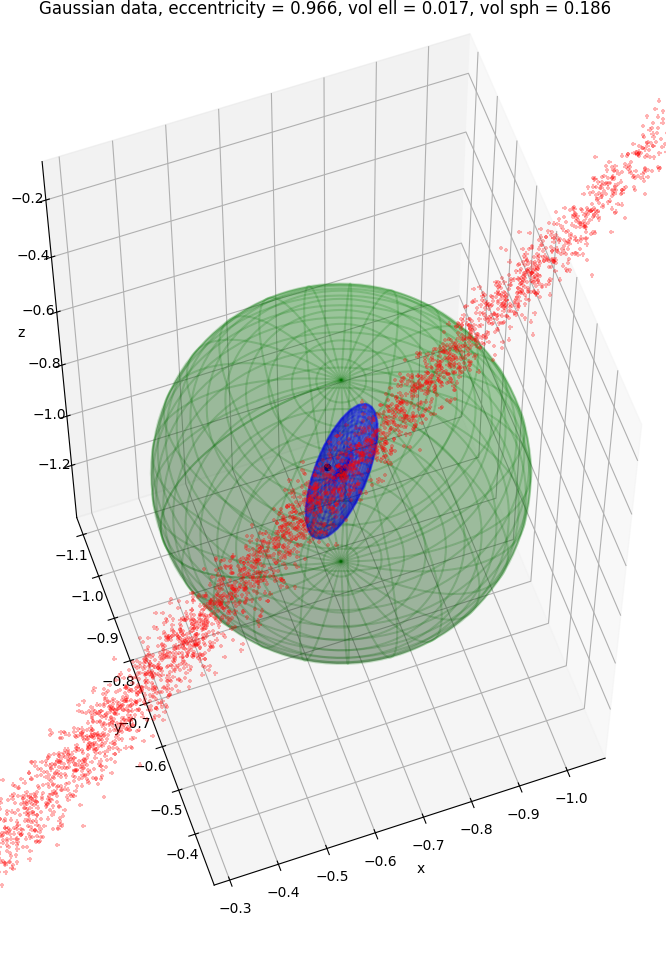}
    \caption{Top view}
\end{subfigure}\hspace{2cm}
\begin{subfigure}[b]{0.2\textwidth}
    \includegraphics[width=\textwidth]{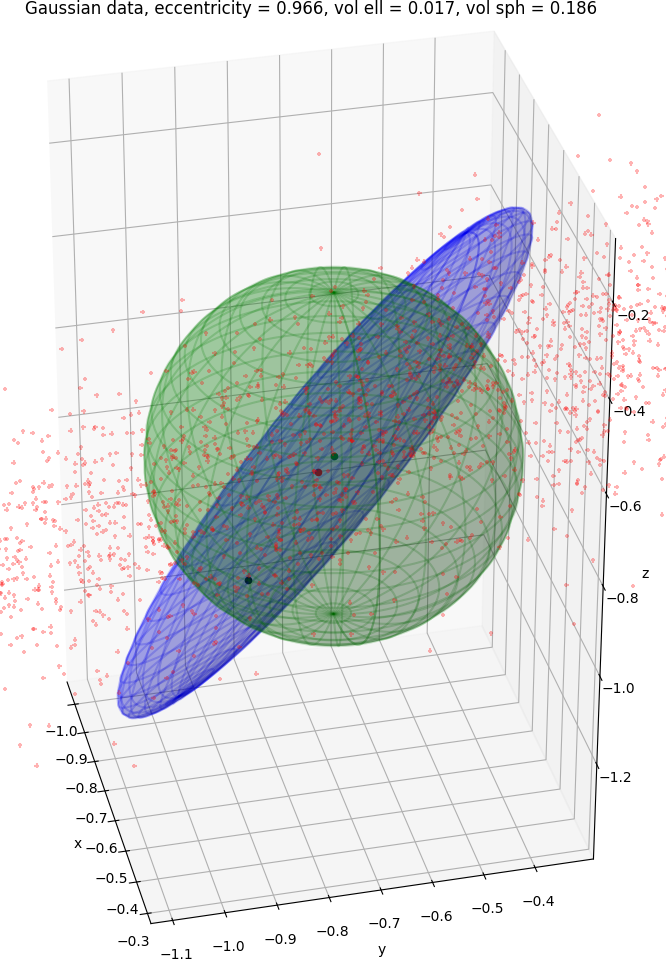}
    \caption{Side view (same example)}
\end{subfigure}
    \caption{Sample example of $\calE_\alpha^n$, Gaussian data. Black dot: real value position. Large red dot: center of ellipsoid. Green dot: predictor $\widehat{Y}_{n+1}$ (also the center of the sphere). Small red dots: other residuals ($n = 2000$).}
    \label{fig:gaussian}
\end{figure}
The histograms in Figure \ref{fig:histo_gaussian} correspond to the Gaussian data described above. They correspond to $n_{\nsplit} = 5000, \ n_{\calib} = 200, \ n_{\test} = 100$ and $n_{\histo} = 2000$.
We see that $\widehat{\mathbb{E}}_{n_{\histo}}\widehat{\mathbb{E}}_{n_{\test}}[\Vol(\calE_{\alpha}^{n_{\calib}})] = 1.54$ \iain{and $\widehat{\mathbb{E}}_{n_{\histo}}\widehat{\mathbb{E}}_{n_{\test}}[\Vol(\calF_{\alpha}^{n_{\calib}})] = 0.895$} while we have $\widehat{\mathbb{E}}_{n_{\histo}}\widehat{\mathbb{E}}_{n_{\test}}[\Vol(\calB_{\alpha}^{n_{\calib}})] = 9.35$, meaning that the volume of the confidence regions is, on average, divided by $6$ using the score $S(z)$ \iain{and divided by $10$ using the score $S'(z)$}. We also see that our method slightly overcovers (top left and middle histograms): the average coverage for \iain{$\calE_{\alpha}^{n_{\calib}}$ and $\calF_{\alpha}^{n_{\calib}}$} is $1-\widehat{\mathbb{E}}_{n_{\test}}[\alpha_{\calE_{\alpha}^{n_\calib}}] \simeq 1-\widehat{\mathbb{E}}_{n_{\test}}[\alpha_{\calF_{\alpha}^{n_\calib}}] \simeq 0.905$, while it is $1-\widehat{\mathbb{E}}_{n_{\test}}[\alpha_{\calB_{\alpha}^{n_\calib}}] = 0.9004$ for the sphere, for a target of $0.9$. This is expected from the construction of \iain{$\calE_{\alpha}^{n_{\calib}}$ and $\calF_{\alpha}^{n_{\calib}}$}, since $n_{\calib} = 200$ is far from the asymptotic regime (see Lemma \ref{lemma:approx_score} and Proposition \ref{prop:bound}). Still, concerning the volume of $\calF_\alpha^n$ (middle column of Figure \ref{fig:histo_gaussian}, bottom row), we can see that the convergence of $\Vol(\calF_\alpha^n)$ to its limit  $\Vol(\calF_\alpha^\infty)$ is already quite advanced, from the width of the associated histogram. 

A random realisation of $\calE_\alpha^n$ is given in Figure \ref{fig:gaussian} \iain{(the corresponding  $\calF_\alpha^n$ has the same appearance, the only difference with $\calE_\alpha^n$ being the squared radius $\rho_{n,\alpha}$)}. This figure shows that that the calibration residuals (red dots) are not distributed along the blue ellipsoid. This is also expected : the calibration residuals have a covariance $\Cov(R_1) = \mathbf{\Sigma}^{22}$, while the matrix associated to $\calE_{\alpha}^{\infty}$ is $\Cov(R_1|X_1) = \mathbf{\Sigma}_{\lambda}^{22} - \mathbf{\Sigma}^{21}(\mathbf{\Sigma}_{\lambda}^{11})^{-1}\mathbf{\Sigma}^{12}$. 
\begin{figure}[ht!]
\centering
\begin{subfigure}[b]{\textwidth}
    \centering
\includegraphics[width=0.7\textwidth]{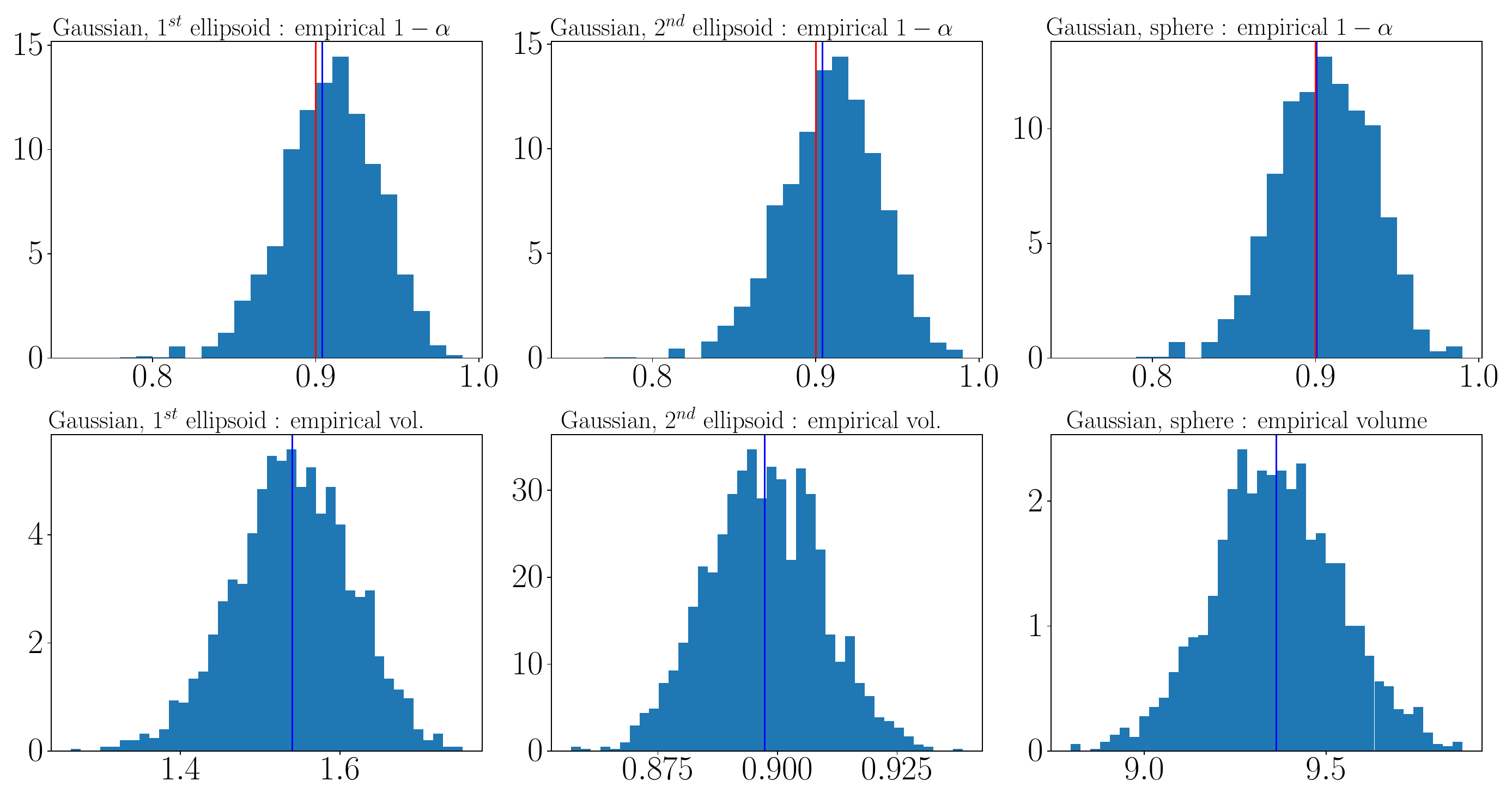}
    \caption{Histograms of empirical volumes and coverage ($1^\text{st}$ row : 25 bins; $2^\text{nd}$ row : 40 bins). Red lines : $x = 1-\alpha = 0.9$. Blue lines : mean of each histogram. $1^\text{st}$ row (left to right) : histograms of $1-\widehat{\mathbb{E}}_{n_{\test}}[\alpha_{\calE_{\alpha}^{n_\calib}}], 1-\widehat{\mathbb{E}}_{n_{\test}}[\alpha_{\calF_{\alpha}^{n_\calib}}], 1-\widehat{\mathbb{E}}_{n_{\test}}[\alpha_{\calB_{\alpha}^{n_\calib}}]$.  $2^\text{nd}$ row (left to right) : histograms of  $\widehat{\mathbb{E}}_{n_{\test}}[\Vol(\calE_{\alpha}^{n_{\calib}})], \widehat{\mathbb{E}}_{n_{\test}}[\Vol(\calF_{\alpha}^{n_{\calib}})], \widehat{\mathbb{E}}_{n_{\test}}[\Vol(\calB_{\alpha}^{n_{\calib}})]$.}
    \label{fig:histo_gaussian}
\end{subfigure}\\
\begin{subfigure}[b]{0.4\textwidth}
    \includegraphics[width=\textwidth]{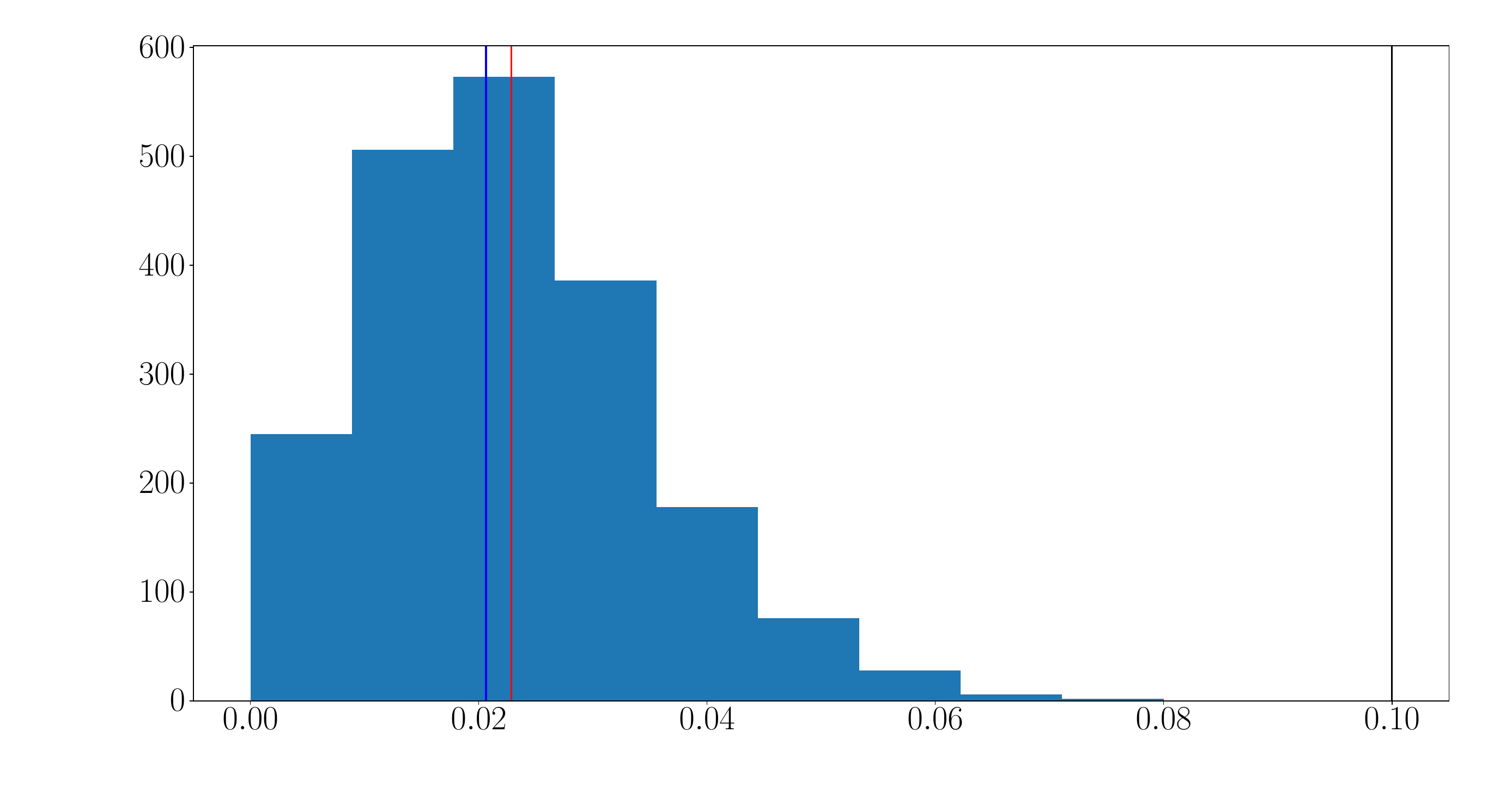}
    \caption{Histogram of the proportion of empty $\calE_\alpha^{n_\calib}$. Blue line : mean of the histogram. red line : asymptotic value given by eq. \eqref{eq:proba_empty_gauss}. Black line : $\alpha = 0.1$.}
    \label{fig:empty_gaussian}
\end{subfigure}\hspace{1.5cm}
\begin{subfigure}[b]{0.4\textwidth}
    \includegraphics[width=\textwidth]{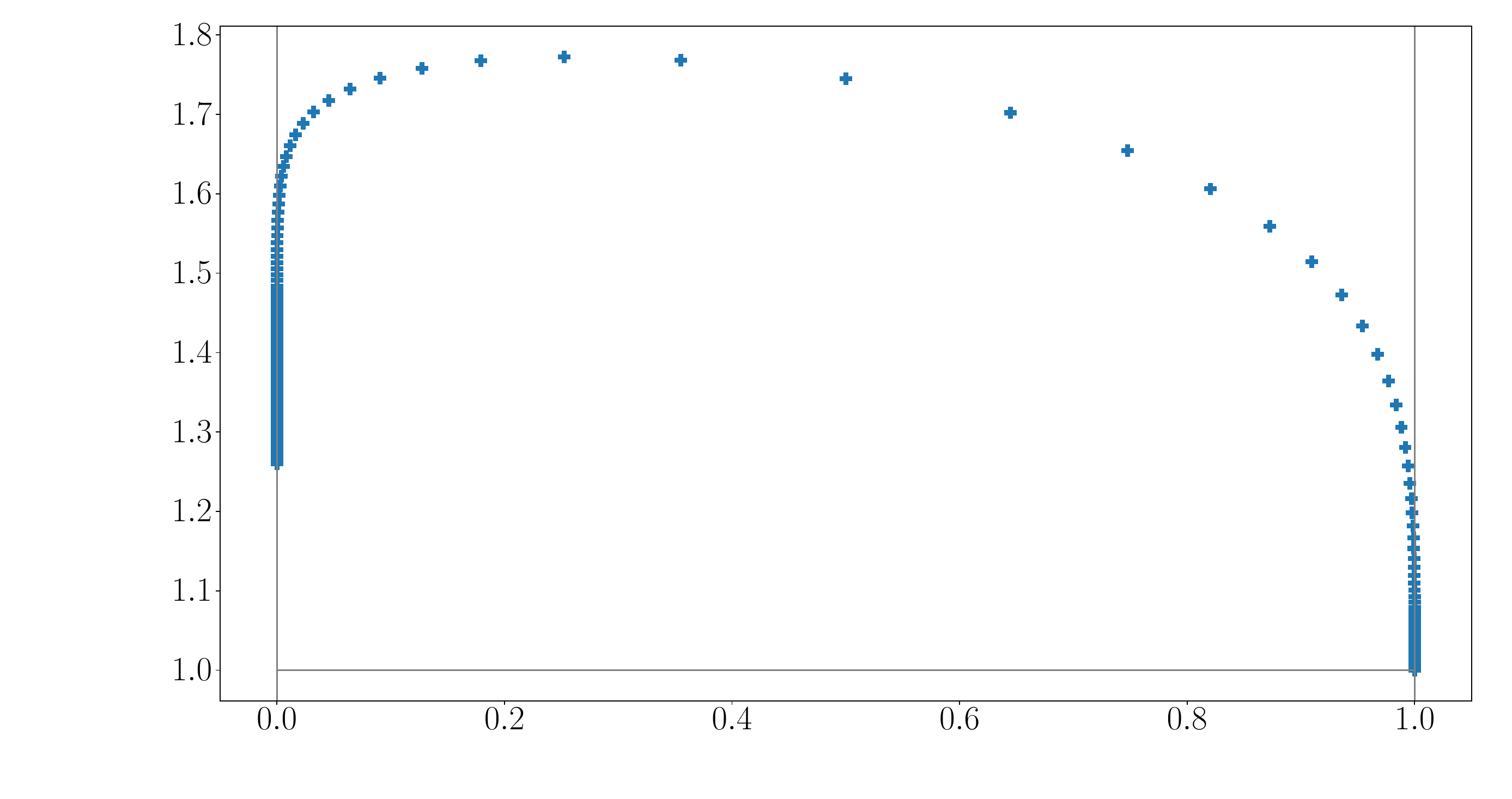}
    \caption{Plot of the ratio $\mathbb{E}[\Vol(\calE_\alpha^\infty)]/\Vol(\calF_\alpha^\infty)$, as a function of $\alpha \in [10^{-15},1-10^{-15}]$, for Gaussian data. This ratio is always greater than $1$.}
    \label{fig:ratio_vol_gaussian}
\end{subfigure}
\caption{Numerical results, Gaussian data with $\alpha = 0.1$, \ $n_{\nsplit} = 5000, \ n_{\calib} = 200, \ n_{\test} = 500$ and $n_{\histo} = 1000$.}
\end{figure}

\paragraph*{Study of the moments $\mathbb{E}[\Vol(\calE_{\alpha}^n)^q]$}
Table \ref{tab:expectations} provides an empirical validation of formula \eqref{eq:exp_vol_gauss}. These results were obtained on a centered $3/2$-Matérn Gaussian random vector with variance $1$ and lengthscale $L=5$. The sample sizes are $n_{\nsplit} = 5000,\ n_{\calib} = 50000,\ n_{\test} = 40,\ n_{\histo} = 1000$. The relative error corresponds to 
\begin{align}
    e_{\mathrm{rel}} =\big|\mathbb{E}[\Vol(\calE_{\alpha}^{\infty})^q] -\widehat{\mathbb{E}}_{n_{\histo}}\widehat{\mathbb{E}}_{n_{\test}}[\Vol(\calE_{\alpha}^{n_\calib})^q]\Big|\Big/{\mathbb{E}[\Vol(\calE_{\alpha}^{\infty})^q]}.
\end{align}

\begin{table}[b!]
    \centering\footnotesize
    \begin{tabular}{c|ccc}\hline
         $q$ & $\mathbb{E}[\Vol(\calE_{\alpha}^{\infty})^q]$ & $\widehat{\mathbb{E}}_{n_{\histo}}\widehat{\mathbb{E}}_{n_{\test}}[\Vol(\calE_{\alpha}^{n_\calib})^q]$ & Rel. err. $e_{\mathrm{rel}}$ \\ \hline
         1 &  1.4007 & 1.4039 & 2.3e-3\\
         2 &  2.4183 & 2.4295 & 4.6e-3\\
         3 &  4.5815 & 4.6096 & 6.1e-3\\ \hline
    \end{tabular}
    \caption{Empirical validation of formula \eqref{eq:exp_vol_gauss}.}
    \label{tab:expectations}
\end{table}

\paragraph*{Optimal number of input dimensions $k$} In this paragraph, we illustrate the tradeoff principle \iain{concerning $\calE_\alpha^n$}, as discussed following Proposition \ref{prop:gen_k}. we consider a variance $\sigma^2 = 1$ and a lengthscale $L = 5$. Following \cite{gpml}, Section 4.2, Gaussian processes with covariance function $k_{q-1/2}$ are $\AR(q)$ processes. As such, we expect the optimal value of $k$ to be $k_{\opt}=q$. Indeed, seeing the $\AR(q)$ relation as a linear recursive sequence of order $q$, one should at least use $(U_{6}, \ldots, U_{7-q})$ to predict $U_7$ accurately, and the minimal reasonable value of $k$ is thus equal to $q$. This fact is confirmed by Table \ref{tab:optimal_k}, which provides empirical estimates of $\mathbb{E}[\Vol(\calE_{\alpha}^n)]$. Do note that in numerical experiments, especially for $q=3$ and $q=4$, it may happen that the optimal value identified with such empirical estimates be $k_{\opt} = q+1$ and not $q$. Indeed, the corresponding values in the table below are 0.135 versus 0.142 ($q=3$) and 2.21e-02 versus 2.29e-02 ($q=4$). In fact it should be observed in this example that, after $k_{\opt}$ is reached, the average volume only slightly increases for $k>k_{\opt}$. On the contrary, until $k_{\opt}$ is reached, the average volume is potentially much poorer.

\begin{table}[b!]
    \centering\footnotesize
    \begin{tabular}{c|ccccccc}\hline
         \diagbox{Type}{$k$} &0 & 1& 2 & 3  & 4 & 5 & 6\\ \hline
$q = 1$ &22.9	&\textbf{15.2}	&17.1	&19.1	&20.5	&22.4	&23.8\\ \hline

$q = 2$ &4.96	&1.32	&\textbf{1.12}	&1.23	&1.33	&1.43	&1.53\\ \hline

$q = 3$ &2.31	&3.02e-01	&1.49e-01	&\textbf{1.35e-01}	&1.42e-01	&1.54e-01	&1.63e-01\\ \hline

$q = 4$ &1.61	&1.25e-01	&3.69e-02	&2.37e-02	&\textbf{2.21e-02}	&2.29e-02	&2.41e-02\\ \hline
    \end{tabular}
    \caption{Empirical volume of the confidence ellipsoid $\calE_\alpha^n$, for $\alpha = 0.1$, for Gaussian data.}
    \label{tab:optimal_k}
\end{table}

\begin{remark}
Even though CCLE is formulated in the standard regression framework, the presented example is essentially a time-series example. However, we do not address the typical issues encountered in conformal inference for time series, such as the balance between longitudinal and transversal coverage \cite{Linetal22a}. The analysis of our method in the time-series framework is the topic of an upcoming paper.
\end{remark}
\subsection{Cauchy data}\label{sub:num_cauchy}
Proposition \ref{prop:asymptotics} suggests that the confidence regions $\calE_{\alpha}^n$ may be expected to blow up in volume for heavy tailed distributions, as the limiting ellipsoid is expressed in terms of the covariance of the data, which is undefined in this case. On the contrary, since the limit sphere from the standard score $S_i' = \|Y^i-\widehat{Y}^i\|^2$ have a deterministic asymptotic volume, one could expect the standard score to beat our covariance based score in terms of volume. Here we consider a central multivariate Cauchy distribution, $U_1\sim C(0, \mathbf{\Sigma})$ \cite{kulik2020}, where $\mathbf{\Sigma}$ is the same as that of Section \ref{sub:num_gaussian}, i.e. $\mathbf{\Sigma}_{ij} = k_{3/2}(i-j), \ \sigma^2 = 1, \ L = 5$ ($\mathbf{\Sigma}$ cannot be interpreted as a covariance matrix anymore). In particular, $C(0,\mathbf{\Sigma})$ is an elliptical distribution. The numerical results are given in Figure \ref{fig:histo_cauchy}, which corresponds to $n_{\nsplit}=3\times 10^6, \ n_\calib = 10^4, \  n_\test = 800$ and $n_\histo = 10^3$ (we have used high values for $n_\nsplit$ and $n_\calib$, hoping that the corresponding numerical simulations may illustrate a form of convergence of our method for Cauchy data).
For these parameters, we see that our method still exhibits an empirical average volume of $\widehat{\mathbb{E}}_{n_{\histo}}\widehat{\mathbb{E}}_{n_{\test}}[\Vol(\calE_{\alpha}^{n_\calib})] \simeq 790$, whereas that of the residual score is $\widehat{\mathbb{E}}_{n_{\histo}}\widehat{\mathbb{E}}_{n_{\test}}[\Vol(\calB_{\alpha}^{n_\calib})] \simeq 1560$. Interestingly, these volumes also seem to exhibit a form of asymptotic normality, suggesting that $\calE_\alpha^n$ is indeed robust to distributions with infinite variance. 

It also seems that our method is able to recover the matrix $\mathbf{\Sigma}/\mathbf{\Sigma}^{11}$ (which dictates the dispersion of $R_1 - \mathbf{\Sigma}^{21}( \mathbf{\Sigma}^{21})^{-1}X_1$), as the shape and orientation of the ellipsoid for Cauchy data is similar to that of the ellipsoid for Gaussian data (Figures \ref{fig:gaussian} and \ref{fig:cauchy}). We refer to Appendix \ref{app:infinite_var} for further comments on this behaviour. Of course, the volumes for the Cauchy distribution are several orders of magnitude larger than that of the Gaussian data, but this is expected as Cauchy distributions are heavy tailed. 

\iain{Let us turn to $\calF_\alpha^n$ : because Cauchy data are heavy-tailed and Proposition \ref{prop:calF_full_space} assumes a finite $4q^{\mathrm{th}}$ moment, we expect that arbitrarily large ellipsoids occur even in the regime of large $n_\calib$ 
(for example we obtained 9 full space ellipsoids out of $200\times 500 = 10^5$ tests, for $n_\calib = 1000$). For heavy-tailed distributions, histograms are not easily interpreted because of the large disparity of the observed values. Instead we present in Table \ref{tab:cauchy_fna} with empirical quantiles of the distribution of $\widehat{\mathbb{E}}_{n_{\test}}[\Vol(\calF_{\alpha}^{n_\calib})]$ compared to those of $\widehat{\mathbb{E}}_{n_{\test}}[\Vol(\calB_{\alpha}^{n_\calib})]$, as a function of $n_\calib$. This table shows that, from a quantile point of view, the use of $\calF_\alpha^n$ is increasingly interesting as $n_\calib$ increases, although empirical mean and variance still suffer from the heavy-tailed nature of Cauchy distributions. Note that $\calE_\alpha^n$ does not seem to suffer from this property, as shown in Figure \ref{fig:histo_cauchy} (and as could be conjectured from Proposition \ref{thm:eq_ellipsoid} and Lemma \ref{lemma:q_bound}).}
\begin{figure}[t!]\centering
\begin{subfigure}[b]{0.25\textwidth}
    \includegraphics[width=\textwidth]{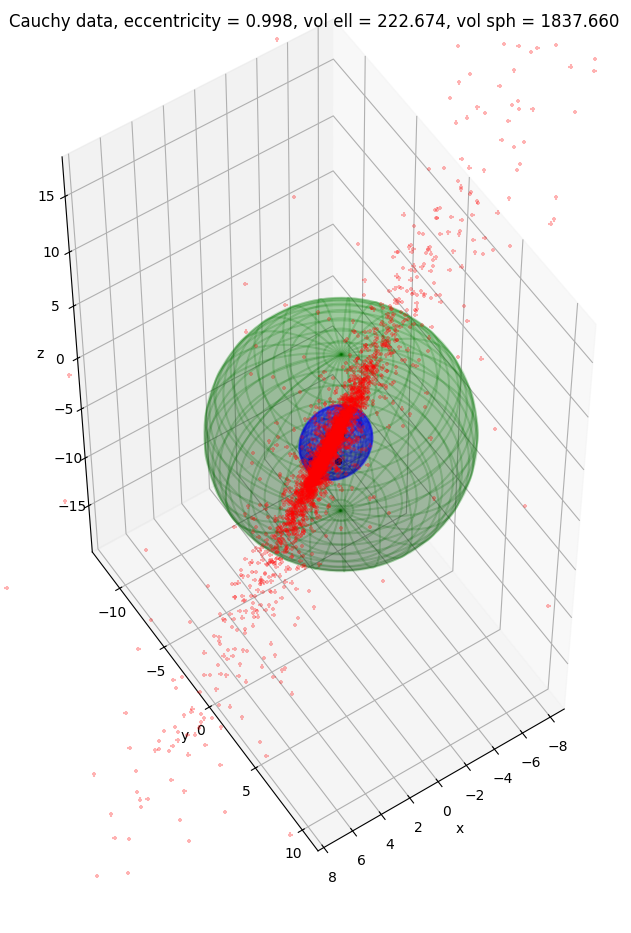}
    \caption{Top view}
\end{subfigure}\hspace{2cm}
\begin{subfigure}[b]{0.25\textwidth}
    \includegraphics[width=\textwidth]{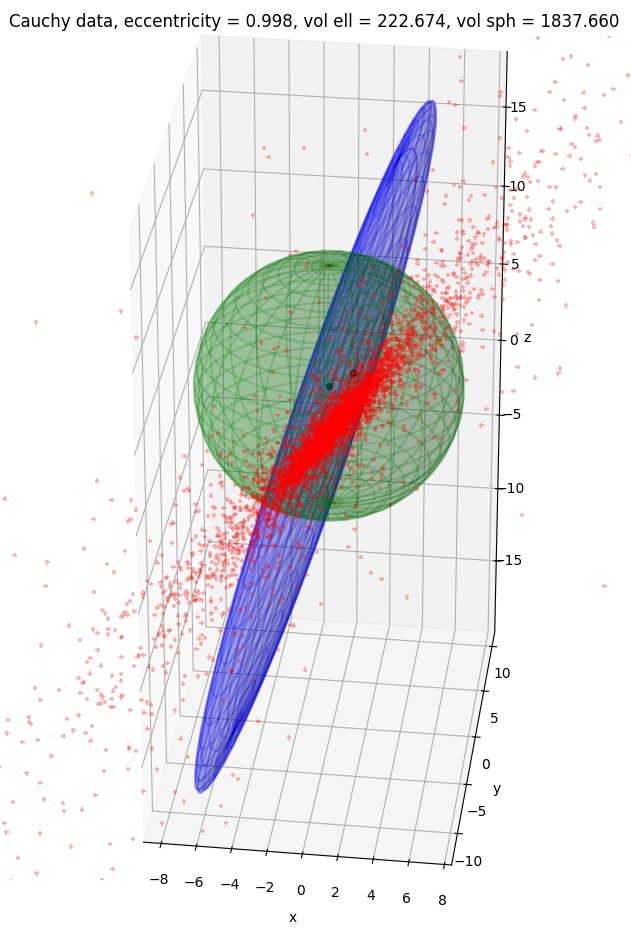}
    \caption{Side view (same example)}
\end{subfigure}
    \caption{Sample example of $\calE_\alpha^n$, Cauchy data with $\rho = 0.8$. Black dot: real value position. Large red dot: center of ellipsoid. Green dot: predictor $\widehat{Y}_{n+1}$ (also the center of the sphere). Small red dots: other residuals ($n = 6000$).}
    \label{fig:cauchy}
\end{figure}
\begin{table}[hb!]
    \centering\footnotesize
    \begin{tabular}{c||c|ccccccc||ccc}\hline
          CI set & \diagbox{$n_{\calib}$}{Quantile} & 0.1 & 0.25 & 0.5 & 0.75 & 0.95 & 0.99 & 1 & Mean & Std. dev. & Nb. inf.\\ \hline $\calF_\alpha^n$ & \multirow{2}{*}{$500$} &  \textbf{3.6e2} & \textbf{3.9e2} & \textbf{4.3e2} & 2.4e4 & $\infty$ & $\infty$ & $\infty$ & $\infty$ & $\infty$ & 36 \\
           $\calB_\alpha^n$ & & 9.3e3 & 9.5e3 & 9.7e3 & \textbf{9.9e3} & \textbf{1e4} & \textbf{1e4} & \textbf{1.1e4} & \textbf{9.7e3} & \textbf{3.1e2} & \textbf{0} \\ \hline $\calF_\alpha^n$ & \multirow{2}{*}{$10^3$} &  \textbf{3.4e2} & \textbf{3.5e2} & \textbf{3.7e2} & \textbf{4.3e2} & 3.3e5 & $\infty$ & $\infty$ & $\infty$ & $\infty$ & 9 \\ 
           $\calB_\alpha^n$ & & 8.84e3 & 8.94e3 & 8.61e3 & 9.07e3 & \textbf{9.2e3} & \textbf{9.38e3} & \textbf{9.51e3} & \textbf{9.07e3} & \textbf{1.8e2} & \textbf{0} \\ \hline
           $\calF_\alpha^n$ & \multirow{2}{*}{$10^4$} &  \textbf{3.1e2} & \textbf{3.3e2} & \textbf{3.4e2} & \textbf{3.6e2} & \textbf{4e2} & \textbf{8.2e3} & 9.2e8 & 1.8e6 & 4e7 & \textbf{0} \\ 
           $\calB_\alpha^n$ & & 8.54e3 & 8.58e3 & 8.61e3 & 8.65e3 & 8.69e3 & 8.73e3 & \textbf{8.75e3} & \textbf{8.61e3} & \textbf{5.2e1} & \textbf{0} \\ \hline
    \end{tabular}
    \caption{Cauchy data : Empirical quantiles of the distribution of $\widehat{\mathbb{E}}_{n_{\test}}[\Vol(\calF_{\alpha}^{n_\calib})]$ (empirical volume average over $n_\test = 200$), estimated from a sample of size $n_\histo = 500$,  as a function of $n_\calib$. Last column : number of full space $\calF_\alpha^n$ over the $n_\test \times n_\histo = 10^5$ test samples used to estimate the  quantiles of $\widehat{\mathbb{E}}_{n_{\test}}[\Vol(\calF_{\alpha}^{n_\calib})]$.}
    \label{tab:cauchy_fna}
\end{table}

\begin{figure}
    \centering
    \includegraphics[width=0.7\textwidth]{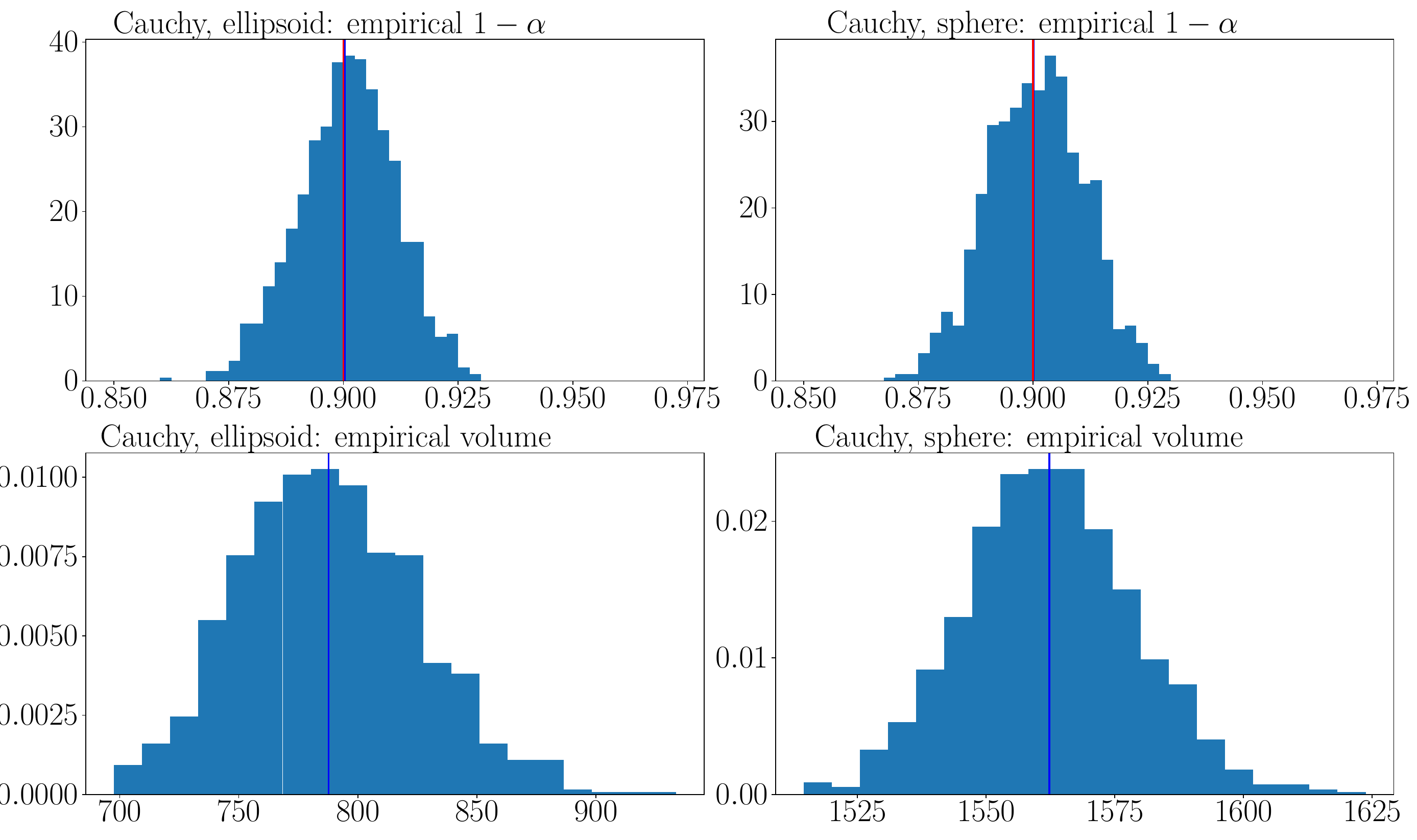}

    \caption{Empirical volumes and coverage, Cauchy data. $n_{\nsplit} = 3 \times 10^6, \ n_{\calib} = 10^4, \ n_{\test} = 800$ and $n_{\histo} = 10^3$. $50$ bins were used. Red vertical lines correspond to $x = 0.9$. Blue vertical lines correspond to the mean of each histogram.}
    \label{fig:histo_cauchy}
\end{figure}



\subsection{Non elliptical data : the inverse Dirichlet distribution}
\iain{
In this section we assume that the distribution of $(X_1^\top \ Y_1^\top)^\top$ is supported on the positive orthant $\mathbb{R}_+^p$, with a density of the form 
\begin{align}\label{eq:def_inv_dirichlet}
    f(v) \propto \frac{1}{(1 + |v|_1)^{b + |a|_1}} \prod_{i=1}^pv_i^{a_i-1}, \ \ \ v\in\mathbb{R}_+^p.
\end{align}
In equation \eqref{eq:def_inv_dirichlet}, $b>0$, $(a_1, \ldots, a_p)$ is a vector of parameters with $a_i>0$, and $|x|_1 = \sum_{i=1}^p |x_i|$. $V_1$ is said to follow the inverted Dirichlet distribution which we denote by $V_1 \sim \text{IDirichlet}(a_1, \ldots, a_p;b)$ (\cite{dirichlet_livre}, Example 8.6). Depending on its parameters, this distribution is mildly to heavily non elliptical, as is visible from FIgure \ref{fig:dirichlet_examples}. Observe also that $b$ controls the decay of the distribution. We next use a linear regressor trained on $n_\nsplit = 5000$ examples. }
\iain{Practical experiments show that our method is efficient when the distribution of $V_1$ is approximately elliptical. This can be seen in Figure \ref{fig:histo_dirichlet_good}, which corresponds to $k=1, \ \ell = 3$, $b = 3$ and $a = (1, 1, 10^{-1}, 10^{-2})$. Its bottom histograms show that for this distribution, the volume of $\calF_\alpha^n$ is in average much than that of $\calB_\alpha^n$. A sample corresponding to this histogram is given in Figure \ref{fig:dirichlet_examples}, subfigures (a) and (b). These figures show that the set $\calF_\alpha^n$ was able to capture part of the directionality associated to the parameters $b$ and $a$ above. We next consider an unfavourable case, where $b = 1$ and $a = 10^{-2}\times(1, 1, 1, 1)$. A sample example in given in Figure \ref{fig:dirichlet_examples}, subfigures (c) and (d). This figure shows that the calibration residuals are, for the most part, oriented along one of the three $x, y$ or $z$ axes. on this example, compared to the ball $B_\alpha^n$, the ellipsoid $\calF_\alpha^n$ seems to be overly sensitive to the observation of $X_1$. As for the Cauchy data (Table \ref{tab:cauchy_fna}), unfavourable test cases are best studied with a table of empirical quantiles which we provide in Table \ref{tab:dirichlet_fna}. This table shows that, even for $n_\calib = 10^4$, CCLE is severely outperformed by the norm-residual score (at least two order of magnitudes between corresponding quantiles, even for the empirical quantile of order $0.1$).}
\begin{figure}[ht!]
\centering
\begin{subfigure}[b]{0.3\textwidth}
    \includegraphics[width=\textwidth]{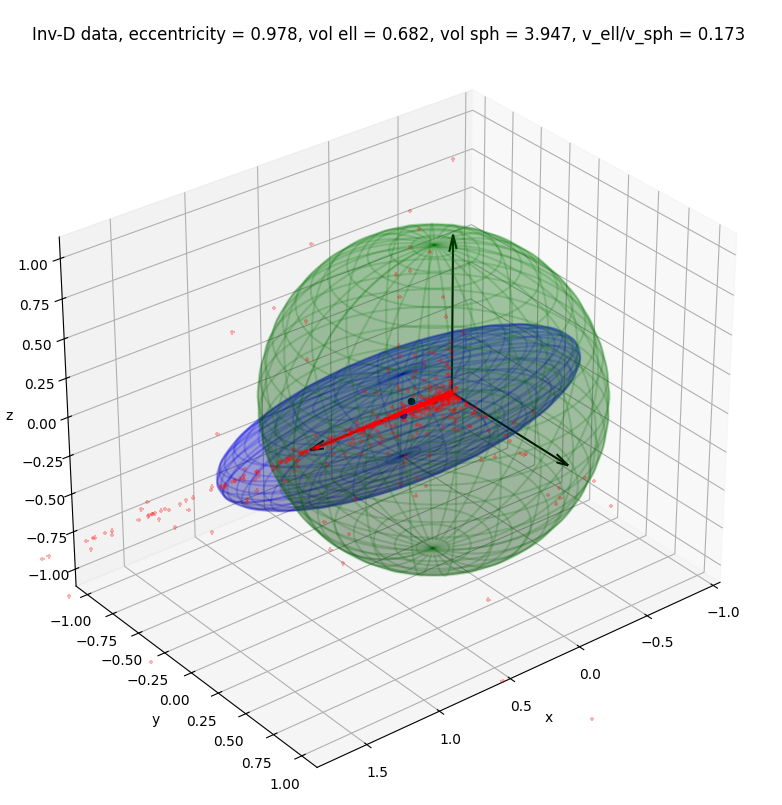}
    \caption{\centering Favourable inv. Dirichlet distribution (side)}
\end{subfigure}\hspace{2cm}
\begin{subfigure}[b]{0.3\textwidth}
    \includegraphics[width=\textwidth]{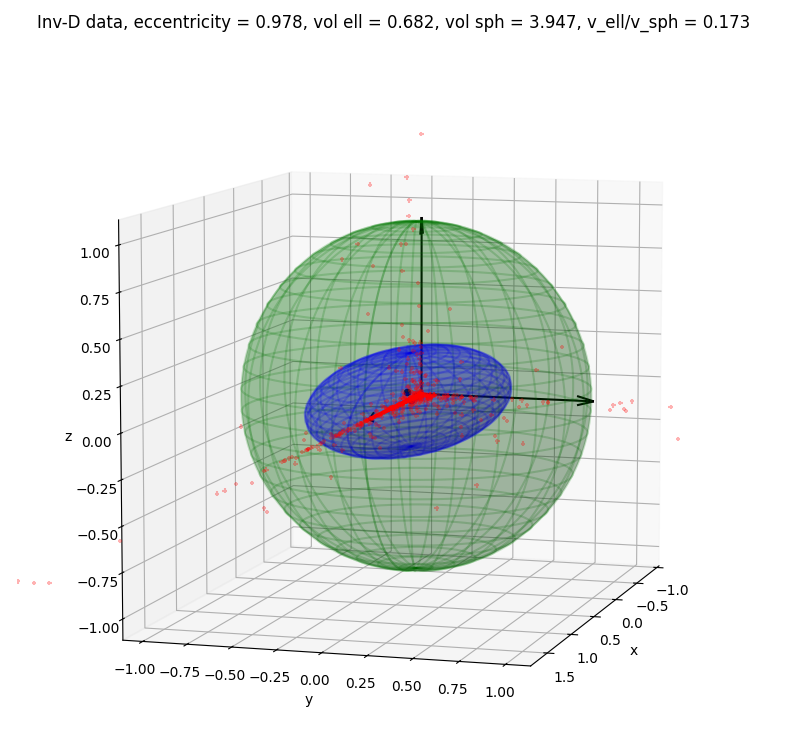}
    \caption{\centering Favourable inv. Dirichlet distribution (side)}
\end{subfigure}\\
\begin{subfigure}[b]{0.3\textwidth}
    \includegraphics[width=\textwidth]{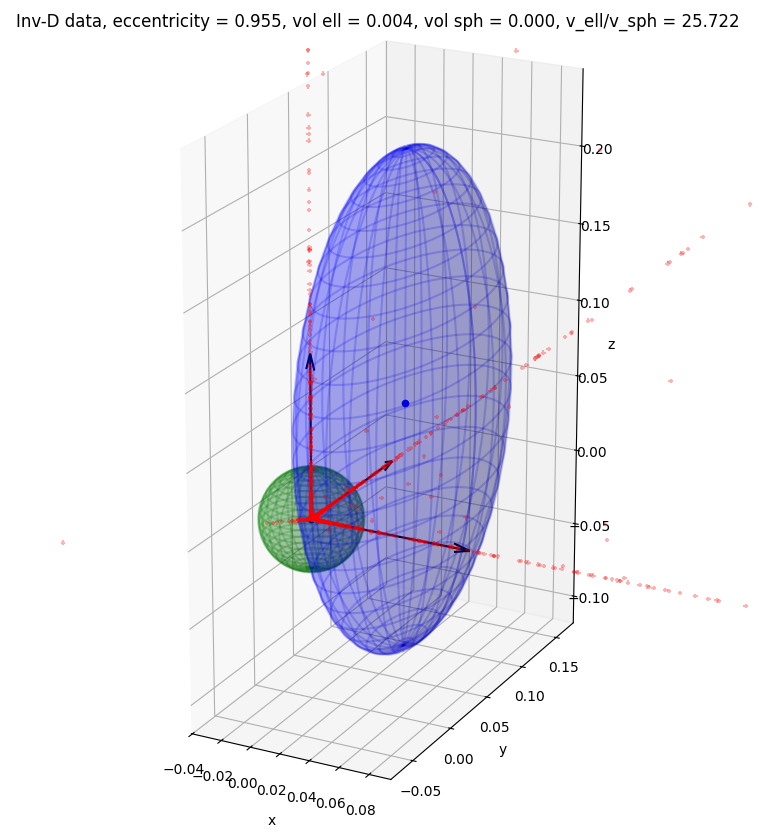}
    \caption{\centering Unfavourable inv. Dirichlet distribution (side)}
\end{subfigure}\hspace{2cm}
\begin{subfigure}[b]{0.3\textwidth}
    \includegraphics[width=\textwidth]{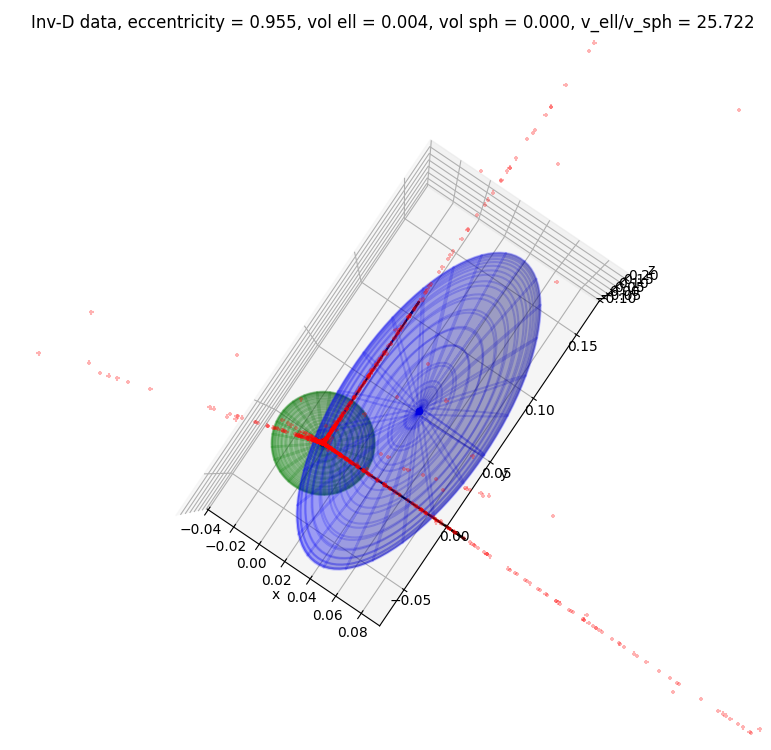}
    \caption{\centering Unfavourable inv. Dirichlet distribution (top)}
\end{subfigure}
    \caption{Two sample examples, inverse Dirichlet data. Black dot: real value position. Blue dot: center of ellipsoid. Green dot: predictor $\widehat{Y}_{n+1}$ (also the center of the sphere). Small red dots: other residuals ($6000$). Subfigures (a) and (b) : side views of an example simulated with a favourable inv. Dirichlet distribution (parameters $b=3$, $a = (1, 1, 10^{-1}, 10^{-2}).$). Subfigures (c) and (d) : side and top views of an example simulated with an unfavourable inv. Dirichlet distribution (parameters $b=1, \ a = 10^{-2} \times (1, 1, 1, 1)$).}
    \label{fig:dirichlet_examples}
\end{figure}
\begin{figure}[!ht]
    \centering
    \includegraphics[width=0.7\textwidth]{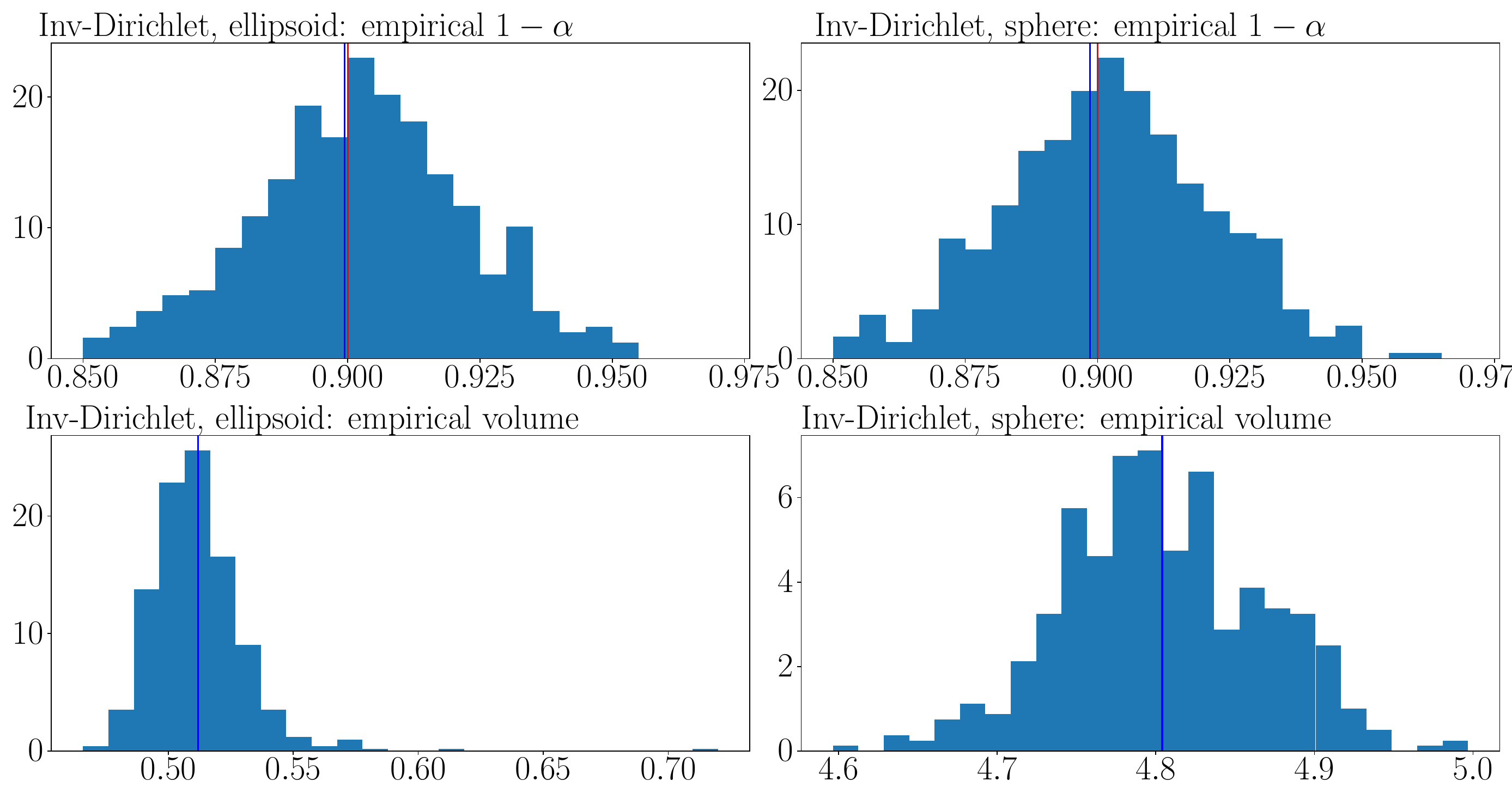}
    \caption{Empirical volumes and coverage, Inverse Dirichlet data with $b=3, a = (1, 1, 10^{-1}, 5.10^{-2})$. $n_{\nsplit} = 5000, \ n_{\calib} = 1000, \ n_{\test} = 200$ and $n_{\histo} = 500$. $25$ bins were used. Red vertical lines correspond to $x = 0.9$. Blue vertical lines correspond to the mean of each histogram. }
    \label{fig:histo_dirichlet_good}
\end{figure}
\begin{table}[!ht]
    \centering\footnotesize
    \begin{tabular}{c||c|ccccccc||ccc}\hline
          CI set & \diagbox{$n_{\calib}$}{Quantile} & 0.1 & 0.25 & 0.5 & 0.75 & 0.95 & 0.99 & 1 & Mean & Std. dev. & Nb. inf.\\ \hline  $\calF_\alpha^n$ & \multirow{2}{*}{$500$} & 8.2e-3 & 1.1e-2 & 1.8e-2 & 4.2e-2 & 4.2 & 4.7e4 & $\infty$ & $\infty$ & $\infty$ & 3 \\ 
           $\calB_\alpha^n$ & &  \textbf{5.4e-4} & \textbf{5.8e-4} & \textbf{6.3e-4} & \textbf{6.9e-4} & \textbf{7.7e-4} & \textbf{8.6e-4} & \textbf{9.2e-4} & \textbf{6.3e-4} & \textbf{7.7e-5} & \textbf{0} \\ \hline
           $\calF_\alpha^n$ & \multirow{2}{*}{$10^3$} & 1.2e-2 & 1.5e-2 & 2.3e-2 & 4.8e-2 & 8.2e-1 & 6.7e2 & $\infty$ & $\infty$ & $\infty$ & 1 \\ 
           $\calB_\alpha^n$ & &  \textbf{2.3e-4} & \textbf{2.4e-4} & \textbf{2.6e-4} & \textbf{2.7e-4} & \textbf{2.9e-4} & \textbf{3.1e-4} & \textbf{3.5e-4} & \textbf{2.6e-4} & \textbf{2.3e-5} & \textbf{0} \\ \hline
           $\calF_\alpha^n$ & \multirow{2}{*}{$10^4$} &  4e-2 & 4.9e-2 & 6.8e-2 & 1.1e-1 & 4.9e-1 & 2.1 & 1e3 & 2.3 & 45.7 & \textbf{0} \\ 
           $\calB_\alpha^n$ & & \textbf{3.3e-4} & \textbf{3.3e-4} & \textbf{3.4e-4} & \textbf{3.4e-4} & \textbf{3.5e-4} & \textbf{3.5e-4} & \textbf{3.6e-4} & \textbf{3.4e-4} & \textbf{7.1e-6} & \textbf{0} \\ \hline
    \end{tabular}
    \caption{Inverse Dirichlet data with parameters $b=1$ and $a=10^{-2} \times (1,1,1,1)$: empirical quantiles of the distribution of $\widehat{\mathbb{E}}_{n_{\test}}[\Vol(\calF_{\alpha}^{n_\calib})]$ (empirical volume average over $n_\test = 200$), estimated from a sample of size $n_\histo = 500$,  as a function of $n_\calib$. Last column : number of full space $\calF_\alpha^n$ over the $n_\test \times n_\histo = 10^5$ test samples used to estimate the  quantiles of $\widehat{\mathbb{E}}_{n_{\test}}[\Vol(\calF_{\alpha}^{n_\calib})]$.}
    \label{tab:dirichlet_fna}
\end{table}
\section{Conclusion}
In this article, we have introduced two covariance based scores, for conformal inference in multivariate regression. We have shown that the resulting confidence regions are conservatively and accurately approximated with an explicit ellipsoid. We have also studied their asymptotic properties, and compared them with the regions stemming from the standard conformal score function (norm of the residual) as used in the literature. Our first set of numerical experiments shows that our scores \iain{typically} perform better than that of the standard norm score, in the sense that the corresponding volume is typically much lower than that of the spheres associated with the residual norm score, provided that the residuals $R_i = Y_i - \widehat{Y}_i$ are not uncorrelated with the input $X_i$. The next step will be to apply our methodology on time series, where the adaptivity of our score function may be expected to improve the performance of CI for time series, especially concerning the longitudinal coverage (coverage within a single sample of the time series); this study will be accompanied by a Python package. We will soon extend our methodology to functional data \cite{GHMR_forthcoming}.
Finally, a limitation of our scoring rule is that is may not be expected to be optimal for multimodal or non elliptical distributions of the residuals: the shape of $\widetilde{\calC}_{\alpha}^n$ is bound to be an ellipsoid. As such, generalizations of our score may be sought to include moments of order $3$, to take asymmetry into account, although the matrix algebra may become intractable.

%

\begin{appendix}
\section{Additional remarks}
\subsection{Interpretations of the score vector \texorpdfstring{$\diag(\bfS(z))$}{diag S(z)}} \label{app:score_interpretation}
 \paragraph*{Leverage score}
    Consider a multivariate linear model between the matrices $\bfR$ and $\bfX$, $\bfR = \bfX\beta + \varepsilon$.  The ridge estimator in this linear regression model is $\widehat{\beta} = (\bfX^\top \bfX + \mu \mathbf{I})^{-1}\bfX^\top \bfR$. Let us denote $\widehat{\bfR} = \bfX\widehat{\beta} = \bfX(\bfX^\top \bfX + \mu \mathbf{I})^{-1}\bfX^\top \bfR = \bfM\bfR$. The leverage score of the $\ith$ example is the diagonal coefficient $\bfM_{ii}\in[0,1]$. 
     As $M=\nabla_R \widehat{R}$, the leverage score can be interpreted as a self-influence score. A high leverage score $M_{ii}$ means that the $\ith$ example $(X_i, R_i)$ is very influential in the behaviour of the linear model \cite{kim2020sources, leverage1986}. 
     The leverage score  has been used, for example, for outliers detection  \cite{leverage1986}.  This last task can also be tackled using conformal inference \cite{Batesetal23}.  Note further that the leverage score  does not depend on any underlying model nor on $Y$. Hence, no observational model is required for computing the leverage score. 
    Our score is thus a form of leverage score, where 
    we replaced the matrix $\bfM$ 
    by $\bfM' = \bfW(\bfW^\top \bfW + \mu \mathbf{I}_p)^{-1}\bfW^\top$. Hence, we jointly consider both the explanatory variables and the residual errors in the leverage score. Further, there is  
    another difference with the classical leverage score in the centering step, as we consider $\bfW$ instead of $\bfV$. 
    \paragraph*{Shape theory} 
    In shape theory \cite{dryden2016}, iid random vectors are observed and stacked, say row-wise, in a matrix $\bfV$. In shape theory, one then typically studies the rectangular matrix $\bfV$ according to its polar decomposition, $\bfV = \bfP\bfT$, where $\bfT$ is the {\it modulus} and $\bfP$ is the orientation ($\bfT = (\bfV^\top \bfV)^{1/2}$ and $\bfP = \bfV(\bfV^\top \bfV)^{-1/2}$). The matrices $\bfP$ and $\bfT$ are then studied separately (\cite{downs1972}, \cite{chikuse2003}, Section 1.3.3). In particular, the random orthogonal projector $\bfP\bfP^\top$ is an orientation statistic \cite{chikuse2003}, Section 1.3.3, encoding directional information of the data. In our case, when $\lambda = 0$, the connection between our score \eqref{eq:def_score_S} and the orientation statistic $\bfP\bfP^\top$ is the equality $\bfV(\bfV^\top \bfV)^{-1}\bfV^\top  = \bfP \bfP^\top$. As for the leverage score though, the data is usually not centered in shape theory: our score in equation \eqref{eq:def_score_S} is, when $\lambda=0$, $n^{-1}\bfS = \bfW(\bfW^\top \bfW)^{-1}\bfW^\top \neq \bfV(\bfV^\top \bfV)^{-1}\bfV^\top$. In our framework the centering step is useful to deal with potentially biased predictors $\widehat{f}$, i.e. non centered residuals.

\subsection{Different assumptions for Proposition \ref{prop:bound}} \label{app:rem_prop:bound}
\paragraph*{The case $\lambda=0$}
Our proof of Proposition \ref{prop:bound} uses the fact that $\lambda > 0$ (equation \eqref{eq:sn_interm}): if $\lambda = 0$, we should require that $\min \Spec(\mathbf{\Sigma}) >0$. However, even if we can show that $\min \Spec(\widehat{\mathbf{\Sigma}}_n)$ tends to $\min \Spec(\mathbf{\Sigma})$ almost surely (proof of Proposition \ref{prop:asymptotics}), adapting the proof of Proposition \ref{prop:bound} to $\lambda=0$ would require the control of the moments of $\min \Spec(\widehat{\mathbf{\Sigma}}_n)^{-1}$, which is a difficult task in itself. We leave the details of the corresponding proof to the interested reader, e.g. using that almost surely, $\min \Spec(\widehat{\mathbf{\Sigma}}_{n,0}) \rightarrow \min \Spec(\mathbf{\Sigma})$, from the SLLN on $\widehat{\mathbf{\Sigma}}_{n,0}$ and the continuity of the spectrum for symmetric matrices (see e.g. \mbox{\cite[Problem 1, p. 198]{horn}}). 

\paragraph*{Finite moment of $V_1$} \iain{The assumption that $\mathbb{E}[\|V_1\|^{4q}] < +\infty$ (which is also used in several other places) enables us to use Jensen's inequality, to show that $\sup_{z\in K}\max_{1\leq i\leq n}nb_i^n(z)^2 \rightarrow 0$ in probability. Under the weaker assumption that} $\mathbb{E}[\|V_1\|^{2}] < + \infty$, we can still show that the first $n$ coordinates of $b^n$ verify $ n\sum_{i=1}^nb_i^n(z)^2 = (1 + \|z\|^2)O_{\Pb}(1)$,
where the $O_{\mathbb{P}}(1)$ term is uniform in $z$ (e.g. adapt equation \eqref{eq:control_sum_bnx_dnx}). Unfortunately, this result alone is not strong enough to imply equation \eqref{eq:vol_surplus}.

\subsection{Infinite variance case} \label{app:infinite_var}
From equation \eqref{eq:ellipsoid}, $\calE_{\alpha}^{\infty}$ displays a scaling invariance, which we expect to have important implications for the analysis of our method in the case of heavy tailed data \cite{tyler}, Condition 1.1. For the moment, let us view $\mathbf{\Sigma}_{\lambda}$ and $V_1$ as independent parameters. Observing that $\boldcalA_\infty, \ Z_0^\infty$ and $\rho_{\infty,\alpha}$ are functions of $\mathbf{\Sigma}_{\lambda}$, we denote $\calE_{\alpha}^\infty = \calE_{\alpha}(\mathbf{\Sigma}_{\lambda}, V_1)$. We now prove that, as sets, $\calE_{\alpha}(\mathbf{\Sigma}_{\lambda}, V_1) = \calE_{\alpha}(\delta\mathbf{\Sigma}_{\lambda}, V_1)$ for all $\delta>0$. Indeed, assume that $\mathbf{\Sigma}_\lambda$ is replaced with $\delta\mathbf{\Sigma}_\lambda$ for some $\delta>0$, then from equation \eqref{eq:pred_unbiased} $Z_0^\infty$ is left unchanged. Likewise, $\boldcalA_\infty$ is changed to $\delta \boldcalA_{\infty}$, while
\begin{align}
    \frac{\rho_{\alpha,\infty}}{\delta} &= \frac{1}{\delta}q_{1-\alpha}(V_\rmc^\top\mathbf{\Sigma}_\lambda^{-1}V_\rmc) - \frac{1}{\delta}X_\rmc^\top(\mathbf{\Sigma}_\lambda^{11})^{-1}X_\rmc \nonumber\\ &= q_{1-\alpha}(V_\rmc^\top(\delta\mathbf{\Sigma}_\lambda)^{-1}V_\rmc) - X_\rmc^\top(\delta\mathbf{\Sigma}_\lambda^{11})^{-1}X_\rmc,
\end{align}
so that the product $\rho_{\alpha,\infty}\boldcalA_{\infty} =(\delta^{-1}\rho_{\alpha,\infty})(\delta\boldcalA_{\infty})$ is also unchanged when $\mathbf{\Sigma}_\lambda$ is replaced with $\delta\mathbf{\Sigma}_\lambda$. Hence, following equation \eqref{eq:conditional_ellipsoid_asymptotic}, $\calE_{\alpha}(\mathbf{\Sigma}_{\lambda}, V_1) = \calE_{\alpha}(\delta\mathbf{\Sigma}_{\lambda}, V_1)$. This scaling invariance is, in fact, a key property for building the Tyler dispersion matrix estimator \cite{tyler} (see equation \eqref{eq:ellip_assumption} and its comments for the dispersion matrix), which is known to be robust to infinite variance in the case of elliptical distributions. \iain{The same scaling arguments can also be made for $\calF_\alpha^\infty$, but the convergence of $\rho_{n,\alpha}'$ to $q_{1-\alpha}^\calF$ is not at all clear (see e.g. the assumptions of Proposition \ref{prop:calf_asymptotic}), as can be seen from the numerical experiments with Cauchy data.} The further study of our method for heavy-tailed data is left for future work.

\subsection{Non asymptotic analysis in the Gaussian case} \label{app:non_asymp_gauss}
One may wish to understand the properties of $\calE_{\alpha}^n$ for finite sample size $n$. For example, assuming that $\lambda=0$, it is known that $\boldcalA_n = \widehat{\mathbf{\Sigma}}_n/\widehat{\mathbf{\Sigma}}_n^{11}$ follows the Wishart distribution $W_{\ell}(n-k-1,n^{-1}\mathbf{\Sigma}/\mathbf{\Sigma}^{11})$ \cite{muirhead1982}, Theorem  3.2.10. We now turn to the quantities appearing in $\rho_{n,\alpha}$. Observing that $\sqrt{n/(n+1)}X_{n+1}^\rmc\sim\calN(0,\mathbf{\Sigma}^{11})$, that $n\widehat{\mathbf{\Sigma}}_n^{11}\sim W_k(n-1,\mathbf{\Sigma}^{11})$ and that both are independent \cite{mardia_kent}, Corollary 3.3.3.2, we obtain that $[(n-1)/(n+1)](X^{n+1}_\rmc)^\top (\widehat{\mathbf{\Sigma}}_{n}^{11})^{-1}X_{n+1}^\rmc$ follows the Hotelling $T^2$ distribution $T^2(k,n-1)$ \cite{mardia_kent}, Theorem 3.5.1. Let us now also assume that $\mathbb{E}[V_1]$ is known. In this case, it is reasonable to replace $n^{-1}\sum_{i=1}^nV_i$ with $\mathbb{E}[V_1]$ in $\bfB_n$ (equation \eqref{eq:def_Bn}), so that $\bfB_n$ follows the matrix normal distribution $\bfB_n \sim \calN(0, \mathbf{I}_n, \mathbf{\Sigma})$ \cite{chikuse2003}, Section 1.5.3. In this case, $\bfP_{n,0}$ is a random orthogonal projector whose distribution is uniform over the so-called ``special manifold'' $\bfP_{p,n-p} \subset \calM_{n,n}$ of orthogonal projectors of rank $p$, which is isomorphic to the Grassmanian manifold $G_{p, n}$ \cite{chikuse2003}, Theorem 2.4.9. Finally, building on
\cite{muirhead1982}, Exercise 3.15 p. 117, we can show that $({\bfP_{n,0}})_{ii}$ follows a Beta distribution $B(p/2,(n-p)/2)$. Note though that those results, while instructive, are not sufficient to describe the distribution of $\rho_{n,\alpha}$.

When $\mathbb{E}[V_1]$ is unknown, we expect to obtain perturbations of the distributions above for $\bfP_{n,0}$ and $({\bfP_{n,0}})_{ii}$. Still, to our knowledge, there are no simple expressions available in this case. Indeed, the procedure for building $\bfB_n$ correlates the columns of $\bfB_n$ through equation \eqref{eq:B_ij}, and $\bfB_n \sim \calN(0, \pi^{\perp}_{\mathbbm{1}}, \mathbf{\Sigma})$ (here, $\mathbbm{1}\in\mathbb{R}^n$). As such, the distribution of $\bfP_{n,0}$ does not seem to be available in closed form since neither $\pi^{\perp}_{\mathbbm{1}}$ nor $\mathbf{\Sigma}$ are equal to the identity matrix \cite{chikuse2003}, Chapter 2.

\section{Technical lemmas}  \label{app:lemmas}
The following lemma will be very useful in several proofs.
\begin{lemma}[Block matrix inversion, \cite{horn},  equation (7.7.5) p. 472]\label{lemma:block_inversion}
If $\bfM = \begin{pmatrix}
             \bfA &\bfB \\
             \bfB^\top &\bfC
    \end{pmatrix}$ is invertible and $\bfA$ is invertible, then $\bfM/\bfA = \bfC - \bfB^\top \bfA^{-1}\bfB$ is invertible and
\begin{align}
    \bfM^{-1} = \begin{pmatrix}
             \bfA^{-1} + \bfA^{-1}\bfB(\bfM/\bfA)^{-1}\bfB^\top \bfA^{-1} &-\bfA^{-1}\bfB(\bfM/\bfA)^{-1} \\
             -(\bfM/\bfA)^{-1}\bfB^\top \bfA^{-1} & (\bfM/\bfA)^{-1}
    \end{pmatrix}.
\end{align}
\end{lemma}
\iain{A central consequence of this lemma is that if $\bfM$ is partitioned as above and $(x^\top y^\top)^\top$ is a similarly partitioned vector, then setting $t = y - \bfA^{-1}\bfB x$,
\begin{align}\label{eq:apply_block_lemma}
    \begin{pmatrix}
         x\\y
    \end{pmatrix}^\top\bfM^{-1}\begin{pmatrix}
         x\\y
    \end{pmatrix} = x^\top\bfA^{-1}x + t^\top(\bfM/\bfA)^{-1}t.
\end{align}
}
Next, the efficient computation of the score matrix is enabled by Lemma \ref{lemma:scores_transpose} below, which provides an ``explicit'' representation of equation \ref{eq:def_score_S}, in the case where $W(z)$ is a rank one perturbation of a reference matrix.
\begin{lemma}\label{lemma:scores_transpose}
If $\mu\geq 0$, $\bfM\in \calM_{n,p}$ is of the form $\bfM = \bfA + wu^\top $ where $w \in\calM_{n,1}$, $\|w\| = 1$ and $u\in\calM_{p,1}$, then there exists $\bfC\in\calM_{n,n}$, $b\in\calM_{n,1}$ and $d\in\Rbb^+$ such that
\begin{align}\label{eq:matrix_inversion}
    \bfM(\bfM^\top \bfM+\mu \mathbf{I})^{-1}\bfM^\top  = \bfC - \frac{bb^\top }{1+d}.
\end{align}
$\bfC, b$ and $d$ are given by
\begin{align}
\bfC = \bfB \bfD_{\mu}^{-1}\bfB^\top  + ww^\top , \ \ b = \bfB \bfD_{\mu}^{-1}r-w, \ \
    d = r^\top \bfD_{\mu}^{-1}r, 
\end{align}
where
\begin{align}
     \bfB = (I-ww^\top )\bfA = \pi^{\perp}_{w}\bfA,\ \ \bfD_{\mu} = \bfB^\top \bfB+\mu \mathbf{I}, \ \ r = u + \bfA^\top w.
\end{align}
\end{lemma} 
Above, $\bfB_{\mu}^+ \coloneqq \bfB\bfD_{\mu}^{-1}$ is a regularized pseudo-inverse of $\bfB^\top$: when $\mu = 0$, $\bfB_{\mu}^+\bfB^\top $ is the orthogonal projector onto the range of $\bfB$. Note also that $\bfB^\top w = \bfA^\top (I-ww^\top )w = 0$, hence $\bfC = \bfB_{\mu}^+\bfB^\top  + ww^\top $ is also a ($\mu$-regularized) orthogonal projector. Due to the fact that the score matrix is built upon centered data, the quantities in equation \eqref{eq:matrix_inversion} can be further simplified, as stated in the following lemma.
\begin{lemma}\label{lemma:simplify}
Assume that the matrix $\bfA$ in Lemma \ref{lemma:scores_transpose} lies in $\calM_{n+1,p}$, that it is of the form $\bfA = \pi_{\mathbbm{1}}^{\perp}\bfV$ for some matrix $V\in\calM_{n+1, p}$, and that $w$ in Lemma \ref{lemma:scores_transpose} is $w = v/\|v\|$, where $v$ is given in equation \eqref{eq:def_v}. Then $\bfB$ is given by
\begin{align}
\bfB_{ij} &= \bfV_{ij} - \frac{1}{n}\sum_{l=1}^n\bfV_{il}, \ \ \ e_{n+1}^\top \bfB = 0 \ \ \ \text{(null last row).}    \label{eq:B_ij}
\end{align}
The last equation implies that  for $\bfC$ and $b$ given in Lemma \ref{thm:eq_ellipsoid},
\begin{align}
    (bb^\top )_{n+1,n+1} = b_{n+1}^2 = \bfC_{n+1,n+1} =  w_{n+1}^2 = {n}/{(n+1)}.
\end{align} 
\end{lemma}
The key property for this lemma to hold is that $\mathbbm{1}^\top v = 0$. Applying this lemma to our score, we can further describe the elements of the score matrix $\bfS(z)$.
\begin{lemma}\label{lemma:expr_bn_dn_rn}
In the expression of the score matrix
\begin{align}
    \bfS(z) = n\bfC_n - n\frac{b^n(z)b^n(z)^\top}{1+d_n(z)},
\end{align}
the matrix $\bfC_n$, the vector $b^n(z)$ and the scalar $d_n(z)$ are given by 
\begin{align}
\bfC_n &= \begin{pmatrix}
             \bfP_{n,\lambda}  &\mathbf{0}_{n,1} \\
             \mathbf{0}_{1,n}  &0 
    \end{pmatrix} + ww^\top, \\
 d_n(z) &= \frac{1}{n} r_n(z)^\top \widehat{\mathbf{\Sigma}}_{n,\lambda}^{-1}r_n(z), \ \     b^n(z) = \frac{1}{n} \begin{pmatrix}
         \bfB_n \widehat{\mathbf{\Sigma}}_{n,\lambda}^{-1}r_n(z) \\ 
         0
\end{pmatrix} - w. \label{eq:d_n_b_n}
\end{align}
Here, $r_n(z)$ is given by
\begin{align}\label{eq:expr_rn}
r_n(z)  = \|v\|\bigg[\begin{pmatrix}
         X_{n+1} \\ z
\end{pmatrix} -  \frac{1}{n}\sum_{i=1}^nV_i\bigg]\in \calM_{p,1}.
\end{align}
\end{lemma}
Above, $r_n(z)$ corresponds to $r$ in Lemma \ref{lemma:scores_transpose}.
Do note that above, $v$ and $w$ also depend on $n$. \iain{We can use the same arguments (lemmas \ref{lemma:scores_transpose} and \ref{lemma:simplify}) to obtain a similar expression for the matrix $\bfS_X$ defined in equation \eqref{eq:def_Sp_Sx}.
\begin{lemma}\label{lemma:expr_bnx_dnx_rnx}
We can write the matrix $\bfS_X$ as
\begin{align}
    \bfS_X = n\bfC_n^X - n\frac{b_X^n(b_X^n)^\top}{1+d_n^X},
\end{align}
where the vector $b_X^n$ and the scalar $d_n^X$ are given by 
\begin{align}
 d_n^X &= \frac{1}{n+1} (X_{n+1}^\rmc)^\top (\widehat{\mathbf{\Sigma}}_{n,\lambda}^{11})^{-1}X_{n+1}^\rmc, \ \     b_X^n = \frac{\|v\|}{n} \begin{pmatrix}
         \bfB_n^X (\widehat{\mathbf{\Sigma}}_{n,\lambda}^{11})^{-1}X_{n+1}^\rmc \\
         0
\end{pmatrix} - w. \label{eq:d_nX_b_nX}
\end{align}
In $b_X^n$, $\bfB_n^X$ is the left $n\times k$ block of $\bfB_n$ and $\bfC_n^X$ is given by
\begin{align*}
\bfC_n^X &= \begin{pmatrix}
             \bfP_{n,\lambda}^{XX}  &\mathbf{0}_{n,1} \\
             \mathbf{0}_{1,n}  &0 
    \end{pmatrix} + ww^\top, \ \ \
    \text{where} \ \ \ \bfP_{n,\lambda}^{XX} = \bfB_n^X\Big((\bfB_n^X)^\top\bfB_n^X + n \lambda \bfI_k\Big)^{-1}(\bfB_n^X)^\top.
\end{align*}
\end{lemma}
}
Above, be wary that $\bfP_{n,\lambda}^{XX}$ is not a sub-block of $\bfP_{n,\lambda}$.
Our last lemma is pivotal to Section \ref{sec:asymp}.
\begin{lemma}\label{lemma:cv_quantile}
Take the assumptions of Proposition \ref{prop:asymptotics}. Then we have that almost surely,
\begin{align}
    \forall t \in \mathbb{R}^p, \ \frac{1}{n}\sum_{j=1}^n \exp\left({i t^\top \widehat{\mathbf{\Sigma}}_{n,\lambda}^{-1/2}(V_i - \overline{V}_n)}\right) \xlongrightarrow[n\rightarrow \infty]{} \mathbb{E}\big[\exp\big({i t^\top \mathbf{\Sigma}_{\lambda}^{-1/2}V_\rmc}\big)\big],
\end{align}
where $\overline{V}_n = n^{-1}\sum_{i=1}^nV_i$ and $V_\rmc = V_1 -\mathbb{E}[V_1]$.
In particular, under the assumptions of Proposition \ref{prop:asymptotics}, 
\begin{align}\label{eq:cvqna_e}
    q_{n,\alpha} \xlongrightarrow[n\rightarrow \infty]{a.s.} q_{1-\alpha}^{\calE}.
\end{align}
\iain{
Next, for all $i\in\{1, \ldots, n\},$ set $ p_{i,n}'' \coloneqq (\bfP_{n,\lambda} - \bfP_{n,\lambda}^{XX})_{ii}$ and define $q_{n,\alpha}''$ as the order statistic of order $n_\alpha$ of the $n$-tuple $(np_{1,n}'', \ldots, np_{n,n}'')$.  Then under the assumptions of Proposition \ref{prop:asymptotics},
\begin{align}\label{eq:cvqna_second}
    q_{n,\alpha}'' \xlongrightarrow[n\rightarrow \infty]{a.s.} q_{1-\alpha}^{\calF}.
\end{align}
Finally, if additionally $\mathbb{E}[\|X_1\|^{4q}] < + \infty$ for some $q>1$ and taking the definition of $q_{n,\alpha}'$ in equation \eqref{eq:def_pn_qn_prime}, then
\begin{align}\label{eq:cvqna_prime}
    q_{n,\alpha}' \xlongrightarrow[n\rightarrow \infty]{\mathbb{P}} q_{1-\alpha}^{\calF}.
\end{align}
}
\end{lemma}
The main difficulty of this lemma is the fact that $q_{n,\alpha}$ \iain{(resp. $q_{n,\alpha}'$)} is an order statistic built from $(np_{1,n},  \ldots, np_{n,n})$ \iain{(resp. $(np_{1,n}',  \ldots, np_{n,n}')$)} , which are identically distributed but not independent.
\paragraph*{Matérn covariance functions}
We conclude this section with the expressions of the Matérn covariance functions that are used in Section \ref{sec:numexp}. Setting $H = |h|\sqrt{2\nu}/L$, $k_{\nu}$ is given by \cite{gpml}, Section 4.2 p 85,
\begin{align}
    k_{\nu=1/2}(h) &= \sigma^2\exp(-H),\\
    k_{\nu=3/2}(h) &= \sigma^2(1+H)\exp(-H),\label{eq:matern-3/2}\\
    k_{\nu=5/2}(h) &= \sigma^2(1+H + H^2/3)\exp(-H),\\
    k_{\nu=7/2}(h) &= \sigma^2(1+H + 2H^2/5 + H^3/15)\exp(-H).
\end{align}

\section{Proofs} \label{app:proofs}

\begin{proof}[Proof of Lemma \ref{lemma:scores_transpose}]
We have
\begin{align}
    \bfM^\top \bfM &= \bfA^\top \bfA + (uw^\top \bfA+\bfA^\top wu^\top ) + uu^\top = \bfA^\top \bfA + (u + \bfA^\top {w})(u + \bfA^\top {w})^\top  - \bfA^\top ww^\top \bfA\nonumber \\
    &= (\bfA^\top \bfA - \bfA^\top ww^\top \bfA) + rr^\top  = \bfB^\top \bfB + rr^\top .
\end{align}
With $\bfD_{\mu} \coloneqq \bfB^\top \bfB + \mu \mathbf{I}$, the Sherman-Morrison formula yields
\begin{align}
    (\bfM^\top \bfM + \mu \mathbf{I})^{-1} = (\bfB^\top \bfB + \mu \mathbf{I} + rr^\top )^{-1} = \bfD_{\mu}^{-1} - \frac{\bfD_{\mu}^{-1}rr^\top \bfD_{\mu}^{-1}}{1 + r^\top \bfD_{\mu}^{-1}r} =  \bfD_{\mu}^{-1} - \frac{\bfD_{\mu}^{-1}rr^\top \bfD_{\mu}^{-1}}{1 + d}.
\end{align}
Moreover,
\begin{align}
    \bfM = \bfA + wu^\top  = (I-ww^\top )\bfA + ww^\top \bfA + wu^\top  = \bfB + w(\bfA^\top w+u)^\top  = \bfB + wr^\top .
\end{align}
Next, denoting $\bfB_{\mu}^+ \coloneqq \bfB\bfD_{\mu}^{-1}$,
\begin{align*}
    \bfM(\bfM^\top \bfM+\mu \mathbf{I})^{-1}\bfM^\top  &= \bfM\bigg(\bfD_{\mu}^{-1} - \frac{\bfD_{\mu}^{-1}rr^\top \bfD_{\mu}^{-1}}{1 + d}\bigg)(\bfB^\top +rw^\top )\\
    &= \bfM\bigg(\bfD_{\mu}^{-1}\bfB^\top  + \bfD_{\mu}^{-1}rw^\top  - \frac{\bfD_{\mu}^{-1}rr^\top \bfD_{\mu}^{-1}\bfB^\top }{1+d} - \frac{\bfD_{\mu}^{-1}r(r^\top \bfD_{\mu}^{-1}r)w^\top }{1+d}\bigg) \\
    &= \bfM\bigg(\bfD_{\mu}^{-1}\bfB^\top  + \bfD_{\mu}^{-1}rw^\top  - \frac{\bfD_{\mu}^{-1}rr^\top \bfD_{\mu}^{-1}\bfB^\top }{1+d} - \frac{d}{1+d}\bfD_{\mu}^{-1}rw^\top \bigg)\nonumber\\
    &= (\bfB+wr^\top )\bigg(\bfD_{\mu}^{-1}\bfB^\top  + \frac{\bfD_{\mu}^{-1}rw^\top }{1+d} - \frac{\bfD_{\mu}^{-1}rr^\top \bfD_{\mu}^{-1}\bfB^\top }{1+d}\bigg)\\
    &= \bfB_{\mu}^+\bfB^\top  + \frac{\bfB_{\mu}^+rw^\top }{1+d} - \frac{\bfB_{\mu}^+rr^\top (\bfB_{\mu}^+)^\top }{1+d} + wr^\top (\bfB_{\mu}^+)^\top \\
     &\hspace{5cm} + \frac{d}{1+d}ww^\top  - \frac{d}{1+d}wr^\top (\bfB_{\mu}^+)^\top  \\
    &= \bfB_{\mu}^+\bfB^\top  + \frac{\bfB_{\mu}^+rw^\top }{1+d} - \frac{\bfB_{\mu}^+rr^\top (\bfB_{\mu}^+)^\top }{1+d} +  \frac{wr^\top (\bfB_{\mu}^+)^\top }{1+d} + \frac{d}{1+d}ww^\top  \\
    &= \bfB_{\mu}^+\bfB^\top  + \frac{d}{1+d}ww^\top  - \frac{(\bfB_{\mu}^+r - w)(\bfB_{\mu}^+r-w)^\top  - ww^\top }{1+d} \\
    &= \bfB_{\mu}^+\bfB^\top  + ww^\top  - \frac{(\bfB_{\mu}^+r - w)(\bfB_{\mu}^+r-w)^\top }{1+d},
\end{align*}
which concludes the proof.
\end{proof}

\begin{proof}[Proof of Lemma \ref{lemma:simplify}] Observe first that $\bfB=\pi^{\perp}_{v}\pi^{\perp}_{\mathbbm{1}}\bfV$, with $\pi^{\perp}_{v}\pi^{\perp}_{\mathbbm{1}} = \mathbf{I}_{n+1} - (n+1)^{-1}\mathbbm{1}\mathbbm{1}^\top - ww^\top$, since $w^\top\mathbbm{1} = 0$. Now, denoting $\mathbbm{1}_n\in\mathbb{R}^n$ the vector made up of ones (the difference with $\mathbbm{1}$ is that $\mathbbm{1}\in\mathbb{R}^{n+1}$), a simple computation shows that
\begin{align}
    ww^\top = \begin{pmatrix}
             \frac{1}{n(n+1)}\mathbbm{1}_n\mathbbm{1}_n^\top & -\frac{1}{n+1}\mathbbm{1}_n\\
             -\frac{1}{n+1}\mathbbm{1}_n^\top & \frac{n}{n+1}
    \end{pmatrix} = \frac{1}{n+1}\begin{pmatrix}
             \frac{1}{n}\mathbbm{1}_n\mathbbm{1}_n^\top & -\mathbbm{1}_n\\
             -\mathbbm{1}_n^\top & n
    \end{pmatrix}.
\end{align}
Therefore,
\begin{align}
    \mathbf{I}_{n+1} - \frac{1}{n+1}\mathbbm{1}\mathbbm{1}^\top - ww^\top &= \begin{pmatrix}
             \mathbf{I}_{n} & \mathbf{0}_{n,1}\\
             \mathbf{0}_{1,n} & 1
    \end{pmatrix} - \frac{1}{n+1}\begin{pmatrix}
             \mathbbm{1}_n\mathbbm{1}_n^\top & \mathbbm{1}_n \\
             \mathbbm{1}_n^\top & 1 
    \end{pmatrix} -\frac{1}{n+1}\begin{pmatrix}
             \frac{1}{n}\mathbbm{1}_n\mathbbm{1}_n^\top & -\mathbbm{1}_n\\
             -\mathbbm{1}_n^\top & n
    \end{pmatrix}\nonumber\\
    &=\begin{pmatrix}
             \mathbf{I}_n - \frac{1}{n}\mathbbm{1}_n\mathbbm{1}_n^\top & \mathbf{0}_{n,1}\\
             \mathbf{0}_{1,n} & 0
    \end{pmatrix}.
\end{align}
In particular, writing $\bfV$ blockwise as $\bfV^\top = (\bfV_0^\top \ \bfV_{n+1}^\top)$, where $\bfV_0\in\calM_{n,p}$ and $\bfV_{n+1}\in\calM_{1,p}$, then
\begin{align}
    \bfB = \bigg(\mathbf{I}_{n+1} - \frac{1}{n+1}\mathbbm{1}\mathbbm{1}^\top - ww^\top\bigg)\bfV = \begin{pmatrix}
            (\mathbf{I}_n - \frac{1}{n}\mathbbm{1}_n\mathbbm{1}_n^\top)\bfV_0 \\ \mathbf{0}_{1,n+1}
    \end{pmatrix},
\end{align}
which finishes the proof of equation \eqref{eq:B_ij}.
As a consequence, the last coordinate of $\bfD_{\mu}^{+}r$ is null: $(\bfD_{\lambda}^{+}r)_{n+1} = \big(\bfB\bfD_{\mu}^{-1}r\big)_{n+1} = (e_{n+1}^\top\bfB)\bfD_{\mu}^{-1}r = 0$.
Thus,
\begin{align}
    \big(bb^\top \big)_{n+1,n+1} = \Big((\bfB_{\mu}^+r-w)(\bfB_{\mu}^+r-w)^\top \Big)_{n+1,n+1} = \Big((\bfB_{\mu}^+r-w)_{n+1}\Big)^2 = w_{n+1}^2 = \frac{n}{n+1}.
\end{align}
Using the same reasoning, $(\bfB_{\mu}^+\bfB^\top )_{n+1,n+1} = 0$ and thus, $\bfC_{n+1,n+1} = w_{n+1}^2 = n/(n+1)$.
\end{proof}
\begin{proof}[Proof of Lemma \ref{lemma:expr_bn_dn_rn}]
Given $z\in\mathbb{R}^\ell$, we apply Lemma \ref{lemma:scores_transpose} to $n^{-1}\bfS(z)$ (equation \eqref{eq:bad_score}) by setting $\bfA = \bfW(0)$, $u = \|v\|\bfL z$, $w = v/\|v\|$ and $\mu = n\lambda$. We thus express $ n^{-1}\bfS(z)$ in the form
\begin{align*}
    n^{-1}\bfS(z) = \bfC - \frac{bb^\top}{1+d},
\end{align*}
where $\bfC, b$ and $d$ are given in Lemma \ref{lemma:scores_transpose}.
We next apply Lemma \ref{lemma:simplify}, using that $\bfA = \bfW(0) = \pi_\mathbbm{1}^\perp\bfV(0)$. The matrix $\bfB$ is thus given by
\begin{align}
    \bfB &= \begin{pmatrix}
         \bfB_n \\ 
         \mathbf{0}_{1,p} 
\end{pmatrix} \in \calM_{n+1,p},\label{eq:B_block}
\end{align}
where $\bfB_n$ is given in equation \eqref{eq:def_Bn}.
From equation \eqref{eq:B_block} and Lemma \ref{lemma:scores_transpose}, we obtain the expression of $\bfC_n$ (equation \eqref{eq:decomp_Cn}). Next, from Lemma \ref{lemma:scores_transpose}, $d_n(z) = r_n(z)^\top \bfD_{n\lambda}^{-1}r_n(z)$, where $r_n(z) = \|v\|\bfL z + \bfW(0)^\top w$. We observe that $\bfD_{n\lambda}^{-1} = n^{-1} \widehat{\mathbf{\Sigma}}_{n,\lambda}^{-1}$, which yields the expression of $d_n(z)$. The same observation and equation \eqref{eq:B_block} yield the expression of $b^n(z)$. Finally, we simplify $r_n(z)$, as
\begin{align}
   r_n(z) &= \|v\| \bfL z + W(0)^\top w = \|v\|\begin{pmatrix}
         0 \\  
         z
\end{pmatrix} + \|v\|^{-1}\bfV(0)^\top\pi_\mathbbm{1}^\perp v\nonumber \\
&= \|v\|\begin{pmatrix}
         0 \\  
         z
\end{pmatrix} + \|v\|^{-1}\bfV(0)^\top v = \|v\|\begin{pmatrix}
         0 \\  
         z
\end{pmatrix} + \|v\|^{-1}\begin{pmatrix}
         X_{n+1} - (n+1)^{-1}\sum_{i=1}^{n+1}X_i \\
         - (n+1)^{-1}\sum_{i=1}^{n}R_i
\end{pmatrix} \nonumber \\
&= \|v\|\bigg[\begin{pmatrix}
         0 \\  
         z
\end{pmatrix} + \|v\|^{-2}\begin{pmatrix}
         X_{n+1} - (n+1)^{-1}\sum_{i=1}^{n+1}X_i \\
         - (n+1)^{-1}\sum_{i=1}^{n}R_i
\end{pmatrix}\bigg].
\end{align}
Equation \eqref{eq:expr_rn} is finally obtained by further noticing that
\begin{align}
    \|v\|^{-2}\bigg(X_{n+1} - \frac{1}{n+1}\sum_{i=1}^{n+1}X_i\bigg) &= \frac{n+1}{n}\bigg(\frac{n}{n+1}X_{n+1} - \frac{1}{n+1}\sum_{i=1}^{n}X_i\bigg) = X_{n+1} - \frac{1}{n}\sum_{i=1}^{n}X_i,
\end{align}
and performing a similar computation for $(n+1)^{-1}\sum_{i=1}^{n}R_i$.
\end{proof}
\iain{
\begin{proof}[Proof of Lemma \ref{lemma:expr_bnx_dnx_rnx}]
The expression for $\bfS_X$ is obtained using that
\begin{align*}
    \bfX = (X_1 \ \ldots \ X_n 0_{k,1})^\top + e_{n+1}X_{n+1}^\top.
\end{align*}
Hence we can obtain an expression for $\bfS_X$ similar to equation \eqref{eq:scores_ simple} by performing the three substitutions
\begin{align}\label{eq:substitution}
    \bfV \xleftrightarrow[]{} \bfX, \ \ \ \bfL \xleftrightarrow[]{} \bfI_k, \ \ \ z \xleftrightarrow[]{} X_{n+1}.
\end{align}
With the substitution \eqref{eq:substitution}, the vector $r_n(z)$ in Lemma \ref{lemma:expr_bn_dn_rn} is replaced with
\begin{align*}
    r_n = \|v\|(X_{n+1}-\overline{X}_n) = \|v\|X_{n+1}^\rmc.
\end{align*}
Next, $\bfB_n$ is replaced with its $n\times k$ left sub-block which we denote by $\bfB_n^X$, which verifies
\begin{align*}
    (\bfB_n^X)_{ij} = (X_i)_j - \frac{1}{n}\sum_{k=1}^n(X_k)_j, \ \ \ 1\leq i \leq n, \ \ \ 1\leq j \leq k.
\end{align*}
The matrix $\widehat{\mathbf{\Sigma}}_{n,\lambda}$ is next replaced with its upper $k\times k$ sub-block $\widehat{\mathbf{\Sigma}}_{n,\lambda}^{11}$, since
\begin{align*}
    \widehat{\mathbf{\Sigma}}_{n,\lambda}^{11} = \frac{1}{n}(\bfB_n^X)^\top \bfB_n^X + \lambda \bfI_k.
\end{align*}
Next, $d_n(z)$ is replaced with
\begin{align*}
    d_n^X = \frac{1}{n}r_n^\top(\widehat{\mathbf{\Sigma}}_{n,\lambda}^{11})^{-1}r_n = \frac{\|v\|^2}{n}(X_{n+1}^\rmc)^\top (\widehat{\mathbf{\Sigma}}_{n,\lambda}^{11})^{-1}X_{n+1}^\rmc = \frac{1}{n+1}(X_{n+1}^\rmc)^\top (\widehat{\mathbf{\Sigma}}_{n,\lambda}^{11})^{-1}X_{n+1}^\rmc.
\end{align*}
The vector $b_n(z)$ is replaced with
\begin{align*}
    b_X^n = \frac{1}{n}\begin{pmatrix}
             \bfB_n^X(\widehat{\mathbf{\Sigma}}_{n,\lambda}^{11})^{-1}r_n \\ 0
    \end{pmatrix} - w = \frac{\|v\|}{n}\begin{pmatrix}
             \bfB_n^X(\widehat{\mathbf{\Sigma}}_{n,\lambda}^{11})^{-1}X_{n+1}^\rmc \\ 0
    \end{pmatrix} - w.
\end{align*}
Finally, the matrix $\bfP_{n,\lambda}$ is replaced with
\begin{align*}
    \bfP_{n,\lambda}^{XX} \coloneqq \bfB_n^X\Big((\bfB_n^X)^\top\bfB_n^X + n \lambda \bfI_k\Big)^{-1}(\bfB_n^X)^\top = \frac{1}{n}\bfB_n^X(\widehat{\mathbf{\mathbf{\Sigma}}}_{n,\lambda}^{11})^{-1}(\bfB_n^X)^\top.
\end{align*}
This finishes the proof.
\end{proof}
}
\begin{proof}[Proof of Lemma \ref{lemma:approx_score}]
First, observe from Lemma \ref{lemma:simplify} that
\begin{align}\label{eq:snp1_simplify}
   \frac{1}{n} S_{n+1}(z) = \frac{n}{n+1}\bigg(1-\frac{1}{1 + d_n(z)}\bigg).
\end{align}
Now, recall that the standard conformal region $\calC_{\alpha}^n$ corresponding to our score is given by
\begin{align}
    \calC_{\alpha}^n = \{ z\in\mathbb{R}^{\ell}: S_{n+1}(z) \leq S_{(n_{\alpha})}(z) \},
\end{align}
where $S_{(n_{\alpha})}(z)$ is the order statistic of order $n_{\alpha}$ of the $n$-tuple $(S_1(z), \dots , S_n(z))$.
Next, denote $C_{(n_{\alpha})}$ the order statistic of the $n$-tuple $((\bfC_n)_{11}, \dots , (\bfC_n)_{nn})$. Observe that, by definition, at least $100\times (1-\alpha)\frac{n+1}{n}\%$ values of this tuple are less than or equal to $C_{(n_{\alpha})}$; but for each of such values $(\bfC_n)_{ii}$ and for all $z$, we have $S_i(z) \leq n(\bfC_n)_{ii}$ (equation \eqref{eq:scores_ simple}). Hence, at least $100\times (1-\alpha)\frac{n+1}{n}\%$ values of $(S_1(z), \dots , S_n(z))$ are less or equal than $nC_{(n_{\alpha})}$. Thus, $S_{(n_{\alpha})}(z) \leq nC_{(n_{\alpha})}$, and
\begin{align}\label{eq:ctilde_interm}
    \calC_{\alpha}^n \subset \{ z \in \mathbb{R}^{\ell}: S_{n+1}(z) \leq nC_{(n_{\alpha})} \}.
\end{align}
To finish, recall that $(\bfC_n)_{ii} = p_{i,n} + 1/n(n+1)$: in particular, $nC_{(n_{\alpha})} = q_{n,\alpha} + 1/(n+1)$. This shows that the set on the right hand side of equation \eqref{eq:ctilde_interm} is $\widetilde{\calC}_{\alpha}^n$.
\end{proof}
\begin{proof}[Proof of Proposition \ref{prop:c_a_is_e_a_n}] We start by rewriting the equation defining $\widetilde{\calC}_{\alpha}^n$. Starting from equations \eqref{eq:snp1_simplify} and \eqref{eq:ctilde_interm},
\begin{align}\label{eq:start_eq_cna}
    z\in \widetilde{\calC}_{\alpha}^n \iff -\bigg(\frac{n^2}{n+1}\bigg)\frac{1}{1 + d_n(z)} \leq q_{n,\alpha} + \frac{1}{n+1} - \frac{n^2}{n+1} = q_{n,\alpha}-(n-1).
\end{align}
From the equation above, if $q_{n,\alpha} \geq n-1$, then $\widetilde{\calC}_{\alpha}^n = \Rbb^{\ell}$. Assume now that $q_{n,\alpha} < n-1$, then
\begin{align}\label{eq:rewrite_cna}
    z\in \widetilde{\calC}_{\alpha}^n \iff 1 + d_n(z) \leq \bigg(\frac{n^2}{n+1}\bigg)\frac{1}{{n-1} - {q_{n,\alpha}} } \iff d_n(z) \leq \frac{q_{n,\alpha} + \frac{1}{n+1}}{{n} - 1 -{q_{n,\alpha}} }.
\end{align}
Now, recall from Lemma \ref{lemma:expr_bn_dn_rn} that $d_n(z) = n^{-1}r_n(z)^\top \widehat{\mathbf{\Sigma}}_{n,\lambda}^{-1}r_n(z)$. Following equation \eqref{eq:expr_rn} and the notation $X_{n+1}^\rmc \coloneqq X_{n+1}-\overline{X}_n$, we set $z_\rmc \coloneqq z - \overline{R}_n$, so that
\begin{align}\label{eq:recover_rnz}
    r_n(z) = \|v\|\begin{pmatrix}
             X_{n+1}^\rmc \\
             z_\rmc
    \end{pmatrix}.
\end{align}
Using that $\|v\|^2 = n/(n+1)$ and equation \eqref{eq:apply_block_lemma}, $d_n(z)$ is rewritten as
\begin{align}
    (n+1) d_n(z) &= \frac{n+1}{n} r_n(z)^\top \widehat{\mathbf{\Sigma}}_{n,\lambda}^{-1}r_n(z) = \frac{n+1}{n}\|v\|^2\begin{pmatrix}
         X_{n+1}^\rmc\\z_\rmc
    \end{pmatrix}^\top\widehat{\mathbf{\Sigma}}_{n,\lambda}^{-1}\begin{pmatrix}
         X_{n+1}^\rmc\\z_\rmc
    \end{pmatrix} = \begin{pmatrix}
         X_{n+1}^\rmc\\z_\rmc
    \end{pmatrix}^\top\widehat{\mathbf{\Sigma}}_{n,\lambda}^{-1}\begin{pmatrix}
         X_{n+1}^\rmc\\z_\rmc
    \end{pmatrix}
    \\
    &= (z_\rmc - z_0)^\top \boldcalA_n^{-1}(z_\rmc-z_0) + (X_{n+1}^\rmc)^\top(\widehat{\mathbf{\Sigma}}_{n,\lambda}^{11})^{-1}X_{n+1}^\rmc,
\end{align}
with $\boldcalA_n = \widehat{\mathbf{\Sigma}}_{n,\lambda}/\widehat{\mathbf{\Sigma}}_{n,\lambda}^{11}$ and $z_0 = \widehat{\mathbf{\Sigma}}_{n}^{21}(\widehat{\mathbf{\Sigma}}_{n,\lambda}^{11})^{-1}X_{n+1}^\rmc$.
Going back to equation \eqref{eq:rewrite_cna}, the final equation of $\widetilde{\calC}_{\alpha}^n$ is obtained by setting $ Z_0^n \coloneqq z_0 + \overline{R}_n$, and writing
\begin{align*}
   z\in \widetilde{\calC}_{\alpha}^n \iff (z - Z_0^n)^\top \boldcalA_n^{-1}(z-z_0^n) \leq (n+1)\frac{q_{n,\alpha} + 1/(n+1)}{{n} - 1 -{q_{n,\alpha}} } - (X_{n+1}^\rmc)^\top(\widehat{\mathbf{\Sigma}}_{n,\lambda}^{11})^{-1}X_{n+1}^\rmc.
\end{align*}
The quantile term is simplified as
\begin{align}
    (n+1)\frac{q_{n,\alpha} + 1/(n+1)}{{n} - 1 -{q_{n,\alpha}} } = \frac{(n+1)q_{n,\alpha} + 1 + n -n}{{n} - 1 -{q_{n,\alpha}} } = \frac{n(q_{n,\alpha}+1)}{n-1-q_{n,\alpha}} - 1 = \frac{q_{n,\alpha}+1}{1-(q_{n,\alpha}+1)/n} - 1,
\end{align}
which finishes to show that $\widetilde{C}_{\alpha}^n = \calE_{\alpha}^n$.
\end{proof}
\begin{proof}[Proof of Proposition \ref{prop:bound}]
\adri{The core idea for this result is the observation that for all $z\in K$,
\begin{align}
    q_{n,\alpha}  - \sup_{z\in K}\max_{1\leq i \leq n}nb_i^n(z)^2 \leq S_{(n_{\alpha})}(z)-1/(n+1)\leq q_{n,\alpha},
\end{align}
and the fact that under our moment assumptions,} 
\begin{align}
 b^n(K) \coloneqq\sup_{z\in K}\max_{1\leq i \leq n}nb_i^n(z)^2 \xlongrightarrow[n\rightarrow \infty]{\mathbb{P}} 0.
\end{align}
Let us prove the equation above.
For this, denote $Q_n(z)  \coloneqq\big(\bfB_n \widehat{\mathbf{\Sigma}}_{n,\lambda}^{-1}r_n(z)\big)_1 = (V_1^\rmc)^\top\widehat{\mathbf{\Sigma}}_{n,\lambda}^{-1}r_n(z)$, where $V_1^\rmc = V_1 - n^{-1}\sum_{i=1}^nV_i$ and $r_n(z)$ is given in equation \eqref{eq:expr_rn}. Next, following equation \eqref{eq:d_n_b_n}, write that for all $z\in\mathbb{R}^{\ell}$,
\begin{align}
    b_1^n(z)  &= \frac{Q_n(z)}{n} - w_1 = \frac{Q_n(z)}{n} - \frac{1}{\sqrt{n(n+1)}}, \nonumber\\
    \big(nb_1^n(z)^2\big)^q &= \big(\sqrt{n}b_1^n(z)\big)^{2q} = \bigg(\frac{Q_n(z)}{\sqrt{n}}-\frac{1}{\sqrt{n+1}}\bigg)^{2q} \lesssim \frac{Q_n(z)^{2q}}{n^q} + \frac{1}{n^q}.\label{eq:conv_3p}
\end{align}
Above, we used that $|a+b|^{2q}\leq 2^{q-1}(|a|^{2q}+|b|^{2q})$. Observe now that almost surely, because of the continuity of maps of the form $z\mapsto \bfA z + b$,
\begin{align}
    b^n(K) = \sup_{z\in K} \max_{1\leq i \leq n}nb_i^n(z)^2 = \sup_{z\in K\cap \mathbb{Q}^{\ell}} \max_{1\leq i \leq n}nb_i^n(z)^2.
\end{align}
Hence, $b^n(K)$ is a well-defined random variable (we apply the same reasoning to $\sup_{z\in K}|Q_n(z)|$). Next (explanation below), 
\begin{align}
    \mathbb{E}[b^n(K)]^q = \mathbb{E}\Big[\sup_{z\in K}\max_{1\leq i \leq n}nb_i^n(z)^2\Big]^q &\leq \mathbb{E}\Big[\sup_{z\in K}\max_{1\leq i \leq n}\big(nb_i^n(z)^2\big)^q\Big] \leq \mathbb{E}\bigg[\sup_{z\in K}\sum_{i=1}^n\big(nb_i^n(z)^2\big)^q\bigg]\nonumber\\
    &\leq \mathbb{E}\bigg[\sum_{i=1}^n\sup_{z\in K}\big(nb_i^n(z)^2\big)^q\bigg] = n\mathbb{E}\Big[\sup_{z\in K}\big(nb_1^n(z)^2\big)^q\Big]\nonumber\\
    &\lesssim \frac{1}{n^{q-1}}\mathbb{E}\Big[\sup_{z\in K}|Q_n(z)|^{2q}\Big] + \frac{1}{n^{q-1}}. \label{eq:use_3p}
\end{align}
We used Jensen's inequality in the first inequality, and equation \eqref{eq:conv_3p} in equation \eqref{eq:use_3p}.
We now prove that the assumption yields that $\sup_n \mathbb{E}[\sup_{z\in K}|Q_n(z)|^{2q}]<+\infty$. For this, use equation \eqref{eq:expr_rn} to decompose $r_n(z)$ as
\begin{align}\label{eq:decomp_rn}
    r_n(z) = \|v\|(\bfL z + v_n), \ \ v_n = \begin{pmatrix}
             X_{n+1} \\ 0
    \end{pmatrix} - \frac{1}{n}\sum_{i=1}^nV_i.
\end{align}
(Be wary that $v_n$ above has no link with $v$.) Next (explanation below),
\begin{align}
    |Q_n(z)|^2 &= \Big((V_1^\rmc)^\top\widehat{\mathbf{\Sigma}}_{n,\lambda}^{-1}r_n(z)\Big)^2 \leq (V_1^\rmc)^\top\widehat{\mathbf{\Sigma}}_{n,\lambda}^{-1}V_1^\rmc \times r_n(z)^\top\widehat{\mathbf{\Sigma}}_{n,\lambda}^{-1}r_n(z)\label{eq:cs_matrix} \\
    &\leq \lambda^{-2}\|V_1^\rmc\|^2\|r_n(z)\|^2 \lesssim\|V_1^\rmc\|^2(\|z\|^2 + \|v_n\|^2),\label{eq:sn_interm} \\
    \sup_{z\in K}|Q_n(z)|^2 &\lesssim\|V_1^\rmc\|^2\Big(\sup_{z\in K}\|z\|^2 + \|v_n\|^2\Big).
    \end{align}
    Above, we applied the Cauchy-Schwarz inequality for the inner product $(u,v) \mapsto u^\top \bfA v$ where $\bfA$ is symmetric PSD in equation \eqref{eq:cs_matrix}. In equation \eqref{eq:sn_interm}, we used that $\widehat{\mathbf{\Sigma}}_{n,\lambda}^{-1} \preccurlyeq \lambda^{-1}I$, $\|v\|\leq 1$ and $\|r_n(z)\|^2 \leq 2(\|z\|^2 + \|v_n\|^2)$, the latter inequality being easily deduced from equation \eqref{eq:decomp_rn}.
    Applying the Cauchy-Schwarz inequality for the expectation and the fact that $|a+b|^{2q}\lesssim |a|^{2q}+|b|^{2q}$,
    \begin{align}
    \mathbb{E}\Big[\sup_{z\in K}|Q_n(z)|^{2q}\Big] &\lesssim\mathbb{E}[\|V_1^\rmc\|^{4q}]^{1/2}\mathbb{E}\bigg[\Big(\sup_{z\in K}\|z\|^2 + \|v_n\|^2\Big)^{2q}\bigg]^{1/2}\nonumber\\
    &\lesssim\mathbb{E}[\|V_1^\rmc\|^{4q}]^{1/2}\bigg(\sup_{z\in K}\|z\|^{4q} +  \mathbb{E}\big[\|v_n\|^{4q}\big]\bigg)^{1/2}. \label{eq:prod_esperance}
\end{align}
Similarly, from the triangle inequality and the convexity of $(\cdot)^{4q}$, the right-hand side of equation \eqref{eq:prod_esperance} is bounded as
\begin{align}
   \mathbb{E}[\|V_1^\rmc\|^{4q}] &\lesssim \mathbb{E}[\|V_1\|^{4q}] + \mathbb{E}\bigg[\Big\|\frac{1}{n}\sum_{i=1}^nV_i\Big\|^{4q}\bigg] \leq 2\mathbb{E}[\|V_1\|^{4q}],\nonumber\\
    \mathbb{E}[\|v_n\|^{4q}] &\lesssim \mathbb{E}[\|X_{n+1}\|^{4q}] + \mathbb{E}\bigg[\Big\|\frac{1}{n}\sum_{i=1}^nV_i\Big\|^{4q}\bigg]  \leq 2\mathbb{E}[\|V_1\|^{4q}].\nonumber
\end{align}
Hence, from the assumption, $\sup_n\mathbb{E}\Big[\sup_{z\in K}|Q_n(z)|^{2q}\Big] < + \infty$ and from equation \eqref{eq:use_3p}, $b^n(K) \rightarrow 0$ in $L^1(\mathbb{P})$. Thus,
\begin{align}\label{eq:cv_bnK}
    b^n(K)  \xlongrightarrow[n\rightarrow \infty]{\mathbb{P}} 0.
\end{align}
We now prove equation \eqref{eq:vol_surplus}. For this, we introduce the increasing function $f_n$ and the nonnegative random variable $T_n$, such that $\rho_{n,\alpha} = f_n(q_{n,\alpha}) - T_n$, following equation \eqref{eq:def_rho_n}. We also denote $\overline{q}_{n,\alpha}\coloneqq q_{n,\alpha} -b^n(K)$ $(\overline{q}_{n,\alpha}\leq q_{n,\alpha})$ and $\overline{\rho}_{n,\alpha} \coloneqq f_n(\overline{q}_{n,\alpha}) - T_n$ ($\overline{\rho}_{n,\alpha} \leq \rho_{n,\alpha}$). Finally we define $\overline{\calE}_{\alpha}^n$ as the following ellipsoid,
\begin{align}
    \overline{\calE}_{\alpha}^n \coloneqq \{z \in \mathbb{R}^{\ell}: (z - Z_0^n)^\top \boldcalA_n^{-1}(z - Z_0^n) \leq (\overline{\rho}_{n,\alpha})_+\}.
\end{align}
By copying the proof of Proposition \ref{prop:c_a_is_e_a_n}, $ \overline{\calE}_{\alpha}^n$ is also the set defined by $\{z\in\mathbb{R}^\ell : S_{n+1}(z) \leq \overline{q}_{n,\alpha} + (n+1)^{-1} \}$.
For all $z\in K$, the inequality $\overline{q}_{n,\alpha} + (n+1)^{-1} \leq S_{(n_{\alpha})}(z) \leq q_{n,\alpha} + (n+1)^{-1}$ then yields that
\begin{align}\label{eq:inclusions}
    (\overline{\calE}_{\alpha}^n \cap K) \subset (\calC_{\alpha}^n \cap K) \subset (\calE_{\alpha}^n\cap K).
\end{align}
We now study the associated volumes. First, given $s\geq 1$, recall that from the mean value theorem, $|x^s-y^s|\leq s\max (|x|,|y|)^{s-1}|x-y|, x,y \geq 0$, and observe that, \iain{because $f_n$ is increasing and, for $x,t \geq 0$, $(x-t)_+ \leq x$,}
\begin{align}\label{eq:control_max_fn}
    \max((\rho_{n,\alpha})_+, (\overline{\rho}_{n,\alpha})_+)^{s-1} \leq \max(f_n(q_{n,\alpha}),f_n(\overline{q}_{n,\alpha}))^{s-1} = f_n(q_{n,\alpha})^{s-1}.
\end{align}
If $\ell \geq 2$, this fact and equation \eqref{eq:inclusions} together imply that
\begin{align}
   0 \leq  \Vol\big((\calE_{\alpha}^n\setminus \calC_{\alpha}^n)\cap K\big) &\leq \Vol\big((\calE_{\alpha}^n\setminus \overline{\calE}_{\alpha}^n)\cap K\big) \leq \Vol(\calE_{\alpha}^n\setminus \overline{\calE}_{\alpha}^n) = \Vol(\calE_{\alpha}^n) - \Vol(\overline{\calE}_{\alpha}^n)\nonumber\\
    &\leq \det(\boldcalA_n)^{1/2}((\rho_{n,\alpha})_+^{\ell/2} - (\overline{\rho}_{n,\alpha})_+^{\ell/2})
     \leq \det(\boldcalA_n)^{1/2}\frac{\ell}{2}f_n(q_{n,\alpha})^{\ell/2-1}|(\rho_{n,\alpha})_+ - (\overline{\rho}_{n,\alpha})_+|\nonumber \\
    &\leq \det(\boldcalA_n)^{1/2}\frac{\ell}{2}f_n(q_{n,\alpha})^{\ell/2-1}|\rho_{n,\alpha} - \overline{\rho}_{n,\alpha}| \label{eq:lpz}\\
    &\leq \det(\boldcalA_n)^{1/2}\frac{\ell}{2}f_n(q_{n,\alpha})^{\ell/2-1}|f_n(q_{n,\alpha}) - f_n(\overline{q}_{n,\alpha})|.\label{eq:control_diff_vol}
\end{align}
(We used that $(\cdot)_+$ is $1$-Lipschitz in equation \eqref{eq:lpz}.)
But it is clear, from the definition of $f_n$ as well as Lemma \ref{lemma:cv_quantile}, that $f_n(q_{n,\alpha}) \xrightarrow[]{} q_{1-\alpha}^{\calE}$ in probability, while Lemma \ref{lemma:cv_quantile} and equation \eqref{eq:cv_bnK} imply that $f_n(\overline{q}_{n,\alpha}) \xrightarrow[]{} q_{1-\alpha}^{\calE}$ in probability. Combined with equation \eqref{eq:cv_An} and the continuity of $\det(\cdot)$, equation \eqref{eq:control_diff_vol} implies the desired equation \eqref{eq:vol_surplus}.
If $\ell = 1$, the same proof can be adapted using this time that $|x^{1/2}-y^{1/2}|\leq |x-y|^{1/2}, x,y \geq 0$.
\end{proof}
\begin{proof}[Proof of Lemma \ref{lemma:q_bound}] From equation \eqref{eq:decomp_Cn} and observing that $0 \preccurlyeq \bfP_{n,\lambda} \preccurlyeq \bfP_{n,0}$, we may first write that
\begin{align}\label{eq:sum_cii}
 \text{Tr}(\bfP_{n,\lambda}) \leq  \text{Tr}(\bfP_{n,0}) \leq p.
  \end{align}
 Now, denote $p_{(i)}$ the $\ith$ order statistic of $(p_{1,n}, \dots , p_{n,n})$, $p_{(1)} \leq \dots  \leq p_{(n)}$. Then
\begin{align}\label{eq:q_0}
\text{Tr}(\bfP_{n,\lambda}) =  \sum_{i=1}^n p_{i,n} =  \sum_{i=1}^n p_{(i)} \geq \sum_{i=n_{\alpha}}^{n} p_{(i)} \geq p_{(n_{\alpha})} (n +1 - n_{\alpha}) = \frac{q_{n,\alpha}}{n}(n +1 - n_{\alpha}).
\end{align}
Moreover, 
\begin{align}
n+1-n_{\alpha} &=  n+1 -\lceil (1-\alpha)(n+1)\rceil = n+1 + \lfloor - (1-\alpha)(n+1)\rfloor\nonumber \\
&= \lfloor n+1- (1-\alpha)(n+1)\rfloor = \lfloor \alpha (n+1) \rfloor.\label{eq:simpmify_nalpha}
\end{align}
Thus, combining equations \eqref{eq:sum_cii}, \eqref{eq:q_0} and \eqref{eq:simpmify_nalpha},
\begin{align}\label{eq:control_qnalpha}
    q_{n,\alpha} \leq \frac{np}{n+1-n_{\alpha}} = \frac{np}{\lfloor \alpha (n+1) \rfloor}.
\end{align}
Hence, a sufficient condition for $q_{n,\alpha} < n-1$ is $np/\lfloor \alpha (n+1) \rfloor < n-1$, that is,
\begin{align}
    \frac{np}{n-1} < \lfloor \alpha (n+1)\rfloor.
\end{align}
But this is equivalent to  $\lfloor {np}/({n-1})\rfloor +1 \leq \lfloor\alpha (n+1)\rfloor$, which is also equivalent to
\begin{align}
 \bigg\lfloor \frac{np}{n-1}\bigg\rfloor +1  = p + \bigg\lfloor \frac{p}{n-1}\bigg\rfloor + 1 &\leq \alpha (n+1).
\end{align}
Thus, if $p < n-1$, then $\lfloor {p}/(n-1)\rfloor = 0$ and we obtain the sufficient condition $\alpha \geq (p+1)/(n+1)$.
\end{proof}
\begin{proof}[Proof of Theorem \ref{thm:ellip_prime}]
\iain{The computations follow those of Proposition \ref{prop:c_a_is_e_a_n}. Using Lemmas \ref{lemma:scores_transpose} and \ref{lemma:simplify} on $\bfS(z)$ and $\bfS_X$, we rewrite the score matrix as
\begin{align}
    \frac{1}{n}\bfS'(z) &= \begin{pmatrix}
             \bfP_{n,\lambda}  &\mathbf{0}_{n,1} \\
             \mathbf{0}_{1,n}  &0 
    \end{pmatrix} + ww^\top - \frac{b^n(z)b^n(z)^\top }{1 + d_n(z)} - \Bigg(\begin{pmatrix}
             \bfP_{n,\lambda}^{XX}  &\mathbf{0}_{n,1} \\
             \mathbf{0}_{1,n}  &0 
    \end{pmatrix} + ww^\top - \frac{b_X^n(b_X^n)^\top }{1 + d_n^X}\Bigg)\nonumber \\
    &= \begin{pmatrix}
             \bfP_{n,\lambda} - \bfP_{n,\lambda}^{XX}  &\mathbf{0}_{n,1} \\
             \mathbf{0}_{1,n}  &0 
    \end{pmatrix} - \frac{b^n(z)b^n(z)^\top }{1 + d_n(z)} + \frac{b_X^n(b_X^n)^\top }{1 + d_n^X}.
\end{align}
In particular,
\begin{align}
    \forall i\in\{1, \ldots, n\}, \ \ \ S_i'(z) &= n(\bfP_{n,\lambda} - \bfP_{n,\lambda}^{XX})_{ii} + n\frac{(b_X^n)_i^2}{1+d_n^X} - n\frac{b^n(z)_i^2}{1+d_n(z)}, \\
    S_{n+1}'(z) &= \frac{n^2}{n+1}\bigg(\frac{1}{1+d_n^X} - \frac{1}{1+d_n(z)}\bigg). \label{eq:s_np1_prime}
\end{align}
Equation \eqref{eq:s_np1_prime} is a consequence of Lemma \ref{lemma:simplify} : $b^n(z)_{n+1}^2 = (b_X^n)_{n+1}^2 = n/(n+1)$. Now, since $d_n(z)\geq 0$, the following uniform bound in $z$ holds for all $i\in\{1, \ldots, n\}$ :
\begin{align}
     S_i'(z) \leq n(\bfP_{n,\lambda} - \bfP_{n,\lambda}^{XX})_{ii} + n\frac{(b_X^n)_i^2}{1+d_n^X} = np_{i,n}'.
\end{align}
Using the same arguments as in Lemma \ref{lemma:approx_score}, we have that $S_{(n_\alpha)}'(z) \leq np_{(n_\alpha)}' = q_{n,\alpha}'$ (see equation \eqref{eq:def_pn_qn_prime} for $q_{n,\alpha}'$), hence
\begin{align}\label{eq:def_calf_n_alpha}
    \{z : S_{n+1}'(z) \leq S_{(n_\alpha)}'(z) \} \subset  \{z : S_{n+1}'(z) \leq q_{n,\alpha}' \} \eqqcolon \calF_\alpha^n.
\end{align}
We now show that the set $\calF_\alpha^n$ above corresponds to that of Theorem \ref{thm:ellip_prime}. First, observe from equation \eqref{eq:apply_block_lemma} and the definitions of $d_n(z)$ and $d_n^X$ that
\begin{align*}
    d_n(z) = \frac{1}{n+1}(z - Z_0^n)^\top \boldcalA_n^{-1}(z-Z_0^n) + d_n^X \geq 0.
\end{align*}
Since $d_n(z)$ is quadratic in $z$ with $\min_zd_n(z) = d_n(Z_0^n) = d_n^X$, equation \eqref{eq:s_np1_prime} implies that
\begin{align*}
    \max_{z\in\mathbb{R}^\ell}S_{n+1}'(z) = \bigg(\frac{n^2}{n+1}\bigg)\frac{1}{1+d_n^X} = \frac{n}{t_{n}}.
\end{align*}
Hence, from the definition of $\calF_\alpha^n$ (equation \eqref{eq:def_calf_n_alpha}), $\calF_\alpha^n = \mathbb{R}^\ell$ if and only if $n/t_{n} \leq q_{n,\alpha}'$, which settles the case where $\calF_\alpha^n = \mathbb{R}^\ell$. Assume now that $n > q_{n,\alpha}'t_{n}$. Then,
\begin{align*}
z \in \calF_\alpha^n &\iff \frac{n^2}{n+1}\bigg(\frac{1}{1+d_n^X} - \frac{1}{1+d_n(z)}\bigg) \leq q_{n,\alpha}' \iff \frac{1}{1+d_n(z)} \geq \frac{1}{1+d_n^X} - \frac{n+1}{n^2}q_{n,\alpha}' \\
&\iff 1 +d_n(z) \leq \frac{1}{\frac{1}{1+d_n^X} - \frac{n+1}{n^2}q_{n,\alpha}'} = \frac{1+d_n^X}{1 - \frac{n+1}{n^2}(1+d_n^X)q_{n,\alpha}'} = \frac{1+d_n^X}{1-{t_{n}q_{n,\alpha}'}/n} \\
&\iff (n+1)d_n(z) \leq (n+1)\bigg(\frac{1+d_n^X}{1-{t_{n}q_{n,\alpha}'}/n} -1\bigg)\\
&\iff (z_\rmc - z_0)^\top \boldcalA_n^{-1}(z_\rmc-z_0) \leq (n+1)\bigg(\frac{1+d_n^X}{1-{t_{n}q_{n,\alpha}'}/n} -1\bigg) - (n+1)d_n^X \eqqcolon \rho_{n,\alpha}'.
\end{align*}
Moreover, the right-hand side $\rho_{n,\alpha}'$ above can be rewritten as
\begin{align*}
    \rho_{n,\alpha}' &= (n+1)\bigg(\frac{1+d_n^X}{1-{t_{n}q_{n,\alpha}'}/n} -1\bigg) - (n+1)d_n^X = (n+1)\bigg(\frac{1+d_n^X}{1-{t_{n}q_{n,\alpha}'}/n} -1-d_n^X\bigg)\\
    &= (n+1)(1+d_n^X)\bigg(\frac{1}{1-{t_{n}q_{n,\alpha}'}/n} -1\bigg) =(n+1)(1+d_n^X)\frac{t_{n}q_{n,\alpha}'/n}{1-{t_{n}q_{n,\alpha}'}/n} = t_{n}^2\frac{q_{n,\alpha}'}{1-{t_{n}q_{n,\alpha}'}/n},
    \end{align*}
    which finishes the proof.
}
\end{proof}
\begin{proof}[Proof of Proposition \ref{prop:calF_full_space}]
\iain{
We first prove that $S_i'(z) \geq 0$ for all $1\leq i \leq n+1$. We first introduce the matrices $\bfF$ and $\bfG(z)$ such that $\bfW(z) = \pi_\mathbbm{1}^\perp(\bfX\ \bfR(z)) = (\pi_\mathbbm{1}^\perp\bfX \ \pi_\mathbbm{1}^\perp\bfR(z)) \eqqcolon (\bfF \ \bfG(z))$ and $\widehat{\mathbf{\Sigma}}_{\lambda}$ in equation \eqref{eq:def_sig_hat} as
\begin{align*}
    \widehat{\mathbf{\Sigma}}_{\lambda} = \frac{1}{n}\bfW(z)^\top\bfW(z) + \lambda \bfI_p = \begin{pmatrix}
                \widehat{\mathbf{\Sigma}}_{\lambda}^{11}   & \widehat{\mathbf{\Sigma}}^{12} \\
                \widehat{\mathbf{\Sigma}}^{21}    &\widehat{\mathbf{\Sigma}}_{\lambda}^{22}
\end{pmatrix}, \ \ \ \text{with} \ \ \ \bfS(z) = (\bfF \ \bfG(z)) \ \widehat{\mathbf{\Sigma}}_{\lambda}^{-1} \begin{pmatrix}
         &\bfF^\top \\
         &\bfG(z)
\end{pmatrix}.
\end{align*}
Generalizing equation \eqref{eq:apply_block_lemma} to matrices (using the block inversion lemma \ref{lemma:block_inversion}) and observing that $\bfS_X = \bfF(\widehat{\mathbf{\Sigma}}_{\lambda}^{11})^{-1}\bfF^\top$,
\begin{align*}
    \bfS(z) &= \bfF(\widehat{\mathbf{\Sigma}}_{\lambda}^{11})^{-1}\bfF^\top + (\bfG(z) -  \widehat{\mathbf{\Sigma}}^{21} (\widehat{\mathbf{\Sigma}}_{\lambda}^{11})^{-1}\bfF)(\widehat{\mathbf{\Sigma}}_\lambda/\widehat{\mathbf{\Sigma}}_\lambda^{11})^{-1}(\bfG(z) -(\widehat{\mathbf{\Sigma}}_{\lambda}^{11})^{-1}\bfF)^\top, \\
   \bfS(z)-\bfS_X &= (\bfG(z) -  \widehat{\mathbf{\Sigma}}^{21} (\widehat{\mathbf{\Sigma}}_{\lambda}^{11})^{-1}\bfX)(\widehat{\mathbf{\Sigma}}_\lambda/\widehat{\mathbf{\Sigma}}_\lambda^{11})^{-1}(\bfG(z) -(\widehat{\mathbf{\Sigma}}_{\lambda}^{11})^{-1}\bfX)^\top \succcurlyeq 0.
\end{align*}
Hence, for all $i\in \{1, \ldots, n+1\},$
\begin{align*}
    S_i'(z) = (\bfS(z)-\bfS_X)_{ii} = e_i^\top (\bfS(z)-\bfS_X)e_i \geq 0.
\end{align*}
In particular, for all $i\in \{1,\ldots,n\}$,
\begin{align*}
    np_{i,n}' \geq S_i'(z) \geq 0,
\end{align*}
and as a consequence, $q_{n,\alpha}' \geq 0$. From the definition of $\rho_{n,\alpha}'$, we deduce that $ \rho_{n,\alpha}' \geq 0$. Finally, since $d_n(Z_0^n) = d_n^X$,
we have from equation \eqref{eq:s_np1_prime} that
\begin{align*}
 S_{n+1}'(Z_0^n) = 0 \leq q_{n,\alpha}',   
\end{align*}
and thus, from the definition of $\calF_\alpha^n$ (equation \eqref{eq:def_calf_n_alpha}), $Z_0^n \in \calF_\alpha^n$.
We now study the quantity $\mathbb{P}(\calF_\alpha^n = \mathbb{R}^\ell)$. We first write
\begin{align*}
    \mathbb{P}(\calF_\alpha^n = \mathbb{R}^\ell) &= \mathbb{P}\bigg(\frac{n}{n+1}\frac{1}{1+d_n^X}\leq \frac{q_{n,\alpha}'}{n}\bigg).
\end{align*}
Set $ p_{i,n}'' \coloneqq (\bfP_{n,\lambda} - \bfP_{n,\lambda}^{XX})_{ii}$ for $i\in\{1, \ldots,n\}$. Because $d_n^X\geq 0$, we have the bound
\begin{align*}
    p_{i,n}' &= (\bfP_{n,\lambda} - \bfP_{n,\lambda}^{XX})_{ii} + \frac{(b_X^n)_i^2}{1+d_n^X} \leq p_{i,n}'' + \sum_{j=1}^n(b_X^n)_j^2.
\end{align*}
Hence,
\begin{align}
    q_{n,\alpha}' &\leq q_{n,\alpha}'' +  n\sum_{j=1}^n(b_X^n)_j^2, \label{eq:control_qnaprime_qnasecond}
\end{align}
where $q_{n,\alpha}''$ is the order statistic of order $n_\alpha$ of the $n$-tuple $(np_{1,n}'', \ldots, np_{n,n}'')$. Using equation \eqref{eq:control_qnaprime_qnasecond},
\begin{align*}
    \mathbb{P}(\calF_\alpha^n = \mathbb{R}^\ell) &\leq \mathbb{P}\bigg(\frac{n}{n+1}\frac{1}{1+d_n^X} \leq \frac{q_{n,\alpha}''}{n} + \sum_{j=1}^n(b_X^n)_j^2\bigg) =\mathbb{P}\bigg(\frac{q_{n,\alpha}''}{n} + \sum_{j=1}^n(b_X^n)_j^2 + \bigg(1-\frac{n}{n+1}\frac{1}{1+d_n^X}\bigg) \geq 1\bigg).
\end{align*}
Now, given $X, Y, Z$ three random variables, we have the event inclusion
\begin{align*}
    \{X+Y+Z\geq 1\} \subset \{ X\geq 1/3\}\cup \{ Y\geq 1/3\}\cup \{ Z\geq 1/3\}.
\end{align*}(If both $X<1/3$ and $Y<1/3$, then $Z\geq 1-X-Y \geq 1/3$.)
In particular, using a union bound, 
\begin{align}\label{eq:union_bound}
    \mathbb{P}(\calF_\alpha^n = \mathbb{R}^\ell) \leq \mathbb{P}\bigg(\frac{q_{n,\alpha}''}{n} \geq \frac{1}{3}\bigg) + \mathbb{P}\bigg(\sum_{j=1}^n(b_X^n)_j^2\geq \frac{1}{3}\bigg) + \mathbb{P}\bigg(1-\frac{n}{n+1}\frac{1}{1+d_n^X}\geq \frac{1}{3}\bigg).
\end{align}
We now show that the three random variables in the equation above converge to $0$ in probability, which will imply that $ \mathbb{P}(\calF_\alpha^n = \mathbb{R}^\ell) \rightarrow 0$ when $n \rightarrow \infty$. First, from Lemma \ref{lemma:cv_quantile}, we have that $q_{n,\alpha}''$ converges a.s. to a constant, so a.s., $q_{n,\alpha}''/n \rightarrow 0$ when $n\rightarrow \infty$, in particular
\begin{align}\label{eq:cv_qnasecond_over_n}
    \mathbb{P}\bigg(\frac{q_{n,\alpha}''}{n} \geq \frac{1}{3}\bigg) \xlongrightarrow[n\rightarrow \infty]{} 0.
\end{align}
Next, denoting $\lambda_{\min,n} = \min \Spec (\mathbf{\Sigma}_{n,\lambda=0}^{11})$, we have\begin{align}
    0 \leq d_n^X &= \frac{1}{n+1}\mathbb{E}[X_{n+1}^\rmc (\mathbf{\Sigma}_{n,\lambda}^{11})^{-1}X_{n+1}^\rmc \leq \frac{1}{n+1} \frac{1}{\lambda + \lambda_{\min,n}}\|X_{n+1}-\overline{X}_n\|^2\nonumber \\
    &\leq 
    \frac{2}{(n+1)(\lambda + \lambda_{\min,n})}(\|X_{n+1}\|^2 + \|\overline{X}_n\|^2) =  \frac{2\|X_{n+1}\|^2}{(n+1)(\lambda + \lambda_{\min,n})}+  \frac{2\|\overline{X}_n\|^2}{(n+1)(\lambda + \lambda_{\min,n})}.\label{eq:control_dnx}
\end{align}
From the SLLN and the continuity of the eigenvalues of self adjoint matrices (\cite{horn}, Problem 1 p. 198), we have a.s., $\lambda_{\min,n}\rightarrow \lambda_{\min}$ where $\lambda_{\min,n} = \min \Spec (\mathbf{\Sigma}^{11})$.Thus the second term in equation converges a.s. to $0$. Moreover, $X_{n+1} \stackrel{\rmd}{=} X_1$. Thus, from Slutsky's lemma and the continuous mapping theorem,
\begin{align*}
    \frac{2\|X_{n+1}\|^2}{(n+1)(\lambda + \lambda_{\min,n})} \xlongrightarrow[n\rightarrow\infty]{\rmd} 0 \times \|X_1\|^2 = 0.
\end{align*}
Because the limit above is constant, the convergence above also holds in probability. Thus, $d_n^X \rightarrow 0$ in probability, when $n\rightarrow\infty$. Hence, assuming $n\geq 3$,
\begin{align}\label{eq:cv_dnx}
    \mathbb{P}\bigg(1-\frac{n}{n+1}\frac{1}{1+d_n^X}\geq \frac{1}{3}\bigg) = \mathbb{P}\bigg(d_n^X\geq \frac{3}{2}\frac{n}{n+1}-1\bigg)\leq \mathbb{P}\bigg(d_n^X\geq \frac{9}{8}-1\bigg) = \mathbb{P}\bigg(d_n^X\geq \frac{1}{8}\bigg) \xlongrightarrow[n\rightarrow\infty]{} 0.
\end{align}
Finally, using that $(b_X^n)_{n+1}^2 = w_{n+1}^2 = n/(n+1)$ (see Lemma \ref{lemma:expr_bnx_dnx_rnx}),
\begin{align}
    \|b_X^n\|^2 &= (b_X^n)^\top b_X^n = (b_X^n)_{n+1}^2 + \sum_{j=1}^n(b_X^n)_j^2 = \frac{n}{n+1} + \sum_{j=1}^n(b_X^n)_j^2.  \label{eq:norm_bnx}
\end{align}
Now, denote $\mathbbm{1}_n$ the vector of ones of $\mathbb{R}^n$, and observe that $\bfB_n^X=\pi_{\mathbbm{1}_n}^\perp \bfM$ where $\bfM\in\calM_{n,k}$ is given by $\bfM_{ij} = (X_i)_j$. Observe further that $w = c \times  (-1 \ \ldots -1 \ n)^\top = c \times  (-\mathbbm{1}_n^\top\ n)^\top$, where $c = (n+1)^{-1}\|v\|^{-1}.$  Using that $\pi_{\mathbbm{1}_n}^\perp\mathbbm{1}_n = 0$ and that the matrix $\pi_{\mathbbm{1}_n}^\perp$ is symmetric,
\begin{align*}
    w^\top\begin{pmatrix}
             \bfB_n^X(\mathbf{\Sigma}_{n,\lambda}^{11})^{-1}X_{n+1}^\rmc \\
             0
             \end{pmatrix} = -c\Big(\mathbbm{1}_n^\top\pi_{\mathbbm{1}_n}^\perp\Big)\bfM(\mathbf{\Sigma}_{n,\lambda}^{11})^{-1}(X_{n+1}^\rmc) = 0.
\end{align*}
Hence,
\begin{align}
    \|b_X^n\|^2 &= \bigg(\frac{\|v\|}{n}\begin{pmatrix}
             \bfB_n^X(\mathbf{\Sigma}_{n,\lambda}^{11})^{-1}X_{n+1}^\rmc\\0
    \end{pmatrix}-w\bigg)^\top\bigg(\frac{\|v\|}{n}\begin{pmatrix}
             \bfB_n^X(\mathbf{\Sigma}_{n,\lambda}^{11})^{-1}X_{n+1}^\rmc\\0
    \end{pmatrix}-w\bigg)\nonumber \\
    &= \frac{\|v\|^2}{n}(X_{n+1}^\rmc)^\top(\mathbf{\Sigma}_{n,\lambda}^{11})^{-1}\frac{(\bfB_n^X)^\top\bfB_n^X}{n}(\mathbf{\Sigma}_{n,\lambda}^{11})^{-1}X_{n+1}^\rmc + \|w\|^2\nonumber \\
    &\leq \frac{\|v\|^2}{n}(X_{n+1}^\rmc)^\top(\mathbf{\Sigma}_{n,\lambda}^{11})^{-1}X_{n+1}^\rmc + 1 = d_n^X + 1. \label{eq:bound_norm_b_nx}
\end{align}
We used that $n^{-1}(\bfB_n^X)^\top\bfB_n^X(\mathbf{\Sigma}_{n,\lambda}^{11})^{-1} \preccurlyeq \bfI_k$ in equation \eqref{eq:bound_norm_b_nx}.
Combining equations \eqref{eq:norm_bnx} and \eqref{eq:bound_norm_b_nx},
\begin{align}\label{eq:control_sum_bnx_dnx}
    0 \leq \sum_{j=1}^n(b_X^n)_j^2 \leq d_n^X + \frac{1}{n+1},
\end{align}
and thus $\sum_{j=1}^n(b_X^n)_j^2 \rightarrow 0$ in probability, when $n\rightarrow\infty$. As a consequence,
\begin{align}\label{eq:cv_sum_bnx}
    \mathbb{P}\bigg(\sum_{j=1}^n(b_X^n)_j^2\geq \frac{1}{3}\bigg) \xlongrightarrow[n\rightarrow\infty]{} 0.
\end{align}
Combining equations \eqref{eq:union_bound}, \eqref{eq:cv_qnasecond_over_n}, \eqref{eq:cv_dnx} and \eqref{eq:cv_sum_bnx}, we finally obtain that
\begin{align}
    \mathbb{P}(\calF_n^\alpha = \mathbb{R}^\ell) \xlongrightarrow[n\rightarrow\infty]{} 0,
\end{align}
which finishes the proof.
}
\end{proof}
\begin{proof}[Proof of Proposition \ref{prop:asymptotics}]
We clearly have, from the SLLN,
\begin{align}\label{eq:as_cov}
    \widehat{\mathbf{\Sigma}}_{n,\lambda} = \frac{1}{n}\bfB_n^\top \bfB_n + \lambda \mathbf{I}_p \xrightarrow[n\rightarrow \infty]{a.s.} \mathbf{\Sigma}_{\lambda}.
\end{align}
In particular, since we chose $\lambda > 0$, the continuous mapping theorem implies that
\begin{align}\label{eq:cv_inverse}
        \boldcalA_n  &\xrightarrow[n\rightarrow \infty]{a.s.}  \boldcalA_{\infty}.
\end{align}
Next, the fact that $q_{n,\alpha} \rightarrow q_{1-\alpha}^{\calE}$ almost surely is deferred in Lemma \ref{lemma:cv_quantile}.
This fact also implies that $q_{n,\alpha}/n \xrightarrow[n\rightarrow \infty]{a.s.} 0.$
Thus, using Slutsky's lemma,
\begin{align}
\rho_{n,\alpha} \xlongrightarrow[n\rightarrow \infty]{\rmd} q_{1-\alpha}^{\calE} - X_\rmc^\top (\mathbf{\Sigma}_{\lambda}^{11})^{-1}X_\rmc = \rho_{\infty,\alpha}.
\end{align}
Likewise, the convergence in distribution of $Z_0^n$ is also obtained via Slutsky's lemma. Next, from the continuous mapping theorem, the volume converges in distribution to the random limit volume given by
\begin{align}
    \Vol(\calE_{\alpha}^{\infty}) &= v_{\ell} \det\big((\rho_{\infty,\alpha})_+\boldcalA_{\infty}\big)^{1/2} = v_{\ell}\sqrt{\det(\mathbf{\Sigma}_{\lambda}/\mathbf{\Sigma}_{\lambda}^{11})}\Big(q_{1-\alpha}^{\calE} - X_\rmc^\top (\mathbf{\Sigma}_{\lambda}^{11})^{-1}X_\rmc\Big)^{\ell/2}_+.
\end{align}
The limit probability in equation \eqref{eq:proba_empty} corresponds to $\rho_{\infty,\alpha} < 0$. We now show that it is $\leq \alpha$. For this, we introduce the simplified notations
\begin{align}
    \mathbf{\Sigma}_{\lambda} =
    \begin{pmatrix}
             &\bfA   &\bfB \\
             &\bfB^\top    &\bfC
    \end{pmatrix}, \ \ \bfA = \mathbf{\Sigma}_{\lambda}^{11}.\label{eq:decomp_sigma_lambda_proof}
\end{align}
For conciseness, denote $A =  V_\rmc^\top \mathbf{\Sigma}_{\lambda}^{-1}V_\rmc, \ B = X_\rmc^\top (\mathbf{\Sigma}_{\lambda}^{11})^{-1}X_\rmc$ and $F_A, F_B$ their CDF: we first show that $A \geq B$.
From equation \eqref{eq:apply_block_lemma},
\begin{align}
    A &= X_\rmc^\top \bfA^{-1}X_\rmc + (R_\rmc - \bfB^\top \bfA^{-1} X_\rmc)^\top (\mathbf{\Sigma}_{\lambda}/\bfA)^{-1}(R_\rmc - \bfB^\top \bfA^{-1} X_\rmc)\nonumber\\
    &= B + (R_\rmc - \bfB^\top \bfA^{-1} X_\rmc)^\top (\mathbf{\Sigma}_{\lambda}/\bfA)^{-1}(R_\rmc - \bfB^\top \bfA^{-1} X_\rmc).\label{eq:facto_quad_conditionnel}
    \end{align}
Since $(\mathbf{\Sigma}_{\lambda}/\bfA)\succcurlyeq 0$, the equation above shows that $A \geq B$ and thus $F_A \leq F_B$. Hence, denoting $Q_A$ the quantile function of $A$ and using that $Q_A(p)\leq x \iff p\leq F_A(x)$,
\begin{align}
    \mathbb{P}(\calE_{\alpha}^{\infty} = \varnothing) &= \mathbb{P}(X_\rmc^\top (\mathbf{\Sigma}_{\lambda}^{11})^{-1}X_\rmc > q_{1-\alpha}^{\calE}) =1 - \mathbb{P}(X_\rmc^\top (\mathbf{\Sigma}_{\lambda}^{11})^{-1}X_\rmc \leq q_{1-\alpha}^{\calE}) \nonumber\\
    &= 1 - F_B(Q_A(1-\alpha)) \leq 1 - F_A(Q_A(1-\alpha)) \leq \alpha. 
\end{align}
To finish, if $\lambda = 0$, we assume that $\min \Spec(\mathbf{\Sigma}) > 0$. From the continuity of the smallest eigenvalue over the set of Hermitian matrices \cite{horn}, Problem 1 p. 198, and equation \eqref{eq:as_cov}, we deduce that almost surely, $\min\Spec(\widehat{\mathbf{\Sigma}}_n) >0$ for $n$ large enough ($n$ depends on the sample). Thus, equation \eqref{eq:cv_inverse} also holds, from the continuity of the map $\bfA\mapsto \bfA^{-1}$ over the set of invertible matrices of size $p$. Since $\min\Spec(\mathbf{\Sigma}^{11})\geq \min\Spec(\mathbf{\Sigma})>0$ (\cite{horn}, Theorem 4.3.15), the same argument shows that $(\widehat{\mathbf{\Sigma}}_n^{11})^{-1} \rightarrow (\mathbf{\Sigma}^{11})^{-1}$ almost surely. The rest of the proof is identical to the case $\lambda>0$.
\end{proof}
\begin{proof}[Proof of Proposition \ref{prop:cas_gaussien}]
Under the assumptions, we have $S\coloneqq X_\rmc^\top (\mathbf{\Sigma}^{11})^{-1}X_\rmc \sim \chi^2(k)$, $ V_\rmc^\top \mathbf{\Sigma}^{-1}V_\rmc \sim \chi^2(k+\ell)$ and $q_{1-\alpha}^{\calE} = F^{-1}_{\chi^2(k+\ell)}(1-\alpha).$ 
We begin with computing $\mathbb{E}[\Vol(\calE_{\alpha}^{\infty})^q]$.
Denoting $t = q_{1-\alpha}^{\calE}$, we have that $\mathbb{E}[\Vol(\calE_{\alpha}^{\infty})^q] = v_{\ell}^q\det(\mathbf{\Sigma}/\mathbf{\Sigma}^{11})^{q/2}\mathbb{E}\big[(t-S)_+^{q\ell/2}\big]$. Setting $C_k = 2^{k/2}\Gamma(k/2)$, the expectation is further simplified as
\begin{align}
    \mathbb{E}\big[(t-S)_+^{q\ell/2}\big] &= \int_{\Rbb^+}(t-s)_+^{q\ell/2}\frac{s^{k/2-1}e^{-s/2}}{C_k}ds = C_k^{-1}t^{q\ell/2}\int_0^t(1-s/t)_+^{q\ell/2}s^{k/2-1}e^{-s/2}ds\nonumber\\
    &= C_k^{-1}t^{q\ell/2}\int_0^1(1-v)^{q\ell/2}(tv)^{k/2-1}e^{-tv/2}tdv\nonumber\\
    &= C_k^{-1}t^{(k+q\ell)/2}\int_0^1(1-v)^{q\ell/2}v^{k/2-1}e^{-tv/2}dv\nonumber \\
    &= C_k^{-1}B\bigg(\frac{k}{2},\frac{q\ell}{2}+1\bigg)t^{(k+q\ell)/2}\Phi_{\mathrm{Beta}(\frac{k}{2},\frac{q\ell}{2}+1)}(it/2)\nonumber\\
    &= \frac{\Gamma(q\ell/2 + 1)}{2^{k/2}\Gamma((k+q\ell)/2+1)}t^{(k+q\ell)/2}{_1F_1}\bigg(\frac{k}{2}, \frac{k+q\ell}{2}+1, -\frac{t}{2}\bigg).
\end{align}
Above, $B(x,y) = \Gamma(x)\Gamma(y)/\Gamma(x+y)$ is the Euler Beta function, and $\Phi_{\mu}(t) = \int e^{itx}\mu(dx)$ is the characteristic function of a given measure $\mu$. Here, we recognize the characteristic function of the Beta distribution $\mathrm{Beta}({k}/{2},{q\ell}/{2}+1)$ \cite{continuous_distrib_2}, p 218. For the proof of the convergence statement, we begin with observing that $\lfloor \alpha(n+1)\rfloor >\alpha(n+1)-1$. From equation \eqref{eq:control_qnalpha}, when $n \geq 2(1-\alpha)/\alpha$,
\begin{align}
q_{n,\alpha} \leq \frac{np}{\alpha(n+1)-1}  = \frac{np}{\alpha n-(1-\alpha)} = \frac{p}{\alpha -(1-\alpha)/n} \leq \frac{p}{\alpha -\frac{(1-\alpha)\alpha}{2(1-\alpha)}} =  \frac{2p}{\alpha}.
\end{align}
From this we deduce that, for $n$ large enough so that $1-(1+2p/\alpha)/n >1/2$,
\begin{align}
    \rho_{n,\alpha} &= \bigg(\frac{n(q_{n,\alpha} + 1)}{n- (q_{n,\alpha}+1)}-1 - (X^{n+1}_\rmc)^\top (\widehat{\mathbf{\Sigma}}_{n,\lambda}^{11})^{-1}X_{n+1}^\rmc\bigg)_+ \leq \bigg( \frac{n(q_{n,\alpha} + 1)}{n- (q_{n,\alpha}+1)}\bigg)_+\nonumber\\
    &\leq \frac{q_{n,\alpha} + 1}{1- (q_{n,\alpha}+1)/n} \leq \frac{1+2p/\alpha}{1-(1+2p/\alpha)/n} \leq 2(1+2p/\alpha).
\end{align}
Hence, using that $\det(\widehat{\mathbf{\Sigma}}_n/\widehat{\mathbf{\Sigma}}_n^{11})\leq \det(\widehat{\mathbf{\Sigma}}_n^{22})$, we obtain that
\begin{align}
    \Vol(\calE_{\alpha}^n)& = v_\ell \det(\widehat{\mathbf{\Sigma}}_n/\widehat{\mathbf{\Sigma}}_n^{11})^{1/2}(\rho_{n,\alpha})_+^{\ell/2} \lesssim \det(\widehat{\mathbf{\Sigma}}_n^{22})^{1/2}.
\end{align}
Now, recall from \cite{mardia_kent}, Theorem 3.4.8, that $\det(\widehat{\mathbf{\Sigma}}_n^{22})$ is equal, in distribution, to $\det(\mathbf{\Sigma}^{22})U_0\times \ldots \times U_{\ell-1}$, where $nU_i \sim \chi^2(n-i)$ and the $U_i$ are independent. In particular, for all $q>0$ \cite{continuous_distrib_1}, p 420,
\begin{align}
    \mathbb{E}[\det(\widehat{\mathbf{\Sigma}}_n^{22})^q] = \det(\mathbf{\Sigma}^{22})^q\prod_{i=0}^{\ell-1}n^{-q}\mathbb{E}[(nU_i)^q] = \det(\mathbf{\Sigma}^{22})^q \prod_{i=0}^{\ell-1}\frac{2^q\Gamma(q+(n-i)/2)}{n^q\Gamma((n-i)/2)}.
\end{align}
But, when $s\rightarrow +\infty$ and $q$ is fixed, $\Gamma(s+q)/\Gamma(s)\sim s^{q}$ (\cite{stegun}, Section 6.1.39). Applying this result to each term in the product above, in the regime where $n\rightarrow +\infty$, we obtain that for all $q>0$, $\sup_n\mathbb{E}[\det(\widehat{\mathbf{\Sigma}}_n^{22})^q]<+\infty$. From \cite{van2000asymptotic}, Example 2.21,  and equation \eqref{eq:asymptotic_vol}, all the moments of $\Vol(\calE_{\alpha}^n)$ converge toward those of $\Vol(\calE_{\alpha}^{\infty})$. In particular, this proves equation \eqref{eq:exp_vol_gauss}.
\end{proof}
\begin{proof}[Proof of Proposition \ref{prop:calf_asymptotic}]
\iain{
The convergence of $q_{n,\alpha}'$ is given by Lemma \ref{lemma:cv_quantile}, equation \eqref{eq:cvqna_prime}. It is also proved in Proposition \ref{prop:calF_full_space} that, if $\mathbb{E}[\|X_1\|^2]<+\infty$, then $d_n^X \rightarrow 0$ in probability (see e.g. equation \eqref{eq:cv_dnx}), so that from the definition of $t_{n}$ (equation \eqref{eq:def_tna}), $t_{n} \rightarrow 1$ in probability. Thus, from the definition of $\rho_{n,\alpha}'$ (equation \eqref{eq:def_rho_n_prime}), we obtain that $\rho_{n,\alpha}' \rightarrow q_{1-\alpha}^\calF$ in probability.}
\end{proof}
\begin{proof}[Proof of Proposition \ref{prop:conditional_gaussian}]
\iain{Decompose $\mathbf{\Sigma}$ as in equation \eqref{eq:decomp_sigma}.
As in Propositions \ref{prop:asymptotics} and \ref{prop:calf_asymptotic}, set $V_\rmc \coloneqq V_1 - \mathbb{E}[V_1], \ V_\rmc = (X_\rmc^\top R_\rmc^\top)^\top$ and $T_\rmc = R_\rmc - \mathbf{\Sigma}^{21}(\mathbf{\Sigma}^{11})^{-1}X_\rmc$. Then
\begin{align}\label{eq:t_c_calf}
    Y_{1} - \widehat{Y}_{1} \in \calF_\alpha^\infty \iff R_1 \in \calF_\alpha^\infty \iff T_\rmc^\top(\mathbf{\Sigma}/\mathbf{\Sigma}^{11})^{-1}T_\rmc \leq q_{1-\alpha}^\calF.
\end{align}
Under the assumption that $V_\rmc \sim \calN(0,\mathbf{\Sigma})$, using the bilinearity of $\Cov(\cdot,\cdot)$,
\begin{align*}
    \Cov(T_\rmc,X_1) = \Cov(T_\rmc,X_\rmc) = \Cov(R_\rmc,X_\rmc) -  \mathbf{\Sigma}^{21}(\mathbf{\Sigma}^{11})^{-1}\Cov(X_\rmc,X_\rmc) = \mathbf{\Sigma}^{21} -\mathbf{\Sigma}^{21}(\mathbf{\Sigma}^{11})^{-1}\mathbf{\Sigma}^{11} = 0.
\end{align*}
Because $V_\rmc$ is a Gaussian random vector, this implies that $T_\rmc$ and $X_1$ are independent. In particular, from equation \eqref{eq:t_c_calf} and the definition of $q_{1-\alpha}^\calF$,
\begin{align*}
    \mathbb{P}(Y_{1} - \widehat{Y}_{1} \in \calF_\alpha^\infty |X_1=x) = \mathbb{P}(T_\rmc^\top\boldcalA_\infty^{-1}T_\rmc \leq q_{1-\alpha}^\calF|X_1=x) = \mathbb{P}(T_\rmc^\top\boldcalA_\infty^{-1}T_\rmc \leq q_{1-\alpha}^\calF) = 1-\alpha, 
\end{align*}
which finishes the proof.
}
\end{proof}
\begin{proof}[Proof of Proposition \ref{prop:pred_unbiased}]  
{Denote $\overline{R}_n \coloneqq n^{-1}\sum_{i=1}^nR_i$ and $ \overline{X}_n\coloneqq n^{-1}\sum_{i=1}^nX_i$, and write
\begin{align}
    \mathbb{E}[\widetilde{Y}_{n+1}] - \mathbb{E}[Y_{1}] &= \mathbb{E}[\widehat{Y}_{n+1} - Y_{n+1}] + \mathbb{E}\big[\ \overline{R}_n\ \big] + \mathbb{E}\big[\widehat{\mathbf{\Sigma}}_n^{21}(\widehat{\mathbf{\Sigma}}_{n,\lambda}^{11})^{-1}(X_{n+1}-\overline{X}_n)\big]\nonumber \\
    &=-\mathbb{E}[R_{n+1}] + \mathbb{E}\big[\ \overline{R}_n\ \big] + \mathbb{E}\big[\widehat{\mathbf{\Sigma}}_n^{21}(\widehat{\mathbf{\Sigma}}_{n,\lambda}^{11})^{-1}X_{n+1}\big] - \mathbb{E}\big[\widehat{\mathbf{\Sigma}}_n^{21}(\widehat{\mathbf{\Sigma}}_{n,\lambda}^{11})^{-1}\overline{X}_n\big]\nonumber \\
    &=\mathbb{E}\big[\widehat{\mathbf{\Sigma}}_n^{21}(\widehat{\mathbf{\Sigma}}_{n,\lambda}^{11})^{-1}\big]\mathbb{E}[X_{n+1}] - \mathbb{E}\big[\widehat{\mathbf{\Sigma}}_n^{21}(\widehat{\mathbf{\Sigma}}_{n,\lambda}^{11})^{-1}\overline{X}_n\big]\label{eq:indep} \\
    &= \mathbb{E}\big[\widehat{\mathbf{\Sigma}}_n^{21}(\widehat{\mathbf{\Sigma}}_{n,\lambda}^{11})^{-1}(\mathbb{E}[X_{1}]-\overline{X}_n)\big].\nonumber
\end{align}
In equation \eqref{eq:indep}, we used that $\widehat{\mathbf{\Sigma}}_n^{21}(\widehat{\mathbf{\Sigma}}_{n,\lambda}^{11})^{-1}$ and $X_{n+1}$ are independent.
Hence (explanation below),
\begin{align}
    \|\mathbb{E}[\widetilde{Y}_{n+1}] - \mathbb{E}[Y_{1}] \|_2^2 &\leq \mathbb{E}\big[\|\widehat{\mathbf{\Sigma}}_n^{21}(\widehat{\mathbf{\Sigma}}_{n,\lambda}^{11})^{-1/2}(\widehat{\mathbf{\Sigma}}_{n,\lambda}^{11})^{-1/2}(\mathbb{E}[X_{1}]-\overline{X}_n)\|_2\big]^2 \label{eq:triangle}\\
    &\leq \mathbb{E}\big[\|\widehat{\mathbf{\Sigma}}_n^{21}(\widehat{\mathbf{\Sigma}}_{n,\lambda}^{11})^{-1/2}\|_{\Op}  \|(\widehat{\mathbf{\Sigma}}_{n,\lambda}^{11})^{-1/2}(\mathbb{E}[X_{1}]-\overline{X}_n)\|_2\big]^2 \label{eq:matrix_norm}\\
    &\lesssim\mathbb{E}\big[\|\widehat{\mathbf{\Sigma}}_n^{21}(\widehat{\mathbf{\Sigma}}_{n,\lambda}^{11})^{-1/2}\|_F  \|(\widehat{\mathbf{\Sigma}}_{n,\lambda}^{11})^{-1/2}(\mathbb{E}[X_{1}]-\overline{X}_n)\|_2\big] ^2\label{eq:frob}\\
    &\lesssim\mathbb{E}[\|\widehat{\mathbf{\Sigma}}_n^{21}(\widehat{\mathbf{\Sigma}}_{n,\lambda}^{11})^{-1/2}\|_F^2]\times \mathbb{E}\big[\|(\widehat{\mathbf{\Sigma}}_{n,\lambda}^{11})^{-1/2}(\mathbb{E}[X_{1}]-\overline{X}_n)\|_2^2\big].\label{eq:cs_frob}
\end{align}
We used the triangle inequality in equation \eqref{eq:triangle}, the operator norm such that $\|\bfM x\|_2\leq\|\bfM\|_{\Op}\|x\|_2$ in equation \eqref{eq:matrix_norm}, the equivalence of $\|\cdot\|_{\Op}$ with the Frobenius norm $\|\bfM\|_F = \Tr(\bfM\bfM^\top)^{1/2}$ in equation \eqref{eq:frob}, and the Cauchy-Schwarz inequality in equation \eqref{eq:cs_frob}.
But, since $\widehat{\mathbf{\Sigma}}_{n,\lambda}^{22} - \widehat{\mathbf{\Sigma}}_n^{21}(\widehat{\mathbf{\Sigma}}_{n,\lambda}^{11})^{-1}\widehat{\mathbf{\Sigma}}_n^{12} = \widehat{\mathbf{\Sigma}}_{n,\lambda}/\widehat{\mathbf{\Sigma}}_{n,\lambda}^{11}\succcurlyeq 0$,
\begin{align}
    \mathbb{E}[\|\widehat{\mathbf{\Sigma}}_n^{21}(\widehat{\mathbf{\Sigma}}_{n,\lambda}^{11})^{-1/2}\|_F^2] &= \mathbb{E}[\Tr(\widehat{\mathbf{\Sigma}}_n^{21}(\widehat{\mathbf{\Sigma}}_{n,\lambda}^{11})^{-1}\widehat{\mathbf{\Sigma}}_n^{12})] \leq \mathbb{E}[\Tr(\widehat{\mathbf{\Sigma}}_{n,\lambda}^{22})]\nonumber\\
    &\leq\frac{1}{n}\sum_{i=1}^n \mathbb{E}[ \Tr((R_i-\overline{R}_n)(R_i-\overline{R}_n)^T)] + \lambda\ell\nonumber \\
    &\leq \mathbb{E}[\|R_1-\overline{R}_n\|_2^2] + \lambda\ell = \frac{n-1}{n}\Tr(\boldsymbol{\Sigma^{22}}) + \lambda\ell.\label{eq:frob_Y}
\end{align} 
Moreover, using that  $\|(\widehat{\mathbf{\Sigma}}_{n,\lambda}^{11})^{-1/2}\|_F^2 = \Tr((\widehat{\mathbf{\Sigma}}_{n,\lambda}^{11})^{-1}) \leq \Tr(\lambda^{-1}\mathbf{I}_{k}) = k/\lambda \lesssim 1$ conjointly with the operator and Frobenius norms,
\begin{align}
    \mathbb{E}\big[\|(\widehat{\mathbf{\Sigma}}_{n,\lambda}^{11})^{-1/2}(\mathbb{E}[X_{1}]-\overline{X}_n)\|_2^2\big] &\lesssim \mathbb{E}\big[\|(\widehat{\mathbf{\Sigma}}_{n,\lambda}^{11})^{-1/2}\|_F^2\|\mathbb{E}[X_{1}]-\overline{X}_n\|_2^2\big]\nonumber \\
    &\lesssim \mathbb{E}\big[\|\mathbb{E}[X_{1}]-\overline{X}_n\|_2^2\big] = {\Tr(\mathbf{\Sigma})}/{n}.\label{eq:cv_xbar_l2}
\end{align}
Equations \eqref{eq:cs_frob}, \eqref{eq:frob_Y} and \eqref{eq:cv_xbar_l2} then yield equation \eqref{eq:unbiased}. We now show equation \eqref{eq:cov_assymptotic}. Assuming that there exists $q>1$ such that $\mathbb{E}[\|V_1\|^{4q}]<+\infty$, we first show that the sequence $\| \widetilde{Y}_{n+1} - Y_{n+1}\|_2^{2}$ is uniformly integrable. Write
\begin{align}
    \widetilde{Y}_{n+1} - Y_{n+1} &= \widehat{Y}_{n+1} - Y_{n+1} + \overline{R}_n + \widehat{\mathbf{\Sigma}}_{n}^{21}(\widehat{\mathbf{\Sigma}}_{n,\lambda}^{11})^{-1}(X_{n+1} - \overline{X}_n) \nonumber \\
    &= \overline{R}_n - R_{n+1} + \widehat{\mathbf{\Sigma}}_{n}^{21}(\widehat{\mathbf{\Sigma}}_{n,\lambda}^{11})^{-1}(X_{n+1} - \overline{X}_n).\label{eq:decomp_reste_corrige}
\end{align}
Hence, starting from equation \eqref{eq:decomp_reste_corrige} and following the steps leading to equation \eqref{eq:cs_frob},
\begin{align}
    \| \widetilde{Y}_{n+1} - Y_{n+1}\|_2^{2q} &\lesssim \Big(\| \overline{R}_n - R_{n+1}\|_2 + \|\widehat{\mathbf{\Sigma}}_{n}^{21}(\widehat{\mathbf{\Sigma}}_{n,\lambda}^{11})^{-1/2}\|_{F}\|(\widehat{\mathbf{\Sigma}}_{n,\lambda}^{11})^{-1/2}(X_{n+1} - \overline{X}_n)\|_2\Big)^{2q}\nonumber\\
    &\lesssim \| \overline{R}_n - R_{n+1}\|_2^{2q} + \|\widehat{\mathbf{\Sigma}}_{n}^{21}(\widehat{\mathbf{\Sigma}}_{n,\lambda}^{11})^{-1/2}\|_{F}^{2q}\|(\widehat{\mathbf{\Sigma}}_{n,\lambda}^{11})^{-1/2}(X_{n+1} - \overline{X}_n)\|_2^{2q}.
\end{align}
Taking the expectation and using the Cauchy-Schwarz inequality,
\begin{align}
    \mathbb{E}[ \| \widetilde{Y}_{n+1} &- Y_{n+1}\|_2^{2q}]\lesssim \mathbb{E}[\| \overline{R}_n - R_{n+1}\|_2^{2q}] + \mathbb{E}[\|\widehat{\mathbf{\Sigma}}_{n}^{21}(\widehat{\mathbf{\Sigma}}_{n,\lambda}^{11})^{-1/2}\|_{F}^{4q}]^{1/2}\mathbb{E}[\|X_{n+1} - \overline{X}_n\|_2^{4q}]^{1/2}.\nonumber
\end{align}
Using the convexity of $(\cdot)^{2q}$ and $(\cdot)^{4q}$, we have that $\mathbb{E}[\| \overline{R}_n - R_{n+1}\|_2^{2q}] \lesssim \mathbb{E}[\|R_1\|^{2q}]$ and $\mathbb{E}[\|X_{n+1} - \overline{X}_n\|_2^{4q}] \lesssim \mathbb{E}[\|X_1\|^{4q}]$. Likewise,
\begin{align}
    \mathbb{E}[\|\widehat{\mathbf{\Sigma}}_{n}^{21}(\widehat{\mathbf{\Sigma}}_{n,\lambda}^{11})^{-1/2}\|_{F}^{4q}] &= \mathbb{E}[\Tr(\widehat{\mathbf{\Sigma}}_{n}^{21}(\widehat{\mathbf{\Sigma}}_{n,\lambda}^{11})^{-1}\widehat{\mathbf{\Sigma}}_{n}^{12})^{2q}]\leq  \mathbb{E}[\Tr(\widehat{\mathbf{\Sigma}}_{n,\lambda}^{22})^{2q}]\nonumber \\
    &\leq \mathbb{E}\bigg[\bigg(\ell\lambda + \frac{1}{n}\sum_{i=1}^n\Tr\Big((R_i-\overline{R}_n)(R_i-\overline{R}_n)^\top\Big)\bigg)^{2q}\bigg] \nonumber\\
    &\leq \mathbb{E}\bigg[\Big(\frac{1}{n}\sum_{i=1}^n(\|R_i-\overline{R}_n\|^2+\ell\lambda)\Big)^{2q}\bigg]\nonumber\\
    &\leq \mathbb{E}\bigg[\frac{1}{n}\sum_{i=1}^n(\|R_i-\overline{R}_n\|^2+\ell\lambda)^{2q}\bigg]\nonumber\\
    &\leq \mathbb{E}[(\|R_1-\overline{R}_n\|^2+\ell\lambda)^{2q}] \lesssim \mathbb{E}[\|R_1\|^{4q}] + \lambda^{2q}.
\end{align}
Hence, $\sup_n\mathbb{E}[ \| \widetilde{Y}_{n+1} - Y_{n+1}\|_2^{2q}] < +\infty$. Setting $S \coloneqq \mathbf{\Sigma}^{21}(\mathbf{\Sigma}_{\lambda}^{11})^{-1}X_1 - R_1$ and observing from Proposition \ref{prop:asymptotics} that 
\begin{align}
     \widetilde{Y}_{n+1} - Y_{n+1} \xlongrightarrow[n\rightarrow + \infty]{d} S-\mathbb{E}[S],
\end{align}
we may use \cite{van2000asymptotic}, Example 2.21, so that $\| \widetilde{Y}_{n+1} - Y_{n+1}\|_2^{2}$ is uniformly integrable, with
\begin{align}\label{eq:cv_cov_waart}
    \mathbb{E}[\| \widetilde{Y}_{n+1} - Y_{n+1}\|_2^{2}] \xlongrightarrow[n\rightarrow\infty]{} \mathbb{E}[\|S-\mathbb{E}[S]\|^2] = \Tr\big(\Cov(S)\big).
\end{align}
Finally,
\begin{align}
    \Cov(S) &= \mathbf{\Sigma}^{22} + \mathbf{\Sigma}^{21}(\mathbf{\Sigma}_{\lambda}^{11})^{-1}\mathbf{\Sigma}^{11}(\mathbf{\Sigma}_{\lambda}^{11})^{-1}\mathbf{\Sigma}^{12} - 2\mathbf{\Sigma}^{21}(\mathbf{\Sigma}_{\lambda}^{11})^{-1}\mathbf{\Sigma}^{12}\nonumber\\
    &=\mathbf{\Sigma}^{22} + \lambda \mathbf{I}_{\ell} - \lambda \mathbf{I}_{\ell} + \mathbf{\Sigma}^{21}(\mathbf{\Sigma}_{\lambda}^{11})^{-1}(\mathbf{\Sigma}^{11}+\lambda \mathbf{I}_k-\lambda \mathbf{I}_k)(\mathbf{\Sigma}_{\lambda}^{11})^{-1}\mathbf{\Sigma}^{12}\nonumber \\&\hspace{2cm} - 2\mathbf{\Sigma}^{21}(\mathbf{\Sigma}_{\lambda}^{11})^{-1}\mathbf{\Sigma}^{12} \nonumber\\
    &= \mathbf{\Sigma}_{\lambda}^{22} - \mathbf{\Sigma}^{21}(\mathbf{\Sigma}_{\lambda}^{11})^{-1}\mathbf{\Sigma}^{12} - \lambda(\mathbf{I}_{\ell} + \mathbf{\Sigma}^{21}(\mathbf{\Sigma}_{\lambda}^{11})^{-2}\mathbf{\Sigma}^{12}) \nonumber\\
    &= \mathbf{\Sigma}_{\lambda}/\mathbf{\Sigma}_{\lambda}^{11}  - \lambda(\mathbf{I}_{\ell} + \mathbf{\Sigma}^{21}(\mathbf{\Sigma}_{\lambda}^{11})^{-2}\mathbf{\Sigma}^{12}) = \bfM_{\lambda}.
\end{align}
This finishes to prove that $\mathbb{E}[\| \widetilde{Y}_{n+1} - Y_{n+1}\|_2^{2}] \xlongrightarrow[]{}\Tr (\bfM_{\lambda})$. Using the same previous arguments for extra-diagonal entries of $\Cov(\widetilde{Y}_{n+1}-Y_{n+1})$ and equation \eqref{eq:unbiased}, we obtain equation \eqref{eq:cov_assymptotic}.
}
\end{proof}
\begin{proof}[Proof of Proposition \ref{prop:asymptotic_ball}] For equation \eqref{eq:cv_quantile_iid}, we follow the proof of Théorème 8.9, p. 90 of the lecture notes \cite{agreg_quantiles}. We denote $F$ the CDF of $\|R_1\|^2$ and $F_n$ the empirical CDF obtained from the iid sample $(\|R_1\|,  \ldots, \|R_n\|^2)$. We can write
\begin{align}
    |F(\beta_{n,\alpha}) - F(q_{1-\alpha}(\|R_1\|^2))|&\leq |F(\beta_{n,\alpha}) - F_n(\beta_{n,\alpha})| + |F_n(\beta_{n,\alpha}) - F(q_{1-\alpha}(\|R_1\|^2))|\nonumber\\
    &\leq \|F-F_n\|_{\infty} + \bigg|\frac{\lceil (1-\alpha)(n+1)\rceil}{n} - (1-\alpha)\bigg| \rightarrow 0 \ \ a.s.,\label{eq:GK}
\end{align}
where we used the Glivenko-Cantelli theorem in equation \eqref{eq:GK}.
Hence, a.s., $F(q_{n,\alpha}')$ converges to $F(q_{1-\alpha}(\|R_1\|^2)) = 1-\alpha$.
We obtain equation \eqref{eq:cv_quantile_iid} from the continuity of the quantile function of $\|R_1\|^2$ on a neighbourhood of $1-\alpha$. We then deduce equation \eqref{eq:vol_ball} from the continuous mapping theorem, as $\Vol(\calB_{\alpha}^n) = v_{\ell}(\beta_{n,\alpha})^{\ell/2}$.
\end{proof}
\begin{proof}[Proof of Proposition \ref{prop:barthe}]
Let $B$ be the centered ball of $\mathbb{R}^{\ell}$, of radius $t$. Let $A \subset \mathbb{R}^{\ell}$ be a measurable set such that $\lambda(A) = \lambda(B)$, where $\lambda$ stands for the Lebesgue measure of $\mathbb{R}^{\ell}$. Since $A = (A \cap B)\cup (A\cap B^\rmc)$ where the union is disjoint (and symmetrically for $B$), we also have that
\begin{align}\label{eq:lebesgue_complementaire}
    \lambda(A\cap B^\rmc) = \lambda (B\cap A^\rmc).
\end{align}
Now, let $(\delta_1, \dots , \delta_{\ell})\in\mathbb{R}_+^{\ell}$ be such that $\delta_1\dots \delta_{\ell} = 1$, and $s\in\mathbb{R}^\ell$. Observe that $ \sum_{i=1}^{\ell}\delta_i (z_i+s_i)^2 = (z+s)^\top \bfD_{\delta}(z+s)$ where $\bfD_{\delta}$ is the diagonal matrix such that $(\bfD_{\delta})_{i} = \delta_i$. The set $E_{\delta} \coloneqq \{z \in \mathbb{R}^{\ell}: (z+s)^\top \bfD_{\delta}(z+s) \leq t\}$ is an ellipsoid with volume $\lambda(E_{\delta}) = v_{\ell}\det (\bfD_{\delta})^{-1/2}t^{\ell} = v_{\ell}t^{\ell} = \lambda(B)$. Hence,
\begin{align}
\mathbb{P}\bigg(\sum_{i=1}^{\ell} T_i^2 \leq t\bigg) - \mathbb{P}\bigg(\sum_{i=1}^{\ell} \delta_i (T_i+s_i)^2 \leq t\bigg) &= \int_{\mathbb{R}^{\ell}}\mathbbm{1}_{z^\top z\leq t}g(\|z\|^2)dz \nonumber \\
&\hspace{2cm}- \int_{\mathbb{R}^{\ell}}\mathbbm{1}_{(z+s)^\top \bfD_{\delta}(z+s)\leq t}g(\|z\|^2)dz\nonumber\\
&= \int_B g(\|z\|^2)dz - \int_{E_{\delta}}g(\|z\|^2)dz\nonumber\\
&= \int_{B\cap E_{\delta}^\rmc} g(\|z\|^2)dz - \int_{E_{\delta}\cap B^\rmc}g(\|z\|^2)dz.\label{eq:compare_cdfs}
\end{align}
But since $g$ is decreasing, $B$ is a superlevel set of $g$: $g(\|z\|^2) \geq g(t^2)$ for $z \in B$ and $g(\|z\|^2) \leq g(t^2)$ for $z\in B^\rmc$. Thus, from equations \eqref{eq:lebesgue_complementaire} and \eqref{eq:compare_cdfs},
\begin{align}
    \mathbb{P}\bigg(\sum_{i=1}^{\ell} T_i^2 \leq t\bigg) - \mathbb{P}\bigg(\sum_{i=1}^{\ell} \delta_i (T_i+s_i)^2 \leq t\bigg) \geq g(t)\big(\lambda(B\cap E_{\delta}^\rmc) - \lambda(E_{\delta}\cap B^\rmc)\big) = 0.
\end{align}
Hence, for all $t>0,\ \mathbb{P}(\sum_{i=1}^{\ell} \delta_i(T_i+s_i)^2 \leq t) \leq  \mathbb{P}(\sum_{i=1}^{\ell} T_i^2 \leq t)$. Since $f \leq g$ implies $g^{-1} \leq f^{-1}$, we obtain that for all $\alpha\in (0,1)$,
\begin{align}
 {q_{1-\alpha}\bigg(\sum_{i=1}^{\ell} T_i^2\bigg)} \leq {q_{1-\alpha}\bigg(\sum_{i=1}^{\ell}\delta_i (T_i+s_i)^2\bigg)}.
\end{align}
This finishes the proof.
\end{proof}
\begin{proof}[Proof of Lemma \ref{lemma:distrib_norm_R1}] Given a random vector $X$, we denote $\Phi_X(t) \coloneqq \mathbb{E}[\exp(i t^\top X)]$. Given that the distribution of $T$ is spherical, introduce $\psi$ the function such that $\Phi_T(t) = \psi(\|t\|^2), \ t\in\mathbb{R}^p$. Since $V_1 = \mu + \mathbf{\Sigma}^{1/2}T$, we have $\Phi_{V_1}(t) = \exp(i t^\top\mu)\psi(t^\top\mathbf{\Sigma} t)$ \cite{muirhead1982}, p 34. We also denote $T' \coloneqq (T_1, \ldots, T_\ell)^\top\in\mathbb{R}^\ell$. Decomposing $\mu = (\mu_X^\top \ \mu_R^\top)^\top$ where $\mu_R\in\mathbb{R}^\ell$, we have for $t'\in\mathbb{R}^\ell$ \cite{muirhead1982}, p 34-35,
\begin{align}
    \Phi_{R_1}(t') &= \Phi_{V_1}((0, t')) = e^{i(t')^\top\mu_R}\psi((t')^\top\mathbf{\Sigma}^{22}t'), \\
    \Phi_{T'}(t') &= \Phi_{T}((0, t')) = \psi(\|t'\|^2).
\end{align}
Hence, $R_1 = \mu_R + (\mathbf{\Sigma}^{22})^{1/2}\widetilde{T}$ for some $\widetilde{T}$ with a spherical distribution. It is also clear from the equations above that $\Phi_{\widetilde{T}}(t') = \psi(\|t'\|^2) = \Phi_{T'}(t')$, so that $\widetilde{T}$ and $T'$ are equal in distribution. Write $\mathbf{\Sigma}^{22} = \bfP\bfD\bfP^\top$ an eigendecomposition of $\mathbf{\Sigma}^{22}$, where $\bfP$ is orthogonal and $\bfD$ is diagonal, $\bfD_{ii} = \lambda_i$. Then, setting $s\coloneqq \bfD^{-1/2}\bfP^\top\mu$ and using that $\bfP^\top T'\stackrel{\rmd}{=}T'$ from the sphericity of $T'$,
\begin{align}
    \|R_1\|^2 &= \|\mu_R + \bfP\bfD^{1/2}\bfP^\top \widetilde{T}\|^2 \stackrel{\rmd}{=} \|\mu_R + \bfP\bfD^{1/2}\bfP^\top T'\|^2 \stackrel{\rmd}{=} \|\mu_R + \bfP\bfD^{1/2} T'\|^2\nonumber
    \\ &=  \|\bfD^{1/2}(s +  T')\|^2
    = \sum_{i=1}^\ell \lambda_i(T_i+s_i)^2.
\end{align}
This finishes the proof.
\end{proof}
\begin{proof}[Proof of Proposition \ref{prop:gen_k}] 
Using Lemma \ref{lemma:distrib_norm_R1}, $q_{1-\alpha}(\|R_1\|^2) = q_{1-\alpha}(\sum_{i=1}^\ell\lambda_i(T_i+s_i)^2)$ where the $\lambda_i$ are the eigenvalues of $\mathbf{\Sigma}^{22}$ and from Proposition \ref{prop:asymptotic_ball}, 
\begin{align}
    \bigg(\frac{\Vol(\calE_{\alpha}^{\infty})}{\Vol(\calB_{\alpha}^{\infty})}\bigg)^{2/\ell} &\leq \det(\mathbf{\Sigma}/\mathbf{\Sigma}^{11})^{1/\ell} \frac{q_{1-\alpha}(\sum_{i=1}^{k+\ell} T_i^2)}{q_{1-\alpha}(\sum_{i=1}^{\ell}\lambda_{i} (T_i+s_i)^2)}\label{eq:first_tradeoff}
     \\
    &\leq c_{\alpha}(k,\ell)\bigg(\frac{\det(\mathbf{\Sigma}/\mathbf{\Sigma}^{11})}{\det(\mathbf{\Sigma}^{22})}\bigg)^{1/\ell}\frac{q_{1-\alpha}(\sum_{i=1}^{\ell}T_{i}^2)}{\det(\mathbf{\Sigma}^{22})^{-1/\ell}q_{1-\alpha}(\sum_{i=1}^{\ell}\lambda_{i}(T_{i}+s_i)^2)}\nonumber \\
&\leq c_{\alpha}(k,\ell)\bigg(\frac{\det(\mathbf{\Sigma}/\mathbf{\Sigma}^{11})}{\det(\mathbf{\Sigma}^{22})}\bigg)^{1/\ell}\frac{q_{1-\alpha}(\sum_{i=1}^{\ell} T_{i}^2)}{q_{1-\alpha}(\sum_{i=1}^{\ell}\delta_{i} (T_{i}+s_i)^2)}\leq c_{\alpha}(k,\ell)\bigg(\frac{\det(\mathbf{\Sigma}/\mathbf{\Sigma}^{11})}{\det(\mathbf{\Sigma}^{22})}\bigg)^{1/\ell},\label{eq:cond_vol_gen}
\end{align}
where $\delta_i = \lambda_i/(\lambda_1 \ldots \lambda_\ell)^{1/\ell}$ and where we applied Proposition \ref{prop:barthe} in the last inequality. We have thus proved Proposition \ref{prop:gen_k} when $\lambda = 0$. We now assume that $\lambda>0$. For this, we follow the steps that lead to equation \eqref{eq:cond_vol_gen}, so that
\begin{align}
   \Vol(\calE_{\alpha,\lambda}^{\infty})^{2/\ell} &= v_{\ell}^{2/\ell} \det(\mathbf{\Sigma}_{\lambda}/\mathbf{\Sigma}_{\lambda}^{11})^{1/\ell}(q_{1-\alpha}(V_\rmc^\top \mathbf{\Sigma}_{\lambda}^{-1}V_\rmc) - X_\rmc^\top (\mathbf{\Sigma}_{\lambda}^{11})^{-1}X_\rmc)_+ \nonumber\\
   &\leq v_{\ell}^{2/\ell}\det(\mathbf{\Sigma}_{\lambda}/\mathbf{\Sigma}_{\lambda}^{11})^{1/\ell}q_{1-\alpha}(V_\rmc^\top \mathbf{\Sigma}_{\lambda}^{-1}V_\rmc)\nonumber\\
   &\leq v_{\ell}^{2/\ell}\det(\mathbf{\Sigma}_{\lambda}/\mathbf{\Sigma}_{\lambda}^{11})^{1/\ell}q_{1-\alpha}(V_\rmc^\top \mathbf{\Sigma}^{-1}V_\rmc)\nonumber\\
   &\leq v_{\ell}^{2/\ell}\det(\mathbf{\Sigma}_{\lambda}/\mathbf{\Sigma}_{\lambda}^{11})^{1/\ell} q_{1-\alpha}\bigg(\sum_{i=1}^{k+\ell}T_i^2\bigg) \leq \det(\mathbf{\Sigma}_{\lambda}/\mathbf{\Sigma}_{\lambda}^{11})^{1/\ell}c_{\alpha}(k,\ell)\Vol(\calB_{\alpha}^{\infty})^{2/\ell}.\label{eq:vol_noise}
\end{align}
In view of equation \eqref{eq:vol_noise}, we show that the map $\lambda \mapsto \det(\mathbf{\Sigma}_{\lambda}/\mathbf{\Sigma}_{\lambda}^{11})$ is strictly increasing. For this, denote $0 \leq \lambda_1 \leq \dots \leq \lambda_p $ the eigenvalues of $\mathbf{\Sigma}$, and $0 \leq \lambda_1^1 \leq \dots \leq \lambda_k^1 $ the eigenvalues of $\mathbf{\Sigma}^{11}$. Then, setting $f_i(\lambda) \coloneqq (\lambda_i + \lambda)/(\lambda_i^1 + \lambda)$ for $i = 1, \ldots, k$,
\begin{align}
    \det(\mathbf{\Sigma}_{\lambda}/\mathbf{\Sigma}_{\lambda}^{11}) = \frac{\det(\mathbf{\Sigma}_{\lambda})}{\det(\mathbf{\Sigma}_{\lambda}^{11})} =\prod_{i=1}^k \bigg(\frac{\lambda_i + \lambda}{\lambda_i^1 + \lambda}\bigg)\prod_{i=k+1}^{k+\ell}(\lambda+\lambda_i) = \prod_{i=1}^k f_i(\lambda)\prod_{i=k+1}^{k+\ell}(\lambda+\lambda_i).\label{eq:schur_lambda}
\end{align}
But from the interlacing theorem (\cite{horn}, Theorem 4.3.15), for all $i\in\{1, \dots , k\}$, $\lambda_i \leq \lambda_i^1$. Hence, $f_i'(\lambda) = (\lambda_i^1-\lambda_i)/(\lambda_i^1+\lambda)^2 \geq 0$, and from equation \eqref{eq:schur_lambda}, the map $\lambda \mapsto \det(\mathbf{\Sigma}_{\lambda}/\mathbf{\Sigma}_{\lambda}^{11})$ is strictly increasing (and goes to $+\infty$ when $\lambda\rightarrow+\infty$). As a consequence, if the inequality in equation \eqref{eq:tradeoff} is strict, then there exists an unique $\lambda_0>0$ verifying
\begin{align}
    \det(\mathbf{\Sigma}_{\lambda_0})/\det(\mathbf{\Sigma}_{\lambda_0}^{11}) = \det(\mathbf{\Sigma}_{\lambda_0}/\mathbf{\Sigma}_{\lambda_0}^{11}) = c_{\alpha}(k,\ell)^{-\ell}\det(\mathbf{\Sigma}^{22}),
\end{align}
which corresponds to the announced criterion.
\end{proof}
\begin{proof}[Proof of Proposition \ref{prop:vol_fna}]
\iain{We use notations exclusive to this proof.
Let $T = (X^\top \ Y^\top)^\top\in \mathbb{R}^p$ have a spherical distribution, with density
\begin{align*}
    f_T(x,y) = g(\|x\|^2 + \|y\|^2), \ \ \ x \in \mathbb{R}^k, \ y\in\mathbb{R}^\ell.
\end{align*}
Then $Y$ has a spherical distribution with density
\begin{align}\label{eq:generator_y}
    f_Y(y) = \int_{\mathbb{R}^k}g(\|x\|^2 + \|y\|^2)dx \eqqcolon g_Y(\|y\|^2).
\end{align}
Define now $U\in\mathbb{R}^k, \ V\in\mathbb{R}^\ell$ and $Z\in\mathbb{R}^\ell$ by
\begin{align*}
   \begin{pmatrix}
            U\\ V
   \end{pmatrix}  &= \begin{pmatrix}
            \mu_u \\ \mu_v
   \end{pmatrix} + \mathbf{\Sigma}^{1/2}T, \ \ \ Z \coloneqq (\mathbf{\Sigma}/\mathbf{\Sigma}^{11})^{-1/2}(V-\mu_v - \mathbf{\Sigma}^{21}(\mathbf{\Sigma}^{11})^{-1}(U-\mu_u)).
\end{align*}
We now show that $Z\stackrel{\rmd}{=}Y$. For this, observe first that the density of $(U^\top \ V^\top)^\top$ is
\begin{align*}
f_{(U,V)}(u,v) = \det(\mathbf{\Sigma})^{-1/2}g\bigg(\begin{pmatrix}
         u-\mu_u\\
         v-\mu_v
\end{pmatrix}^\top\mathbf{\Sigma}^{-1}\begin{pmatrix}
         u-\mu_u\\
         v-\mu_v
\end{pmatrix}\bigg).
\end{align*}
Setting $w = v-\mu_v - \mathbf{\Sigma}^{21}(\mathbf{\Sigma}^{11})^{-1}(u-\mu_u)$, we have from equation \eqref{eq:apply_block_lemma} that
\begin{align}\label{eq:block_inversion_w}
    \begin{pmatrix}
         u-\mu_u\\
         v-\mu_v
\end{pmatrix}^\top\mathbf{\Sigma}^{-1}\begin{pmatrix}
         u-\mu_u\\
         v-\mu_v
\end{pmatrix} = (u-\mu_u)^\top(\mathbf{\Sigma}^{11})^{-1}(u-\mu_u) + w^\top(\mathbf{\Sigma}/\mathbf{\Sigma}^{11})^{-1}w.
\end{align}
Taking the definition of $w$ above as a function of $u$ and $v$, we have, for any bounded measurable function $h$ (explanation below),
\begin{align}
    \mathbb{E}[h(Z)] &=  \int_{\mathbb{R}^k}\int_{\mathbb{R}^\ell}h((\mathbf{\Sigma}/\mathbf{\Sigma}^{11})^{-1/2}w)f_{(U,V)}(u,v)dvdu \label{eq:transfer_lemma}\\
    &= \det(\mathbf{\Sigma})^{-1/2} \int_{\mathbb{R}^k}\int_{\mathbb{R}^\ell}h((\mathbf{\Sigma}/\mathbf{\Sigma}^{11})^{-1/2}w)g((u-\mu_u)^\top(\mathbf{\Sigma}^{11})^{-1}(u-\mu_u) + w^\top(\mathbf{\Sigma}/\mathbf{\Sigma}^{11})^{-1}w)dvdu\label{eq:use_block_inversion_w}\\
    &= \det(\mathbf{\Sigma})^{-1/2}\det(\mathbf{\Sigma}/\mathbf{\Sigma}^{11})^{1/2} \int_{\mathbb{R}^k}\int_{\mathbb{R}^\ell}h(z)g((u-\mu_u)^\top(\mathbf{\Sigma}^{11})^{-1}(u-\mu_u) + \|z\|^2)dzdu \label{eq:change_var_v_z} \\
    &= \int_{\mathbb{R}^\ell}h(z)\bigg(\int_{\mathbb{R}^k}g(\|x\|^2 + \|z\|^2)dx\bigg)dz.\label{eq:change_var_u_x}
\end{align}
We used the transfer lemma in equation \eqref{eq:transfer_lemma}, equation \eqref{eq:block_inversion_w} in equation \eqref{eq:use_block_inversion_w}, the change of variable $z = (\mathbf{\Sigma}/\mathbf{\Sigma}^{11})^{-1/2}w$ (where $u$ is fixed) in the inner integral in equation \eqref{eq:change_var_v_z}, and the change of variable $x = (\mathbf{\Sigma}^{11})^{-1/2}(u-\mu_u)$ and Fubini's theorem in equation \eqref{eq:change_var_u_x}.
Because equation \eqref{eq:change_var_u_x} holds for all bounded measurable function $h$, the density of $Z$ is that of $Y$, hence $Z\stackrel{\rmd}{=}Y$ and
\begin{align}\label{eq:eq_law_z_y}
    \|Z\|^2\stackrel{\rmd}{=}\|Y\|^2.
\end{align}
Observe finally that the equation above still holds if $Y$ is replaced with any subvector of $T$ of dimension $\ell$, since the distribution of $T$ is spherical.
}

\iain{We now go back to the vector $V_1 = (X_1^\top \ R_1^\top)^\top$ of Proposition \ref{prop:vol_fna}, which we assume to be of the form $V_1 = \mu +\mathbf{\Sigma}^{1/2}T$, where $T = (T_1, \ldots, T_{p})^\top$ has a spherical distribution with a non-increasing generator $g$. Using the notation of Proposition \ref{prop:calf_asymptotic} for $T_\rmc$ and $\boldcalA_\infty$, equation \eqref{eq:eq_law_z_y} and Lemma \ref{lemma:distrib_norm_R1} respectively show that
\begin{align*}
    T_\rmc^\top\boldcalA_\infty^{-1}T_\rmc \stackrel{\rmd}{=}\sum_{i=1}^\ell T_i^2, \ \ \ \text{and} \ \ \ \|R_1\|^2 \stackrel{\rmd}{=}\sum_{i=1}^\ell\lambda_i(T_i+s_i)^2,
\end{align*}
for some vector $s\in\mathbb{R}^\ell$ and where the $\lambda_i$ are the eignevalues of $\mathbf{\Sigma}^{22}$.
Then, denoting $\delta_i \coloneqq \det(\mathbf{\Sigma}^{22})^{-1/\ell}\lambda_i$, and using that $\det(\mathbf{\Sigma}^{22}-\mathbf{\Sigma}^{21}(\mathbf{\Sigma}^{11})^{-1}\mathbf{\Sigma}^{12})\leq\det(\mathbf{\Sigma}^{22})$,
\begin{align}
    \bigg(\frac{\Vol(\calF_\alpha^\infty)}{\Vol(\calB_\alpha^\infty)}\bigg)^{2/\ell}& = \det(\mathbf{\Sigma}/\mathbf{\Sigma}^{11})^{1/\ell}\frac{q_{1-\alpha}(T_\rmc^\top\boldcalA_\infty^{-1}T_\rmc)}{q_{1-\alpha}(\|R_1\|^2)} \nonumber \\
    &= \bigg(\frac{\det(\mathbf{\Sigma}^{22}-\mathbf{\Sigma}^{21}(\mathbf{\Sigma}^{11})^{-1}\mathbf{\Sigma}^{12})}{\det(\mathbf{\Sigma}^{22})}\bigg)^{1/\ell}\frac{q_{1-\alpha}(\sum_{i=1}^\ell T_i^2)}{q_{1-\alpha}(\sum_{i=1}^\ell\delta_i(T_i+s_i)^2)}\\
    &\leq \frac{q_{1-\alpha}(\sum_{i=1}^\ell T_i^2)}{q_{1-\alpha}(\sum_{i=1}^\ell\delta_i(T_i+s_i)^2)}.\label{eq:wish_barthe}
\end{align}
We now wish to apply Proposition \ref{prop:barthe} on the spherical vector $Y \coloneqq (T_1, \ldots, T_\ell)$ to conclude that the ratio above is smaller than $1$. For this we need to check that the generator of $Y$ is non-increasing. Since $g$ is assumed non-increasing we can use that, for all $0\leq r\leq r'$ and $x\in\mathbb{R}^k, g(\|x\|^2 + r) - g(\|x\|^2 + r')\geq 0$ so that, using equation \eqref{eq:generator_y},
\begin{align*}
   g_Y(r) - g_Y(r') = \int_{\mathbb{R}^k}g(\|x\|^2 + r) - g(\|x\|^2 + r')dx \geq 0.
\end{align*}
Thus Proposition \ref{prop:barthe} applies to equation \eqref{eq:wish_barthe}, which finishes the proof.
}
\end{proof}
\begin{proof}[Proof of Proposition \ref{prop:contre_exemple}]
For all $i$, define $\delta_i \coloneqq \lambda_i/(\lambda_1 \cdots\lambda_{\ell})^{1/\ell}$ and $(G_1, \ldots, G_\ell)$ to be a random vector with independent components, such that $G_i$ follows the gamma distribution $\gamma(\delta_i,\delta_i^{-1})$. Next, introduce $T \coloneqq (T_1, \ldots, T_{\ell})^\top$ such that
 \begin{align}
    T_i = \varepsilon_i\sqrt{G_i} - (1-\varepsilon_i)\sqrt{G_i},
 \end{align}
 where $\varepsilon_i\sim\mathrm{Bernoulli}(1/2)$, $(\varepsilon_1,\ldots,\varepsilon_\ell)$ are mutually independent and are independent of $(G_1,\ldots,G_\ell)$. Then $(T_1,\ldots, T_\ell)$ are independent, $\mathbb{E}[T_i] = 0$ and $T_i^2 = G_i\sim\gamma(\delta_i,\delta_i^{-1})$, so that $\mathrm{Var}(T_i) = \delta_i\delta_i^{-1} = 1$. Next, define $\bfD$ to be the diagonal matrix such that $\bfD_{ii} = \lambda_i^{1/2}$ and define $R_1\coloneqq \bfD T$. Then $\mathbb{E}[R_1] = 0$, and because $\mathrm{Var}(T_i) = 1$, we have $\mathbf{\Sigma} \coloneqq \Cov(R_1) = \bfD^2$, and
 \begin{align}
    \|\mathbf{\Sigma}^{-1/2}R_1\|^2 &= R_1^\top \mathbf{\Sigma}^{-1} R_1 = \|T\|^2 = \sum_{i=1}^{\ell}T_i^2,\\
     \det(\mathbf{\Sigma})^{-1/\ell}\|R_1\|^2 &= \det(\mathbf{\Sigma})^{-1/\ell}T^\top \bfD^2T = \sum_{i=1}^{\ell}\delta_iT_i^2.
 \end{align}
 We now apply Theorem 1.2 from \cite{mimica}, in order to describe the tail behaviour of both random variables above. For this, observe that $\delta_iT_i^2\sim\gamma(\delta_i,1)$. The Laplace transforms of $\sum_{i=1}^{\ell}T_i^2$ and $\sum_{i=1}^{\ell}\delta_iT_i^2$ are then given by
 \begin{align}
     \calL_1(\omega) &\coloneqq\calL\bigg(\sum_{i=1}^{\ell}T_i^2\bigg)(\omega) = \prod_{i=1}^{\ell}\frac{1}{(1+\delta_i^{-1}\omega)^{\delta_i}},\label{eq:laplace_L1}\\
     \calL_2(\omega) &\coloneqq \calL\bigg(\sum_{i=1}^{\ell}\delta_iT_i^2\bigg)(\omega) = \prod_{i=1}^{\ell}\frac{1}{(1+\omega)^{\delta_i}}.\label{eq:laplace_L2}
 \end{align}
 To apply Theorem 1.2 from \cite{mimica}, we need to identify the abscissa of convergence of $\calL_1$ and $\calL_2$. The abscissa of convergence of a (probability) measure $\nu$ over $[0,\infty)$ is the unique scalar $\sigma_0\in\mathbb{R}\cup\{-\infty\}$ such that the integral $f(z) = \int_{[0,+\infty)} e^{-zt}\nu(dt)$ converges for $\mathrm{Re}(z)>\sigma_0$, diverges for $\mathrm{Re}(z)<\sigma_0$ and has a singularity at $\sigma_0$ \cite{mimica}, p 267.
 For $\calL_1$ and $\calL_2$, it is clear from equations \eqref{eq:laplace_L1} and \eqref{eq:laplace_L2} that their abscissa of convergence are given by $\sigma_1 \coloneqq -\min_i \delta_i<0$ and $\sigma_2 \coloneqq -1$ respectively. We now apply \cite{mimica}, Theorem 1.2, on $\calL_1$ and $\calL_2$ (we check the conditions for applying this result at the end of the proof). This yields 
 \begin{align}
     \lim_{t\rightarrow + \infty}\frac{\log \mathbb{P}(\|\mathbf{\Sigma}^{-1/2}R_1\|^2>t)}{\log\mathbb{P}(\det(\mathbf{\Sigma})^{-1/\ell}\|R_1\|^2>t)} = \frac{\sigma_1}{\sigma_2} = \frac{-\min_i\delta_i}{-1} = \min_i\delta_i < 1. 
 \end{align}
 Indeed, if $\min_i\delta_i\geq 1$, then we would have $\lambda_1 =  \ldots =\lambda_{\ell}$ which contradicts our assumption.
 Because both logarithms above are negative, this yields
 \begin{align}
     \log \mathbb{P}(\|\mathbf{\Sigma}^{-1/2}R_1\|^2>t) > \log\mathbb{P}(\det(\mathbf{\Sigma})^{-1/\ell}\|R_1\|^2>t) \ \ \ \text{for large $t$}.\label{eq:compare_logs}
 \end{align}
 We then take the exponential of the equation above, we compare the resulting CDFs (this reverses the ordering w.r.t. ``$>$'') and we compare the inverses of those CDFs (this reverses back the ordering). We finally obtain that for $\alpha$ small enough,
 \begin{align}
     q_{1-\alpha}(\|\mathbf{\Sigma}^{-1/2}R_1\|^2) > q_{1-\alpha}(\det(\mathbf{\Sigma})^{-1/\ell}\|R_1\|^2).\label{eq:compare_quantiles_mimica}
 \end{align}
In particular, for $\alpha$ small enough, $\Vol(\calE_{\alpha}^{\infty}) > \Vol(\calB_{\alpha}^{\infty})$, which is the announced result.
 We now check the two technical conditions of \cite{mimica}, Theorem 1.2, on $\calL_2$ (the proof is similar for $\calL_1$). First, for $\lambda>0$,
 \begin{align}
     \lambda\log\calL_2(\sigma_2 + \lambda) &= -\bigg(\sum_{i=1}^{\ell}\delta_i\bigg)\lambda\log \lambda  \xrightarrow[\lambda \rightarrow 0^+]{} 0.
 \end{align}
 Second, for all $0<\lambda_1\leq \lambda_2$ sufficiently small,
 \begin{align}
     \frac{\calL_2(\sigma_2 + \lambda_2)}{\calL_2(\sigma_2 + \lambda_1)} &= \prod_{i=1}^{\ell}\bigg(\frac{\lambda_1}{\lambda_2}\bigg)^{-\delta_i} = \bigg(\frac{\lambda_2}{\lambda_1}\bigg)^{-\gamma},
 \end{align}
 with $\gamma \coloneqq \sum_{i=1}^{\ell}\delta_i >0$. From \cite{mimica}, equation (1.3) and Lemma 3.1, this finishes the proof.
\end{proof}
\begin{proof}[Proof of Lemma \ref{lemma:cv_quantile}]
In this proof, we denote $m \coloneqq \mathbb{E}[V_1]$, $\widehat{m}_n \coloneqq n^{-1}\sum_{i=1}^nV_i$ for notational conciseness. We also denote
\begin{align}
    T_i &= \mathbf{\Sigma}_{\lambda}^{-1/2}(V_i - m),\ \ \ \widehat{T}_i = \widehat{\mathbf{\Sigma}}_{n,\lambda}^{-1/2}(V_i - \widehat{m}_n),\ \ \ \nu_n = \frac{1}{n}\sum_{j=1}^n\delta_{T_j},\ \ \ \widehat{\nu}_n = \frac{1}{n}\sum_{j=1}^n\delta_{\widehat{T}_j}.
\end{align}
Note that we are interested in the quantiles of $(\|\widehat{T}_1\|^2,\ldots,\|\widehat{T}_n\|^2)$, since $q_{n,\alpha} = np_{(n_{\alpha})}$ and $np_{i,n} = \|\widehat{T}_i\|^2$. 
Below we borrow the empirical process notation, $\nu_n(f) = n^{-1}\sum_{j=1}^nf(T_j)$.
The proof is carried out by studying the empirical characteristic function. Let $t\in\mathbb{R}^p$. From the SLLN,
\begin{align}\label{eq:carac_LLN_t}
    \nu_n\Big(e^{i\langle t, \cdot \rangle}\Big) = \frac{1}{n}\sum_{j=1}^n e^{i\langle t, T_j \rangle} \xrightarrow[]{a.s.} \mathbb{E}\big[e^{i\langle t, T_1\rangle}\big].
\end{align}
Thus, for all $t\in\mathbb{R}$, $\nu_n(e^{i\langle t, \cdot \rangle}) \rightarrow \mathbb{E}\big[e^{i\langle t, T_1\rangle}\big]$ a.s..  We then prove below that thanks to the continuity of the characteristic functions and the separability of $\mathbb{R}^p$, we can interchange the ``a.s.'' and ``$\forall t \in \mathbb{R}^p$'' to obtain that almost surely,
\begin{align}\label{eq:cv_charac_easy}
    \forall t \in \mathbb{R}^p, \ \ \ \nu_n\Big(e^{i\langle t, \cdot \rangle}\Big) \xrightarrow[n\rightarrow \infty]{} \mathbb{E}\big[e^{i\langle t, T_1\rangle}\big].
\end{align}
Let us prove this fact. Set $\calD = \mathbb{Q}^p$, which is a countable dense subset of $\mathbb{R}^p$. From equation \eqref{eq:carac_LLN_t}, for all $t\in \calD$, there exists $\Omega_t$ with $\Pb(\Omega_t) = 1$ such that for all $\omega \in \Omega_t, \ \ \nu_n(e^{i\langle t, \cdot \rangle}) \xrightarrow[]{a.s.} \mathbb{E}\big[e^{i\langle t, T_1\rangle}\big]$. We set $\Omega' \coloneqq \cap_{t\in \calD}\Omega_t$, which has probability one and aim at proving that for all $\omega\in\Omega'$, equation \eqref{eq:cv_charac_easy} holds. Let $\omega\in\Omega'$, $t\in\mathbb{R}^p$, $\varepsilon > 0$ and $t'\in \calD$ to be chosen later. We can write
\begin{align*}
    \Big|\nu_n\Big(e^{i\langle t, \cdot \rangle}\Big) - \mathbb{E}\big[e^{i\langle t, T_1\rangle}\big]\Big| &\leq \Big|\nu_n\Big(e^{i\langle t, \cdot \rangle}\Big) - \nu_n\Big(e^{i\langle t', \cdot \rangle}\Big)\Big| + \Big|\nu_n\Big(e^{i\langle t', \cdot \rangle}\Big) - \mathbb{E}\big[e^{i\langle t', T_1\rangle}\big]\Big| + \big|\mathbb{E}\big[e^{i\langle t', T_1\rangle}\big] - \mathbb{E}\big[e^{i\langle t, T_1\rangle}\big]\big|.
\end{align*}
In the last term above, $t'$ can be chosen such that this term is smaller than $\varepsilon/3$ from the continuity of $t \mapsto \mathbb{E}\big[e^{i\langle t, T_1\rangle}\big]$ and the density of $\calD$, and the second term is smaller than $\varepsilon/3$ for all $n$ large enough, since equation \eqref{eq:carac_LLN_t} applies for $t'\in \calD$. Finally, for the first term,
\begin{align}
    \Big|\nu_n\Big(e^{i\langle t, \cdot \rangle}\Big) - \nu_n\Big(e^{i\langle t', \cdot \rangle}\Big)\Big| &\leq \frac{1}{n}\sum_{j=1}^n \Big|e^{i\langle t, T_j \rangle} - e^{i\langle t', {T}_j \rangle}\Big| = \frac{2}{n}\sum_{j=1}^n\bigg| \sin\bigg(\frac{\langle t-t', {T}_j\rangle}{2}\bigg)\bigg|\leq \bigg(\frac{1}{n}\sum_{j=1}^n \|T_j\|\bigg)\|t - t'\|. \label{eq:cv_nu_ttp}
\end{align}
From the SLLN, the sum above is bounded by some $M_{\omega} > 0$ for $n$ large enough, and thus $t'$ can also be chosen (as a function of $\omega$) such that this term is smaller than $\varepsilon/3$. We have thus proved that almost surely, equation \eqref{eq:cv_charac_easy} holds. 

We now prove that equation \eqref{eq:cv_charac_easy} also holds for $\widehat{\nu}_n$, and work similarly as above. For this, observe that for all $t\in\mathbb{R}^p$,
\begin{align}
    \Big|\widehat{\nu}_n\Big(e^{i\langle t, \cdot \rangle}\Big)  - \nu_n\Big(e^{i\langle t, \cdot \rangle}\Big) \Big| &\leq \frac{1}{n}\sum_{j=1}^n \Big|e^{i\langle t, \widehat{T}_j \rangle} - e^{i\langle t, {T}_j \rangle}\Big| = \frac{2}{n}\sum_{j=1}^n\bigg| \sin\bigg(\frac{\langle t, \widehat{T}_j - {T}_j\rangle}{2}\bigg)\bigg|\leq \frac{\|t\|}{n}\sum_{j=1}^n \| \widehat{T}_j - {T}_j\|.\label{eq:control_charac_sinus} \end{align}
Furthermore, using the matrix operator norm such that $\|Mx\|\leq\|M\|_\Op\|x\|$,
\begin{align*}
    \| \widehat{T}_j - {T}_j\| &= \Big\| \Big(\mathbf{\Sigma}_{\lambda}^{-1/2} - \widehat{\mathbf{\Sigma}}_{n,\lambda}^{-1/2}\Big)V_j - \mathbf{\Sigma}_{\lambda}^{-1/2}m + \widehat{\mathbf{\Sigma}}_{n,\lambda}^{-1/2}\widehat{m}_n\Big\| \\
    &\leq \Big\|\mathbf{\Sigma}_{\lambda}^{-1/2} - \widehat{\mathbf{\Sigma}}_{n,\lambda}^{-1/2}\Big\|_{\Op} \|V_j\| +  \big\|\mathbf{\Sigma}_{\lambda}^{-1/2}m - \widehat{\mathbf{\Sigma}}_{n,\lambda}^{-1/2}\widehat{m}_n\Big\|.
\end{align*}
From equation \eqref{eq:as_cov}, the continuous mapping theorem and the SLLN, for all $t\neq 0$ (the case $t=0$ is trivial),
\begin{align}
     \frac{1}{\|t\|}\Big|\widehat{\nu}_n\Big(e^{i\langle t, \cdot \rangle}\Big)  - \nu_n\Big(e^{i\langle t, \cdot \rangle}\Big) \Big| &\leq \Big\|\mathbf{\Sigma}_{\lambda}^{-1/2} - \widehat{\mathbf{\Sigma}}_{n,\lambda}^{-1/2}\Big\|_{\Op}\bigg(\frac{1}{n}\sum_{j=1}^n\|V_j\|\bigg)+\big\|\mathbf{\Sigma}_{\lambda}^{-1/2}m - \widehat{\mathbf{\Sigma}}_{n,\lambda}^{-1/2}\widehat{m}_n\Big\| \ \ \xrightarrow[n\rightarrow \infty]{a.s.} 0. \label{eq:cv_hat_t}
\end{align}
Note that for $\widehat{\mathbf{\Sigma}}_{n,\lambda}^{-1/2} \rightarrow \mathbf{\Sigma}_{\lambda}^{-1/2}$ to hold a.s., we either require that $\lambda>0$ or $\min\Spec \mathbf{\Sigma} > 0$, so that we can apply the same arguments as those at the end of the proof of Proposition \ref{prop:asymptotics}. Also, the a.s. convergence of equation \eqref{eq:cv_hat_t} does not depend on $t$, as the bound on the right-hand side of equation \eqref{eq:cv_hat_t} does not depend on $t$. Hence, almost surely,
\begin{align}\label{eq:cv_charac_hard}
    \forall t \in \mathbb{R}^p, \ \ \ \Big|\widehat{\nu}_n\Big(e^{i\langle t, \cdot \rangle}\Big)  - \nu_n\Big(e^{i\langle t, \cdot \rangle}\Big) \Big| \xrightarrow[n\rightarrow \infty]{} 0.
\end{align}
Combining equation \eqref{eq:cv_charac_hard} with equation \eqref{eq:cv_charac_easy}, we finally obtain that almost surely,
\begin{align}\label{eq:cv_charac_hard_combine}
    \forall t \in \mathbb{R}^p, \ \ \ \widehat{\nu}_n\Big(e^{i\langle t, \cdot \rangle}\Big) \xrightarrow[n\rightarrow \infty]{} \mathbb{E}\big[e^{i\langle t, T_1\rangle}\big].
\end{align}
Applying Lévy's theorem sample by sample for $\omega\in\Omega'$, we deduce that almost surely, in the Prokhorov metric for the weak convergence of measures over $\mathbb{R}^p$,
\begin{align}
    \widehat{\nu}_n \xlongrightarrow[n\rightarrow \infty]{} \nu_{T_1}.
\end{align}
Above, $\nu_{T_1}$ is the probability measure of $T_1$, defined over $\mathbb{R}^p$.
Applying pushforward integration with the continuous map $N(x) = \|x\|^2$, almost surely (in the Prokhorov metric for probability measures over $\mathbb{R}$),
\begin{align}
    N_{\#}\widehat{\nu}_n = \frac{1}{n}\sum_{j=1}\delta_{\|\widehat{T}_i\|^2} \xlongrightarrow[n\rightarrow \infty]{} N_{\#}\nu_{T_1} \eqqcolon \nu_{\|T_1\|^2}.
\end{align}
Hence, almost surely, for all $t\in\mathbb{R}$ which is a continuity point of $F_{\|T_1\|^2}$, where $F_{\|T_1\|^2}$ denotes the CDF of $\|T_1\|^2$,
\begin{align}\label{eq:cv_empirical_cdf}
    \widehat{F}_n(t) = \frac{1}{n}\sum_{j=1}^n \mathbbm{1}({\|\widehat{T}_i\|^2 \leq t}) \xlongrightarrow[n\rightarrow \infty]{} F_{\|T_1\|^2}(t).
\end{align}
We now denote $Q_{\|T_1\|^2}$ the quantile function of $\|T_1\|^2$ and $\widehat{Q}_n$ the empirical quantile function of $(\|\widehat{T}_1\|^2,  \ldots, \|\widehat{T}_n\|^2)$.
From equation \eqref{eq:cv_empirical_cdf} and \cite{van2000asymptotic}, Lemma 21.2, we have that, almost surely, for all $\beta \in (0,1)$ which is a continuity point of $Q_{\|T_1\|^2}$,
\begin{align}\label{eq:cv_quant_still}
   \widehat{Q}_n(\beta) = \widehat{F}_n^{-1}(\beta) \xlongrightarrow[n\rightarrow \infty]{} Q_{\|T_1\|^2}(\beta).
\end{align}
In particular, this holds for $\beta = 1 -\alpha$, since we have assumed that $1-\alpha$ was a continuity point of $Q_{\|T_1\|^2}$. To finish, we need to show that equation \eqref{eq:cv_quant_still} also holds for the order statistics $q_{n,\alpha} = \widehat{Q}_n(\frac{n+1}{n}(1-\alpha))$. For this, from the assumptions, let $\mathcal{V}$ be a neighbourhood of $1-\alpha$ such that $Q_{\|T_1\|^2}$ is continuous on $\mathcal{V}$, and let $M > 0$ be such that $1 - \alpha + 1/M \in \mathcal{V}$. There exists $n$ large enough such that $\frac{n+1}{n}(1-\alpha) \leq 1-\alpha + 1/M$. Hence, from the monotony of the empirical quantile, we have that for all such $n$, 
\begin{align}
    \widehat{Q}_n(1-\alpha) \leq  \widehat{Q}_n\bigg(\frac{n+1}{n}(1-\alpha)\bigg) \leq \widehat{Q}_n(1-\alpha + 1/M).
\end{align}
From equation \eqref{eq:cv_quant_still} and the continuity of $Q_{\|T_1\|^2}$ at both $1-\alpha$ and $1-\alpha + 1/M$, almost surely,
\begin{align*}
    Q_{\|T_1\|^2}(1-\alpha) &\leq \liminf_n \widehat{Q}_n\bigg(\frac{n+1}{n}(1-\alpha)\bigg) \leq \limsup_n \widehat{Q}_n\bigg(\frac{n+1}{n}(1-\alpha)\bigg) \leq Q_{\|T_1\|^2}(1-\alpha + 1/M).
\end{align*}
Since this equation holds for all $M>0$ large enough, we have that $\liminf_n \widehat{Q}_n(\frac{n+1}{n}(1-\alpha))=\limsup_n \widehat{Q}_n(\frac{n+1}{n}(1-\alpha))$, hence the sequence $(\widehat{Q}_n(\frac{n+1}{n}(1-\alpha)))_n$ is convergent and almost surely,
\begin{align}
    q_{n,\alpha} = \widehat{Q}_n\bigg(\frac{n+1}{n}(1-\alpha)\bigg) \xlongrightarrow[n\rightarrow \infty]{} Q_{\|T_1\|^2}(1-\alpha) = q_{1-\alpha}^\calE.
\end{align}
This finishes the proof of equation \eqref{eq:cvqna_e}. \iain{We now prove equation \eqref{eq:cvqna_second}. For this, introduce $m^X = \mathbb{E}[X_1]$, $\widehat{m}_n^X \coloneqq n^{-1}\sum_{j=1}^nX_i$, and
\begin{align}
    U_i &\coloneqq (\mathbf{\Sigma}_\lambda^{11})^{-1/2}(X_i - m^X), \ \ \ \widehat{U}_i \coloneqq (\widehat{\mathbf{\Sigma}}_{n,\lambda}^{11})^{-1/2}(X_i - \widehat{m}_n^X), \\
    K_i &\coloneqq\begin{pmatrix}
             T_i \\
             U_i
    \end{pmatrix}, \ \ \ \widehat{K}_i  \coloneqq\begin{pmatrix}
             \widehat{T}_i \\
             \widehat{U}_i
    \end{pmatrix}, \ \ \ \mu_n = \frac{1}{n}\sum_{j=1}^n\delta_{K_j},\ \ \ \widehat{\mu}_n = \frac{1}{n}\sum_{j=1}^n\delta_{\widehat{K}_j}.
\end{align}
We quickly show why we still have that almost surely,
    \begin{align}\label{eq:cv_carac_seconddd}
    \forall t\in\mathbb{R}^{p+k}, \ \ \  \widehat{\mu}_n\Big(e^{i\langle t,\cdot\rangle}\Big) \xlongrightarrow[n\rightarrow \infty]{} \mathbb{E}\big[e^{i\langle t,K_1\rangle}\big].
    \end{align}
    First, equation \eqref{eq:cv_charac_easy} still holds when replacing $T_i$ with $K_i$ because the $K_i$ remain iid, and the separability of $\mathbb{R}^{p+k}$ and the continuity of $t\mapsto \mathbb{E}\big[e^{i\langle t,K_1\rangle}\big]$ still hold : that is,
        \begin{align}\label{eq:cv_carac_easy_seconddd}
    \forall t\in\mathbb{R}^{p+k}, \ \ \  \mu_n\Big(e^{i\langle t,\cdot\rangle}\Big) \xlongrightarrow[n\rightarrow \infty]{} \mathbb{E}\big[e^{i\langle t,K_1\rangle}\big].
    \end{align}
    We now show that equation \eqref{eq:cv_charac_hard} still holds when replacing $T_i$ with $K_i$ and $\widehat{T}_i$ with $\widehat{K}_i$. We start from the bound
    \begin{align}
        \|\widehat{K}_j - K_j\| = \sqrt{\|\widehat{T}_i-T_i\|^2 + \|\widehat{U}_i-U_i\|^2} \leq \|\widehat{T}_i-T_i\| + \|\widehat{U}_i-U_i\|.
    \end{align}
    This bound implies that (see equation \eqref{eq:control_charac_sinus}) for all $t\in\mathbb{R}^{p+k}, t \neq 0$,
    \begin{align}
        \frac{1}{\|t\|}\Big|\widehat{\mu}_n\Big(e^{i\langle t, \cdot \rangle}\Big)  - \mu_n\Big(e^{i\langle t, \cdot \rangle}\Big) \Big| \leq \frac{1}{n}\sum_{j=1}^n \| \widehat{K}_j - K_j\| \leq \frac{1}{n}\sum_{j=1}^n \| \widehat{T}_j - T_j\| + \frac{1}{n}\sum_{j=1}^n \| \widehat{U}_j - U_j\|.
    \end{align}
    We can then use the same arguments on the two sums above as the ones that lead to equation \eqref{eq:cv_hat_t}, to obtain that a.s.,
     \begin{align}\label{eq:cv_hard_second}
        \forall t \in \mathbb{R}^{p+k}, \ \ \  \Big|\widehat{\mu}_n\Big(e^{i\langle t, \cdot \rangle}\Big)  - \mu_n\Big(e^{i\langle t, \cdot \rangle}\Big) \Big|  \xrightarrow[n\rightarrow \infty]{} 0.
     \end{align}
Combining equations \eqref{eq:cv_carac_easy_seconddd} and \eqref{eq:cv_hard_second}, we obtain equation \eqref{eq:cv_carac_seconddd}. Equation \eqref{eq:cv_carac_seconddd} implies that in the Prokhorov metric, almost surely,
\begin{align*}
    \widehat{\mu}_n \xlongrightarrow[n\rightarrow
    \infty]{} \mu_K
\end{align*}
Using pushforward integration with the continuous map $(x,y)\mapsto \|x\|^2-\|y\|^2$, we obtain again that almost surely, for all $t\in\mathbb{R}$ which is a continuity point of $F_{\|T_1\|^2- \|U_1\|^2}$, where $F_{\|T_1\|^2- \|U_1\|^2}$ denotes the CDF of $\|T_1\|^2- \|U_1\|^2$,
\begin{align}\label{eq:cv_empirical_cdf_second}
    \widehat{F}_n(t) = \frac{1}{n}\sum_{j=1}^n \mathbbm{1}(\|\widehat{T}_i\|^2- \|\widehat{U}_i\|^2\leq t) \xlongrightarrow[n\rightarrow \infty]{} F_{\|T_1\|^2 - \|U_1\|^2}(t).
\end{align} 
In particular, using the same arguments as for $q_{n,\alpha}$, we obtain that a.s.,
\begin{align} \label{eq:cv_qn_second}
    q_{n,\alpha}'' \xlongrightarrow[n\rightarrow\infty]{}Q_{\|T_1\|^2 - \|U_1\|^2}(1-\alpha) = q_{1-\alpha}^\calF.
\end{align}
We finish with proving equation \eqref{eq:cvqna_prime}. For this, observe that for all $i\in\{1,\ldots,n\},$
\begin{align}
p_{i,n}''\leq p_{i,n}' \leq p_{i,n}'' + \max_{1\leq j\leq n} n(b_X^n)_j^2,
\end{align}
hence
\begin{align}\label{eq:controlqn_prime_qn_second}
     q_{n,\alpha}'' \leq q_{n,\alpha}' \leq q_{n,\alpha}'' + \max_{1\leq j\leq n} n(b_X^n)_j^2.
\end{align}
Using the same arguments as in the proof of Proposition \ref{prop:bound}, the assumption that $\mathbb{E}[\|X_1\|^{4q}]<+\infty$ for some $q>1$ implies that
\begin{align}
    \max_{1\leq j\leq n} n(b_X^n)_j^2 \xlongrightarrow[n\rightarrow \infty]{\mathbb{P}} 0.
\end{align}
Combining equations \eqref{eq:cv_qn_second} and \eqref{eq:controlqn_prime_qn_second}, and using that a.s. convergence implies convergence in probability, we have that $q_{n,\alpha}' \xlongrightarrow[n\rightarrow \infty]{\mathbb{P}} q_{1-\alpha}^\calF$, which finishes the proof.
}
\end{proof}


%
%
\end{appendix}

\paragraph{Acknowledgments}
The authors would like to sincerely thank Mathieu Riou, Elie Goudout and Franck Barthe for fruitful discussions, and F. Barthe's proof of Proposition \ref{prop:barthe} in particular.


\paragraph{Funding}
This work was funded by the project ROMEO (ANR-21-ASIA-0001) from the ASTRID joint program of the French National Research Agency (ANR) and the French Innovation and Defence Agency (AID).

\bibliographystyle{abbrv}
\bibliography{biblio}

\end{document}